\documentclass[11pt, reqno]{amsart}
\usepackage{cases}
\usepackage[square, comma, sort&compress, numbers]{natbib}
\usepackage{float}
\usepackage{pifont}
\usepackage{bbding}
\usepackage{amssymb}
\usepackage{mathrsfs}
\usepackage{subfigure}
\usepackage{caption}
\usepackage{amsmath, amsthm, amscd, amsfonts, amssymb, graphicx, color}
\usepackage{lineno,hyperref}
\newtheorem{thm}{Theorem}
\newtheorem{theorem}{Theorem}[section]

\theoremstyle{definition}

\newtheorem{example}[theorem]{Example}

\newtheorem{remark}{Remark}

\theoremstyle{remark} \numberwithin{equation}{section}
\DeclareMathOperator{\RE}{Re}
\DeclareMathOperator{\IM}{Im}

\newcommand{\ID}{{\mathbb D}}
\newcommand{\IN}{{\mathbb N}}
\begin{document}
\thispagestyle{empty} \setcounter{page}{1}


\title[Univalent harmonic mappings and lift to the minimal surfaces]
{Univalent harmonic mappings and lift to the minimal surfaces}



\author[Y. Jiang]{YuePing Jiang}
\address{Y. Jiang, School of Mathematics and Econometrics, Hunan University, Changsha 410082, Hunan, People's Republic of China.}
\email{\textcolor[rgb]{0.00,0.00,0.84}{ypjiang731@163.com}}

\author[Z. Liu]{ZhiHong Liu  $^\dagger $}
\address{Z.Liu, School of Mathematics and Econometrics, Hunan University, Changsha 410082, Hunan, People's Republic of China.
\vskip.03in College of Mathematics, Honghe University, Mengzi 661199, Yunnan, People's Republic of China.}
\email{\textcolor[rgb]{0.00,0.00,0.84}{liuzhihongmath@163.com}}

\author[S. Ponnusamy]{Saminathan Ponnusamy}
\address{S. Ponnusamy, Indian Statistical Institute (ISI), Chennai Centre,
SETS (Society for Electronic Transactions and Security), MGR Knowledge City, CIT Campus,
Taramani, Chennai 600 113, India.}
\email{\textcolor[rgb]{0.00,0.00,0.84}{samy@isichennai.res.in, samy@iitm.ac.in}}

\thanks{$^\dagger $ Corresponding author.}




\begin{abstract}
We construct sense-preserving univalent harmonic mappings which map the unit disk  onto a domain which is convex in the horizontal
direction, but with varying dilatation. Also, we obtain minimal surfaces associated with such harmonic
mappings. This solves also a recent problem of Dorff and Muir (Abstr. Appl. Anal. (2014)).  In several of the cases, we illustrate mappings
together with their minimal surfaces pictorially with the help of \texttt{Mathematica} software.
\end{abstract}
\maketitle

\section{Introduction}\label{LJS1-sec1}

Let $\mathbb{D}=\{z\in\mathbb{C}:\,|z|<1\}$ be the open unit disk in the complex plane.
Shear construction of univalent harmonic mappings in $\ID$ (see Theorem \ref{thmA})
motivated by Clunie and Sheil-Small~\cite{Clunie} is instrumental in identifying harmonic analog of the
classical Koebe function which and its rotation played the role of extremal for many extremal problems in the theory of univalent analytic mappings in $\ID$. The method of shearing has been used effectively in determining several nice properties and examples of univalent harmonic mappings. Another important result for the study of surfaces using geometry and harmonic mappings is the so called Weierstrass-Enneper representation  (cf.~\cite [p.~177-178] {Duren2004}).
The present article is essentially deal with some application of these two results. Similar applications are obtained in \cite{Dorff2014AAA,Greiner2004,LiulanLi201204,Ponnusamy2014AMC,Ponnusamy2014CVEE} and thus, the present note is a continuation of these recent investigations.

Let $\mathcal B$ be the class of analytic self-maps of the unit disk $\mathbb{D}$ and
${\mathcal B}_0=\{\omega \in  {\mathcal B}:\,  \omega (0)=0 \}$.
In the recent years, the class $\mathcal{H}$ of all complex-valued harmonic mappings $f=h+\overline{g}$ on
$\mathbb{D}$,  normalized by $h(0)=g(0)=h'(0)-1=0$, where $h$ and $g$ are analytic,
attracted the attention of function theorists in many different contexts.
By a result of Lewy~\cite{Lewy}, $f=h+\overline{g}\in \mathcal{H}$ is locally univalent and sense-preserving if and only if $J_{f}(z)>0$ in $\mathbb{D}$,
where $J_{f}(z)= |h'(z)|^{2}-|g'(z)|^{2}$ denotes the Jacobian of $f$. Positivity of the Jacobian is equivalent to the
existence of complex dilatation  $\omega \in  {\mathcal B}$ such that $\omega(z)=g'(z)/h'(z)$.
Let $\mathcal{S}_{H}$ be the class of all sense-preserving harmonic univalent mappings $f\in\mathcal{H}$ and $\mathcal{S}_{H}^{0}$, the subclass of mappings $f\in\mathcal{S}_{H}$ such that $f_{\overline{z}}(0)=0$. Set ${\mathcal S}=\{f=h+\overline{g}\in\mathcal{S}_{H}:\,  g(z)\equiv 0\}$.

We recall that a domain $\Omega\subset \mathbb{C}$ is said to be \textit{convex in the horizontal direction} (CHD) if its intersection with each horizontal line is connected (or empty). We follow the convention that $f\in\mathcal{S}_{H}$ is CHD mapping if $f(\ID)$ is CHD. Now it is appropriate to recall
the following theorem of Clunie and Sheil-Small~\cite{Clunie} which is crucial in the construction of minimal surfaces.

\begin{thm}\label{thmA}
Let $f=h+\overline{g}$ be harmonic and locally univalent in  $\mathbb{D}$. Then $f$ is univalent and its range is {\rm CHD} if and only if $h-g$ has the same properties.
\end{thm}

An algorithmic approach of Theorem \ref{thmA} follows.
For a given CHD conformal mapping $\varphi$ of $\mathbb{D}$ and a dilatation $\omega \in  {\mathcal B}_0$, the shear of $\varphi (z)$ for the given $\omega (z)$
is defined to be the mapping $f=h+\overline{g}\in \mathcal{S}_{H}^{0}$  satisfying the pair of differential equations
\begin{eqnarray*}\label{Beleq}
\left\{\begin{aligned}
h'(z)-g'(z)=\varphi'(z),\\
g'(z)-\omega(z)h'(z)=0.
\end{aligned} \right.
\end{eqnarray*}
Then a straightforward calculation gives the desired mapping $f=h+\overline{g}\in \mathcal{S}_{H}^{0}$ as
\begin{equation}\label{eqf}
\begin{split}
f(z)
&=\RE\left\{2\int_{0}^{z}\frac{\varphi'(\zeta)}{1-\omega(\zeta)}d\zeta-\varphi(z)\right\}+i\IM\{\varphi(z)\}.
\end{split}
\end{equation}
This is the basic here.
Construction of a harmonic mapping that can be lifted to the minimal surface by using the following
version of Weierstrass-Enneper representation (cf.~\cite [pp. 177-178] {Duren2004}).

\begin{thm}\label{thmB}{\rm(Weierstrass-Enneper representation).}
Let $\Omega\subseteq\mathbb{C}$ be a simply connected domain containing the origin. If a minimal graph
\begin{equation*}
\{(u,v,F(u,v)):\,u+iv\in\Omega\}
\end{equation*}
is parameterized by sense-preserving isothermal parameters
$z=x+i y\in\mathbb{D}$, the projection onto its base plane defines a
harmonic mapping $w=u+i v=f(z)$ of $\mathbb{D}$ onto $\Omega$ whose
dilatation is the square of an analytic function. Conversely, if
$f=h+\overline{g}$ is a harmonic univalent mapping of $\mathbb{D}$ onto $\Omega$ with
dilatation $\omega=g'/h'=q^2$, the square of an analytic function $q$,
then with $z=x+i y\in\mathbb{D}$, the parametrization
\begin{equation*}
\mathbf{X}(z)=\left(\RE\{h(z)+g(z)\},\IM\{h(z)-g(z)\},2\IM\left\{\int_{0}^{z}h'(\zeta)q(\zeta)d\zeta\right\}\right)
\end{equation*}
defines a minimal graph whose projection into the complex
plane is $f(\mathbb{D})$. Except for the choice of sign and an arbitrary additive constant in the third coordinate function, this is the only such surface.
\end{thm}

Further information about the relationship between certain harmonic mappings and the associated minimal surfaces can
be found from~\cite{Dorffbook2012,Duren2004,Ponnusamy2014AMC,Quach2014,Ponnusamy2014CVEE,SaRa2013}.
In~\cite{LiulanLi201204}, the authors  considered for example the single slit CHD mapping, namely, the Koebe function $k(z)=z/(1-z)^2$,
and derived the following result.

\begin{thm}\label{thmC}
Let $\mathbf{X}$ be a minimal surface over the slit domain $L=k(\ID)$ with the projection $f=h+\overline{g}\in \mathcal{S}_{H}^{0}$, which satisfies
$$h(z)-g(z)=\frac{z}{(1-z)^2}
$$
and whose dilatation $\omega=z^2$. Then $\mathbf{X}=\{(u,v,F(u,v)):\,u+iv\in L\}$, where
$$
u=\RE\left\{\frac{z(z^2-3z+3)}{3(1-z)^3}\right\},~~v=\IM\left\{\frac{z}{(1-z)^2}\right\},
$$
and
$$F=\IM\left\{\frac{z(2-z)}{(1-z)^2}-\frac{2 z (z^2-3 z+3)}{3(1-z)^3}\right\}.
$$
\end{thm}

As in the recent article of  Dorff and Muir~\cite{Dorff2014AAA}, we consider the generalized Koebe function $k_{c}:\mathbb{D}\to \mathbb{C}$ defined by
\begin{equation}\label{GKF}
k_{c}(z)=\int_{0}^{z}\frac{(1+\zeta)^{c-1}}{(1-\zeta)^{c+1}}d\zeta = \frac{1}{2c}\left[\left(\frac{1+z}{1-z}\right)^{c}-1\right]
\end{equation}
for  $c\in[0,2]$, and in the case of $c=0$, the function $k_{c}(z)$ should be interpreted as the limiting case:
$$k_{0}(z) = \lim_{c\rightarrow 0^{+}} k_c(z)=\frac{1}{2}\log\left(\frac{1+z}{1-z}\right).
$$
Obviously, $k_{1}(z)=z/(1-z)$ and $k_{2}(z)=z/(1-z)^2$.
Moreover, for $c\in[0,2]$, $k_{c}\in \mathcal{S}$ and $k_{c}(\mathbb{D})$ is CHD. Additionally, for $c\in[0,1]$, $k_{c}(\mathbb{D})$ is convex.

\begin{thm}\label{thmD}
{\rm (\cite[Theorem 3]{Dorff2014AAA}) }
For $c\in[0,2]$, define $f_{c}=h_{c}+\overline{g_{c}}\in \mathcal{S}_{H}^{0}$ to be the harmonic mapping satisfying
\begin{equation*}
h_{c}(z)-g_{c}(z)=k_{c}(z),\quad g'_{c}(z)=z^2 h'_{c}(z),
\end{equation*}
where $k_{c}$ is given by~\eqref{GKF}. Then $f_{c}(\mathbb{D})$ is {\rm CHD}, and as $c$ varies from $0$ to $2$, $f_{c}(\mathbb{D})$ transforms from a strip mapping to a slit mapping.
\end{thm}

In~\cite{Dorff2014AAA}, the authors also proposed that the family of harmonic mappings given in Theorem~\ref{thmD} can be generalized by changing the dilatation to $\omega(z)=z^{2m} ~ (m\in\mathbb{N})$. That is, for $c\in[0,2]$ and $n=2m$, let $f_{c,n}=h_{c,n}+\overline{g_{c,n}}\in \mathcal{S}_{H}^{0} $, where
\begin{equation*}
h_{c,n}(z)-g_{c,n}(z)=k_{c}(z) ~\mbox{ and }~ g'_{c,n}(z)=z^{n}h'_{c,n}(z).
\end{equation*}
The case $n=2$ is the basis for Theorem~\ref{thmD}. For the case $n=4$ and $c=2$, it appears that the resulting minimal surface is a helicoid.

It would be interesting to use the shearing construction to investigate the family of mappings $f_{c,n}(z)$ for $n\geq 1$.  In most cases the dilatation is chosen to be $\omega(z)=z^{n}~(n\in\mathbb{N})$. The present article is organized as follows. Section \ref{LJS1-sec2} begins with a set of new examples of CHD  mappings and present an application of shearing theorem to obtain a class of CHD mappings. Later in
Section \ref{LJS1-sec3}, we derive the explicit representation of $f_{c,n}(z)$ when $c=1,2$ and for all $n\in \mathbb{N}$ by using the partial fraction expansion method. The case $c=0$ is known from the work of Greiner \cite{Greiner2004}. Also, we show that $f_{c,n}(\mathbb{D})$ is CHD, and as $c$ varies from $0$ to $2$, $f_{c,n}(\mathbb{D})$ transforms continuously from the strip mapping to the wave plane and finally to the slit mapping. Recall  that if the dilatation of $f_{c,n}(z)$ is a square of an analytic function, then we obtain harmonic mappings which can be lifted to the minimal surfaces expressed by isothermal parameters and thus, we also obtain the minimal surfaces associated with such harmonic
mappings based on the Weierstrass-Enneper representation and thereby, we solves the problem proposed by Dorff and Muir~\cite{Dorff2014AAA}. Finally, we illustrate the harmonic mappings for some special cases together with their minimal surfaces pictorially with the help of \texttt{Mathematica} software.
Throughout the discussion, the images of the corresponding harmonic mappings are shown in Figures as plots of the images of equally spaced
radial segments and concentric circles of the unit disk $\ID$.

In order to represent $f_{c,n}(z)$ explicitly, we need to introduce the Appell hypergeometric function of two variables~\cite{Olver2010}.
The Appell hypergeometric function $F_{1}$ of two variables  is defined for $|x| < 1, |y| < 1$ by the double series:
$$F_{1}(\alpha;\beta_{1},\beta_{2};\gamma;x,y)=\sum_{k=0}^{\infty}\sum_{l=0}^{\infty}
\frac{(\alpha)_{k+l}(\beta_{1})_{k}(\beta_{2})_{l}}{k!l!(\gamma)_{k+l}}x^{k}y^{l},
$$
where $(q)_{0}=1$ for $q\neq 0$ and for $q\in \mathbb{C}\backslash\{0\}$,
$$(q)_{k}=q(q+1)\cdots(q+k-1)=\frac{\Gamma(q+k)}{\Gamma(q)} 
$$
is the Pochhammer symbol.  Appell's $F_1$ can also be written as a one-dimensional Euler-type integral:
$$F_{1}(\alpha;\beta_{1},\beta_{2};\gamma;x,y)=\frac{\Gamma(\gamma)}{\Gamma(\alpha)\Gamma(\gamma-\alpha)}
\int_{0}^{1}\frac{t^{\alpha-1}(1-t)^{\gamma-\alpha-1}}{(1-xt)^{\beta_{1}}(1-yt)^{\beta_{2}}}d t,
$$
where $\RE\gamma>\RE\alpha>0$.

\section{Harmonic mappings with the dilatation $\omega(z)=z\frac{z+a}{1+az}$}\label{LJS1-sec2}
Throughout this section, in the following examples and in Theorem \ref{thmca}, our aim is to construct a family of CHD mappings with the dilatation $\omega(z)=z\frac{z+a}{1+az}$, where $-1\leq a\leq 1$.  For $a=1$ and $a=0$, $\omega(z)$ becomes $z$ and $z^2$, respectively.

\begin{example}\label{Exam1}

Consider the identity mapping $\varphi(z)=z$. Then, by \eqref{eqf}, the shear construction produces the harmonic mappings
 \begin{equation*}
F_{a}(z)=\RE\{-z+(1-a)\log(1+z)-(1+a)\log(1-z)\}+i \IM\{z\}.
\end{equation*}
The images of the unit disk $\ID$ under $F_{a}$ for $-1\leq a\leq 1$ are shown in
Figure~\ref{fida}.  The images of $\ID$ under $F_{0}$ and $F_1$
are shown in Figures~\ref{fida}~(a) and (d) (see also Duren~\cite[Figures 3.1 and 3.2, Section 3.4]{Duren2004}), respectively.
Moreover, the images of $F_{a}(z)$ and $F_{-a}(z)$ are symmetric about the imaginary axis, since
\begin{equation*}
\RE\{F_{-a}(-z)\}=\RE\{z+(1+a)\log(1-z)-(1-a)\log(1+z)\}=-\RE\{F_{a}(z)\}
\end{equation*}
and
$$
\IM\{F_{-a}(-z)\}=\IM\{-z\}=-\IM\{F_{a}(z)\}.
$$
\begin{figure}[!h]
\centering
\subfigure[$a=0$]
{\begin{minipage}[b]{0.45\textwidth}
\includegraphics[height=1.6in,width=2.4in,keepaspectratio]{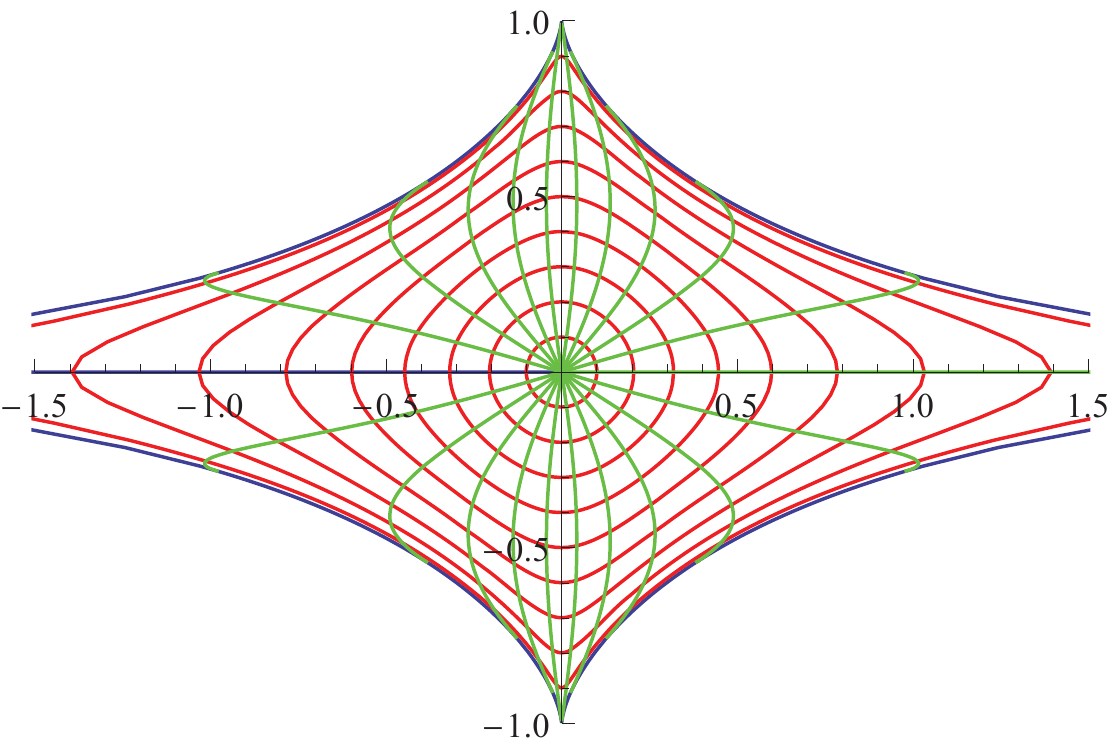}
\end{minipage}}
\subfigure[$a=0.3$]
{\begin{minipage}[b]{0.45\textwidth}
\includegraphics[height=1.6in,width=2.4in,keepaspectratio]{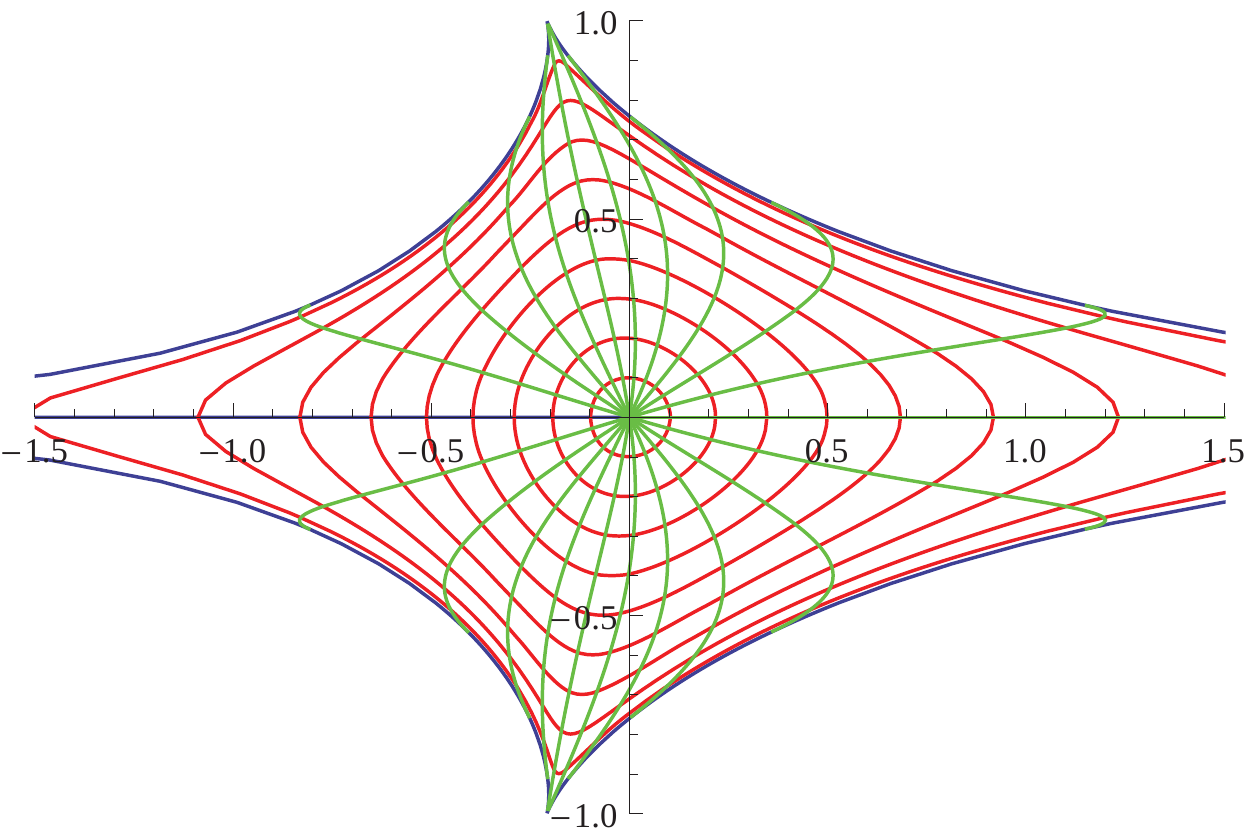}
\end{minipage}}
\subfigure[$a=0.7$]
{\begin{minipage}[b]{0.45\textwidth}
\includegraphics[height=1.6in,width=2.4in,keepaspectratio]{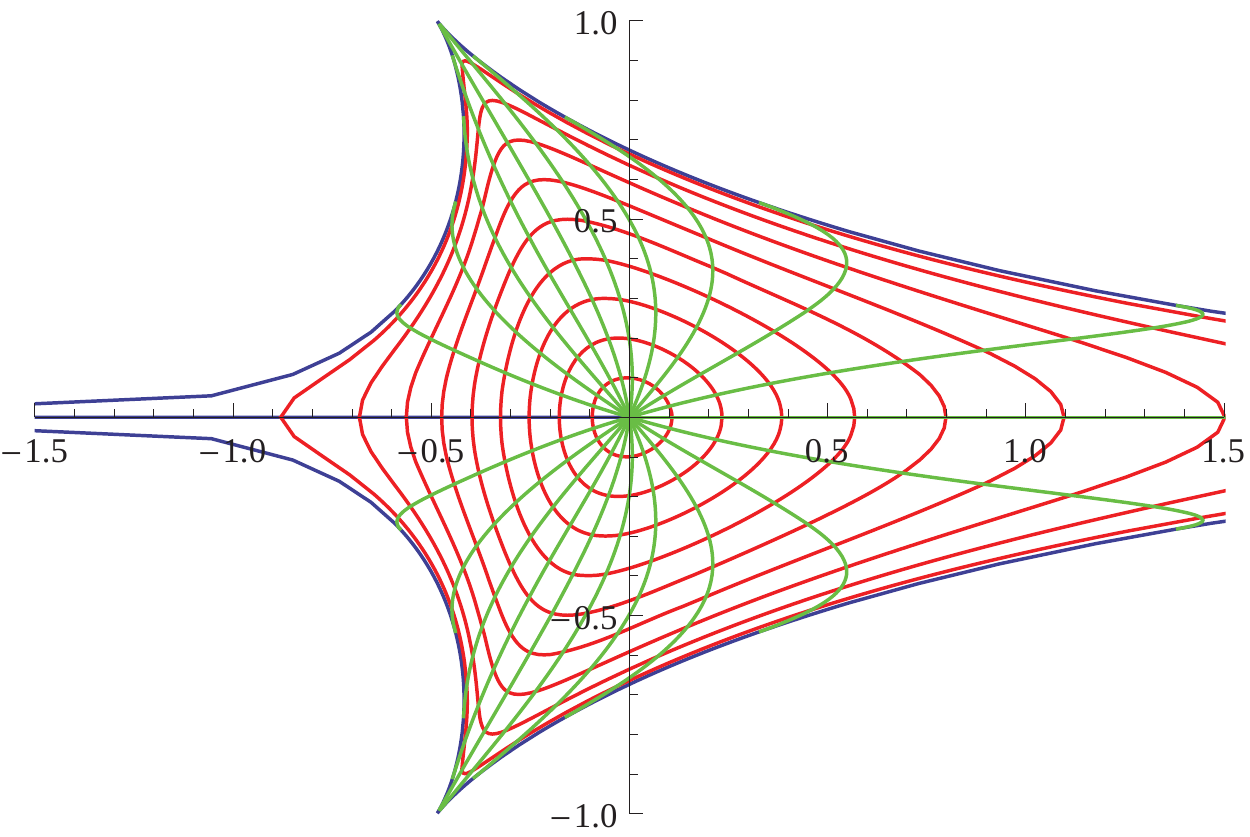}
\end{minipage}}
\subfigure[$a=1$]
{\begin{minipage}[b]{0.45\textwidth}
\includegraphics[height=1.6in,width=2.4in,keepaspectratio]{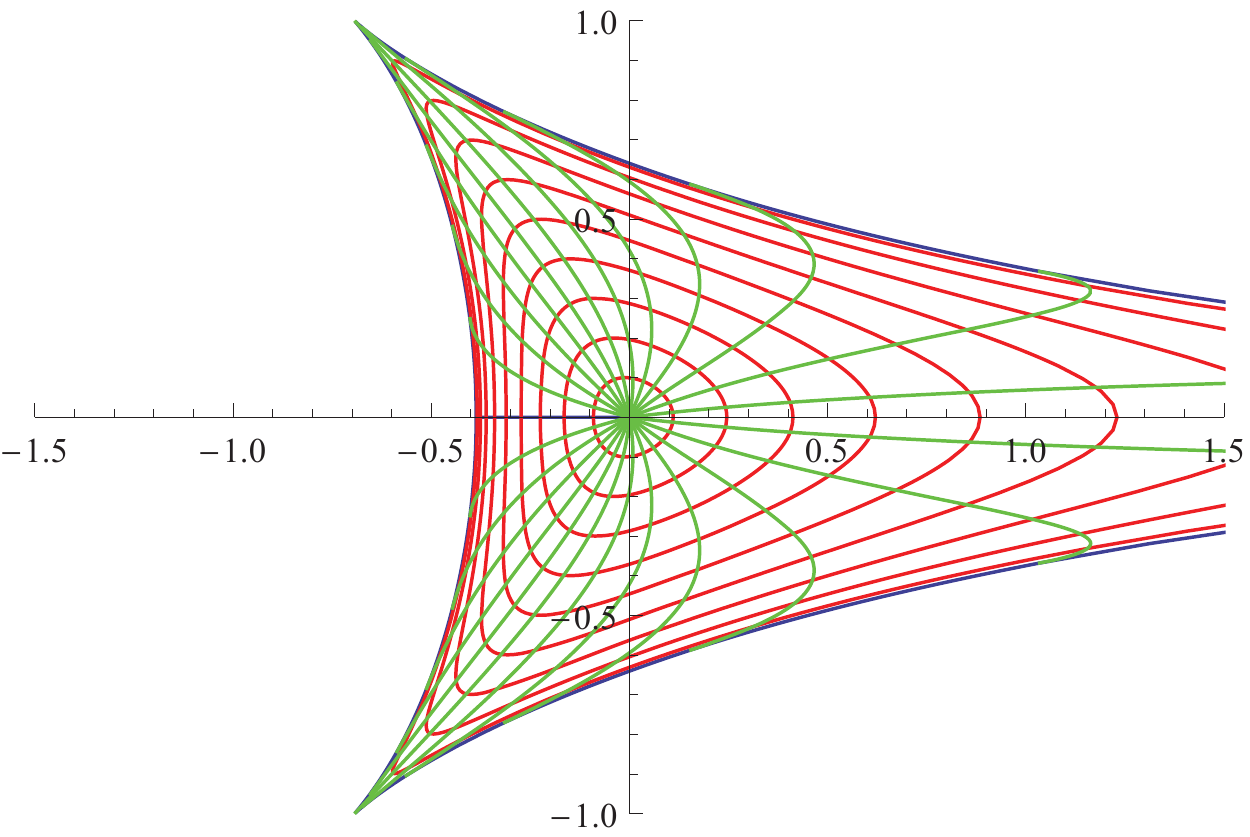}
\end{minipage}}
\caption{Shear of identity mapping when $\omega(z)=z\frac{z+a}{1+az}$. }\label{fida}
\end{figure}
\end{example}

\begin{example}\label{Exam2}
Consider the strip mapping $\varphi(z)=\frac{1}{2}\log\left(\frac{1+z}{1-z}\right)$ which maps $\ID$ onto the horizontal strip $|\IM\{w\}|<\pi/4$.
Then by \eqref{eqf} one obtains CHD mappings
\begin{equation*}
F_{0,a}(z)=\RE\left\{\frac{1+a}{2}\frac{z}{1-z}+\frac{1-a}{2}\frac{z}{1+z}\right\}+i \IM\left\{\frac{1}{2}\log\left(\frac{1+z}{1-z}\right)\right\}.
\end{equation*}
The images of the unit disk under $F_{0,a}$ for $-1\leq a\leq 1$ are shown in
Figure~\ref{stripa}. In particular, the images $F_{0,0}(\ID)$ and $F_{0,1}(\ID)$ are shown in Figures~\ref{stripa}~(a) and (d) (see  also
Duren~\cite[Figures 3.4 and 3.5, Section 3.4]{Duren2004}), respectively. Moreover, the images of $F_{0,a}(z)$ and $F_{0,-a}(z)$ are symmetric about the imaginary axis. Observe that
\begin{equation*}
F_{0,a}(e^{i\theta})
=\left\{ \begin{aligned}
     -\frac{a}{2}+i\frac{\pi}{4} & ~\mbox{ for }~ 0<\theta<\pi\\
     -\frac{a}{2}-i\frac{\pi}{4} & ~\mbox{ for }~  \pi<\theta<2\pi.
\end{aligned} \right.
\end{equation*}
In particular, $F_{0,a}(z)$ collapses the upper and lower semicircles to single point $(-\frac{a}{2},\frac{\pi}{4})$ and
$(-\frac{a}{2},-\frac{\pi}{4})$, respectively.

Actually, we can show that $F_{0,a}(z)$ maps the unit disk $\ID$ onto the full strip
$$\left\{w\in {\mathbb C}:\,|\IM\{w\}|<\pi/4\right\}$$
for $-1<a<1$. We will now show that $-\infty<\RE\{F_{0,a}(z)\}<+\infty$ and  for this, we only need to prove that
$$-\infty<\RE\left\{\frac{1+a}{2}\frac{z}{1-z}+\frac{1-a}{2}\frac{z}{1+z}\right\}<+\infty$$
for $z\in (-1,1)$, where $z=x+iy$. Set
\begin{equation*}
\begin{split}
U(x)=\frac{1+a}{2}\frac{x}{1-x}+\frac{1-a}{2}\frac{x}{1+x}
\end{split}
\end{equation*}
and note that the function $U(x)$ is continuous in the interval $x\in(-1,1)$. Fixing $a$ shows that
$$\lim_{x\to-1^{+}}U(x)=-\infty,~\mbox{ and}\lim_{x\to1^{-}}U(x)=+\infty.$$
Additionally, for the cases $a=-1$ and $a=1$, $F_{0,a}(z)$ maps the unit disk onto the half-strips
$$\left\{w:\RE\{w\}<\frac{1}{2},~|\IM\{w\}|<\frac{\pi}{4}\right\}~\mbox{ and}~\left\{w:\RE\{w\}>-\frac{1}{2},~|\IM\{w\}|<\frac{\pi}{4}\right\},
$$
respectively.
This is complete the proof.
\begin{figure}
\centering
\subfigure[$a=0$]
{\begin{minipage}[b]{0.45\textwidth}
\includegraphics[height=1.6in,width=2.4in,keepaspectratio]{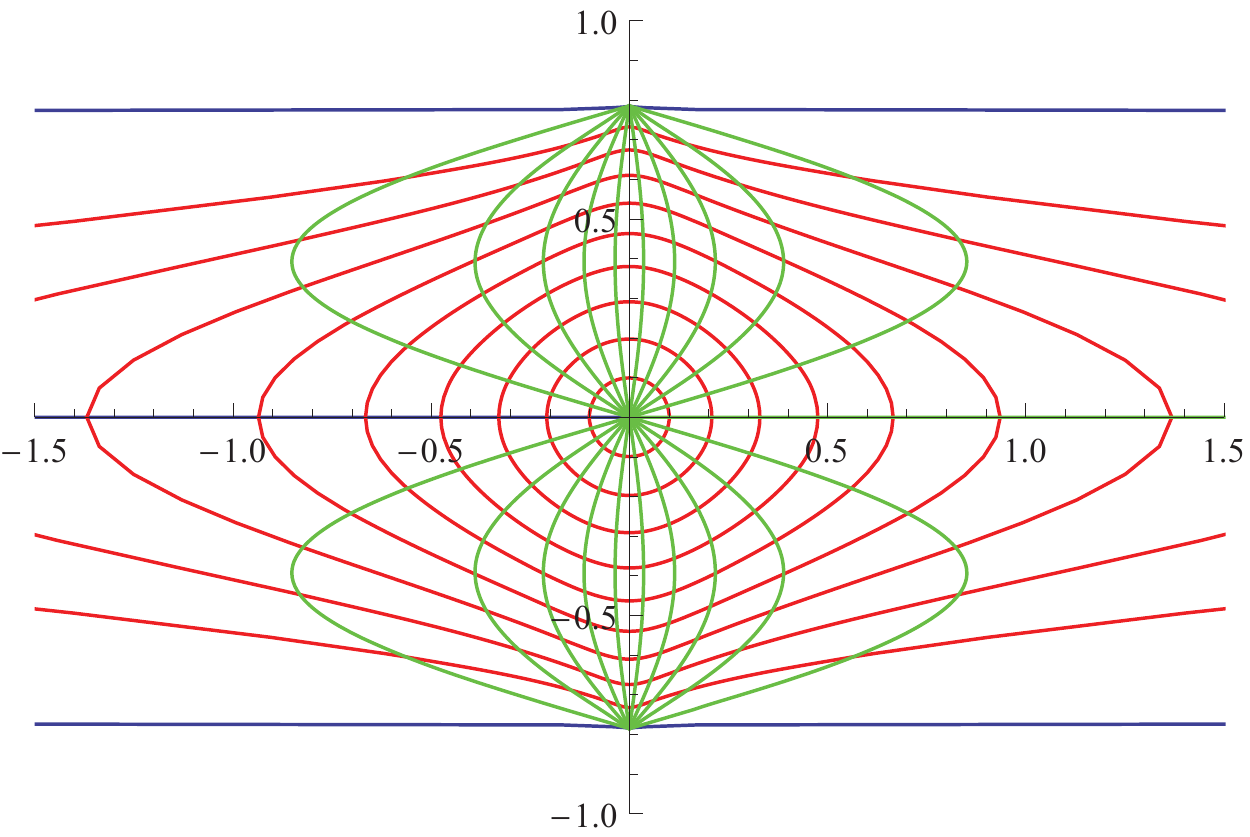}
\end{minipage}}
\subfigure[$a=0.3$]
{\begin{minipage}[b]{0.45\textwidth}
\includegraphics[height=1.6in,width=2.4in,keepaspectratio]{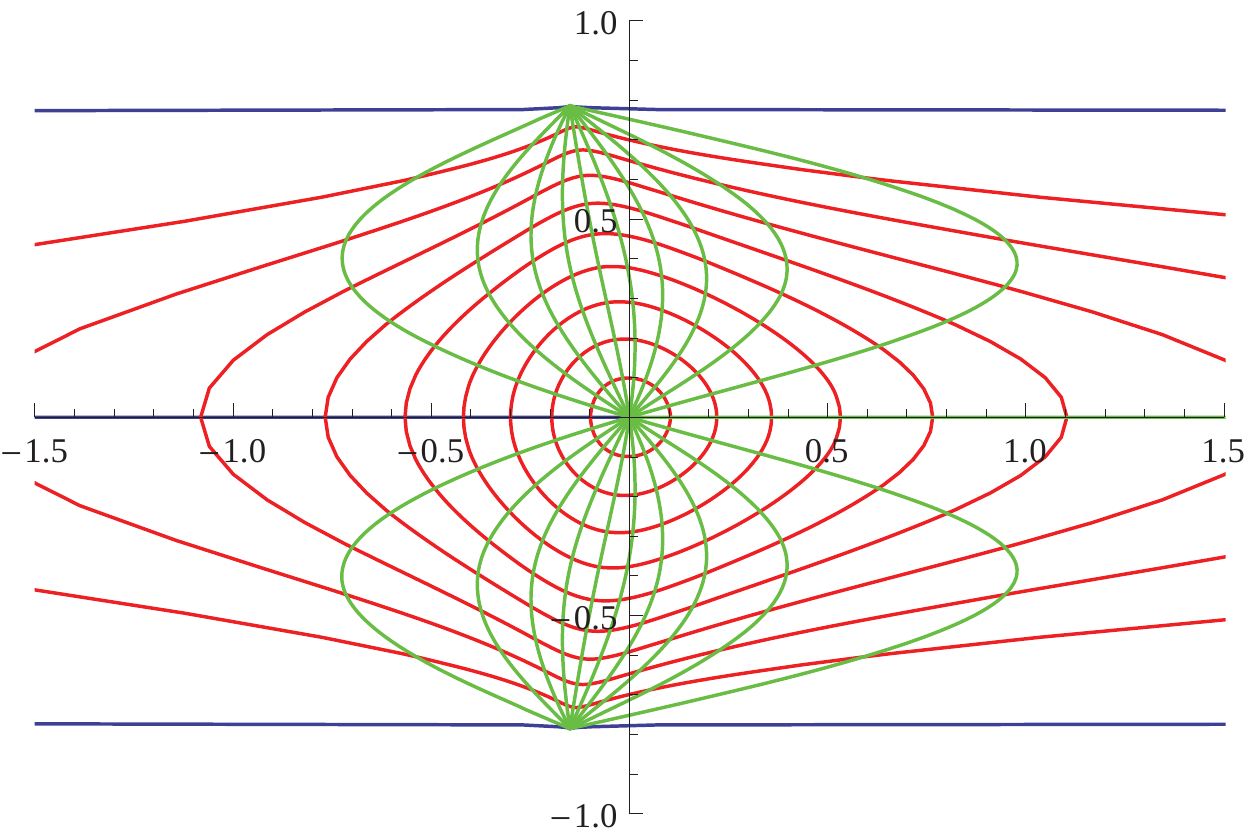}
\end{minipage}}
\subfigure[$a=0.6$]
{\begin{minipage}[b]{0.45\textwidth}
\includegraphics[height=1.6in,width=2.4in,keepaspectratio]{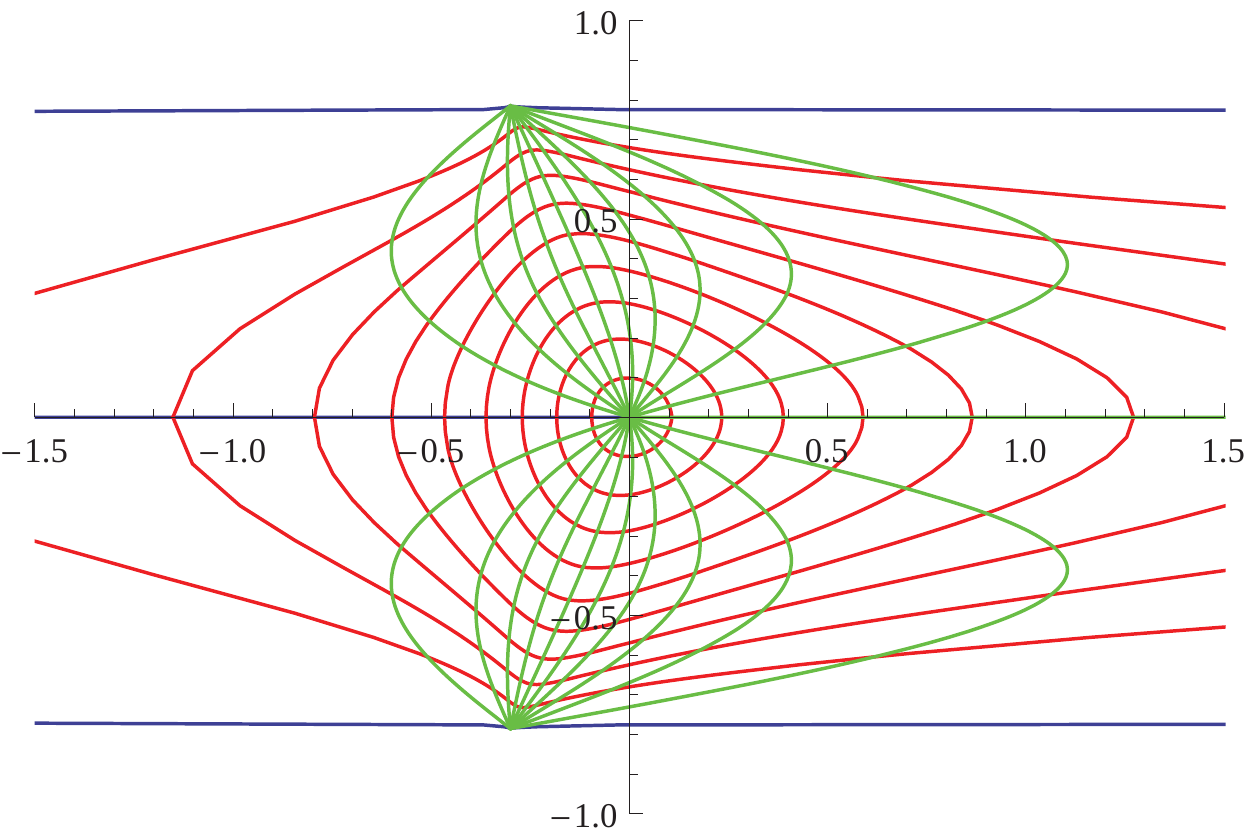}
\end{minipage}}
\subfigure[$a=1$]
{\begin{minipage}[b]{0.45\textwidth}
\includegraphics[height=1.6in,width=2.4in,keepaspectratio]{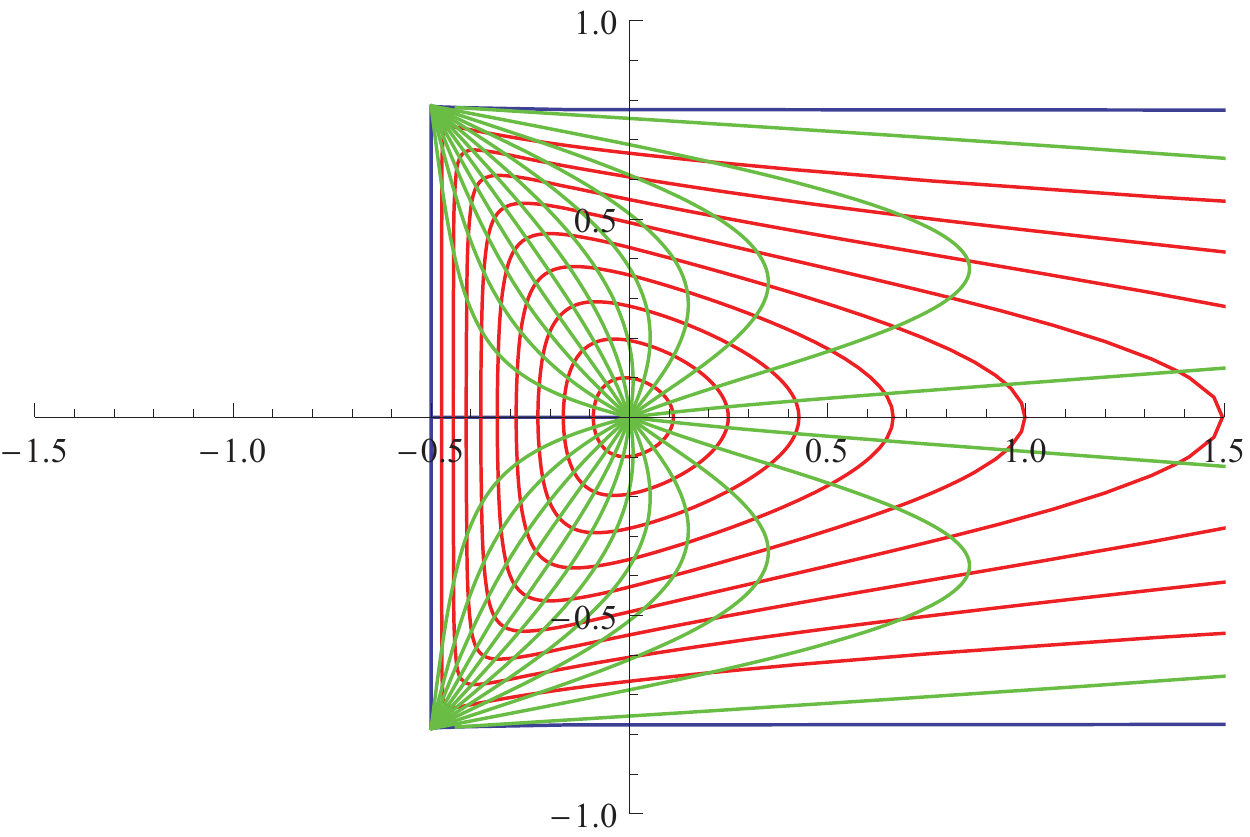}
\end{minipage}}
\caption{Shear of strip mapping for certain values of $a$ in $\omega(z)=z\frac{z+a}{1+az}$. }\label{stripa}
\end{figure}
\end{example}

\begin{example}
Now, consider the half-plane mapping  $\varphi(z)=\frac{z}{1-z}$. Then by \eqref{eqf} shear construction produces the harmonic mappings
(see Figure \ref{halfplane})
\begin{equation*}
F_{1,a}(z)=\RE\left\{\frac{1-a}{4}\log\left(\frac{1+z}{1-z}\right)+\frac{1+a}{2}\frac{z}{(1-z)^2}\right\}+i \IM\left\{\frac{z}{1-z}\right\}.
\end{equation*}
Note that
\begin{equation*}
\begin{split}
\RE\{F_{1,a}(re^{-i\theta})\}&=\frac{1}{8} \bigg\{\frac{4 (a+1) r \left(\left(r^2+1\right) \cos\theta-2 r\right)}{\left(r^2-2 r \cos\theta+1\right)^2}\\
&\qquad+(a-1) \left(\log \left(r^2-2 r \cos\theta+1\right)-\log \left(r^2+2 r \cos\theta+1\right)\right)\bigg\}\\
&=\RE\{F_{1,a}(re^{i\theta})\},
\end{split}
\end{equation*}
and
$$\IM\{F_{1,a}(re^{-i\theta})\}=- \frac{r \sin\theta}{r^2-2 r \cos\theta+1}=-\IM\{F_{1,a}(re^{i\theta})\}.
$$
which imply that the range $F_{1,a}(\ID)$ is symmetric about the real axis.
Since
\begin{equation*}
\begin{split}
F_{1,a}(e^{i\theta})&=\RE\left\{\frac{1-a}{4}\log\left(\frac{1+e^{i\theta}}{1-e^{i\theta}}\right)+\frac{1+a}{2}\frac{e^{i\theta}}{(1-e^{i\theta})^2}\right\}+i \IM\left\{\frac{e^{i\theta}}{1-e^{i\theta}}\right\}\\
&=\frac{1-a}{8}\log\left(\frac{1+\cos\theta}{1-\cos\theta}\right)-\frac{1+a}{4}\frac{1}{1-\cos\theta}+i\frac{1}{2} \cot \frac{\theta}{2}\\
&=:u+iv,
\end{split}
\end{equation*}
we easily find that
$$u=\frac{1-a}{8}\log\left(4v^2\right)-\frac{1+a}{8}(4v^2+1).
$$
In particular, $F_{1,1}(e^{i\theta})$ is the parabola $v^2=-u+(1/4).$
\begin{figure}[!h]
\centering
\subfigure[$a=-1$]
{\begin{minipage}[b]{0.45\textwidth}
\includegraphics[height=2.0in,width=2.0in,keepaspectratio]{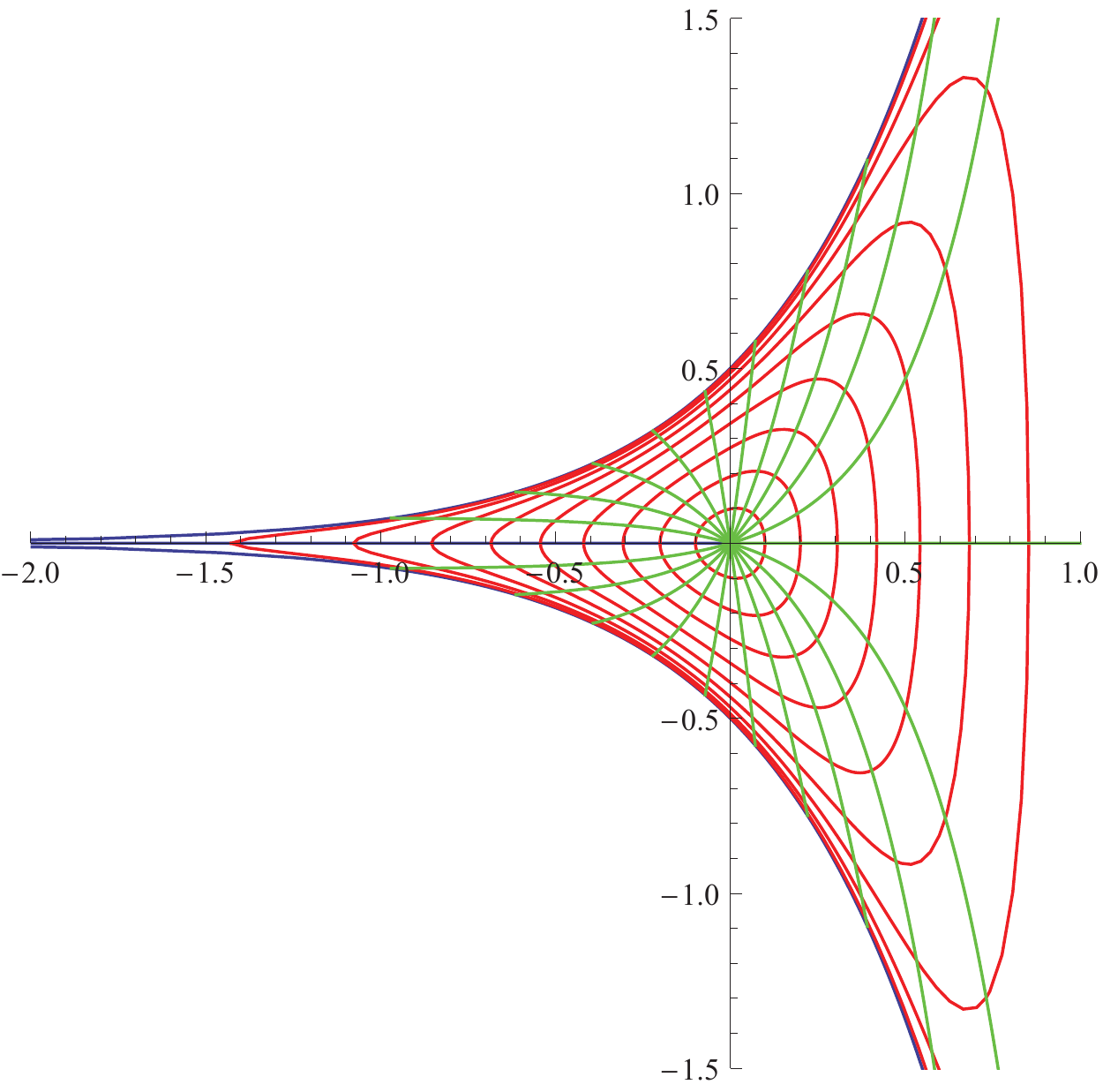}
\end{minipage}}
\subfigure[$a=-0.4$]
{\begin{minipage}[b]{0.45\textwidth}
\includegraphics[height=2.0in,width=2.0in,keepaspectratio]{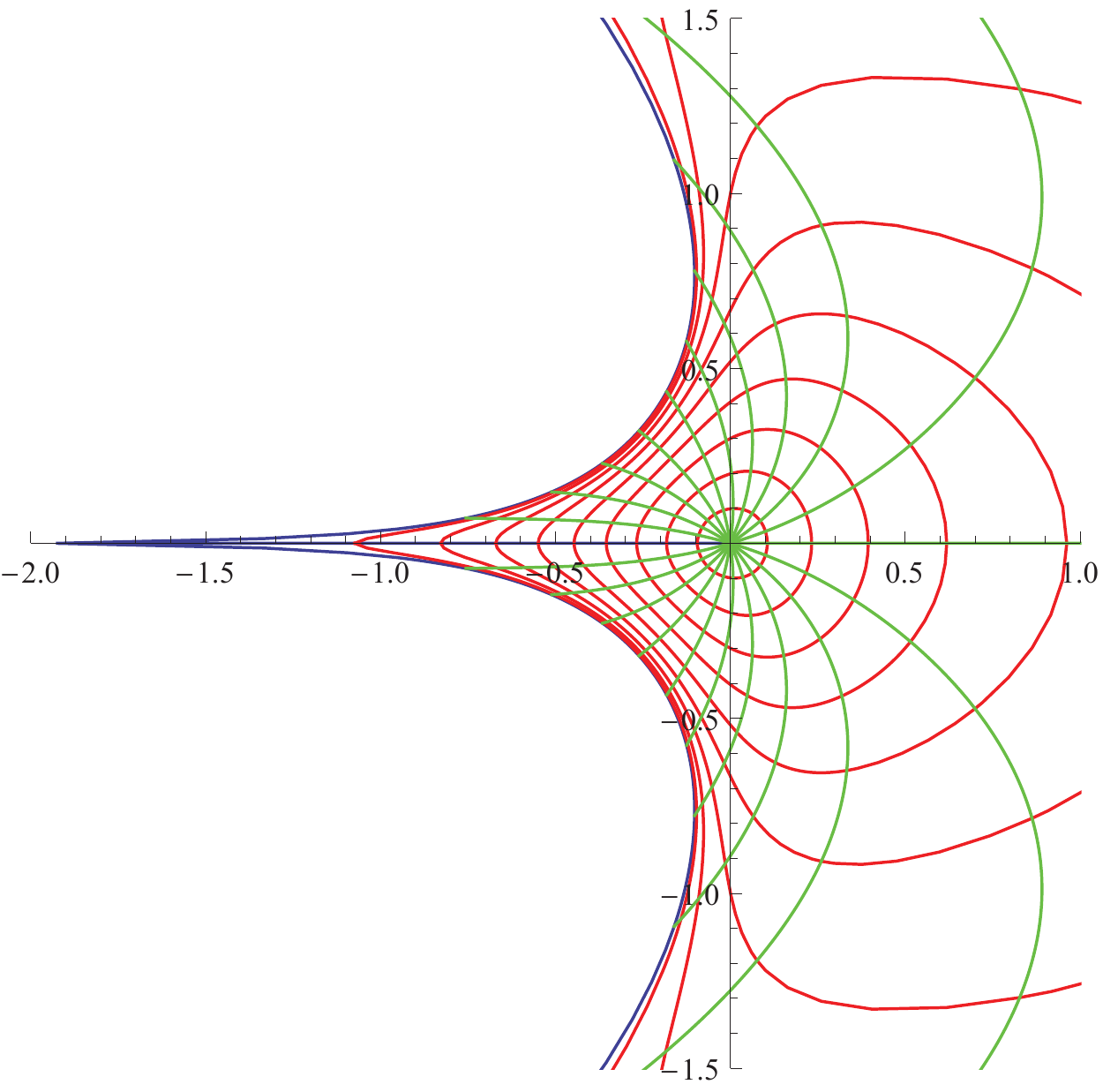}
\end{minipage}}
\subfigure[$a=0$]
{\begin{minipage}[b]{0.45\textwidth}
\includegraphics[height=2.0in,width=2.0in,keepaspectratio]{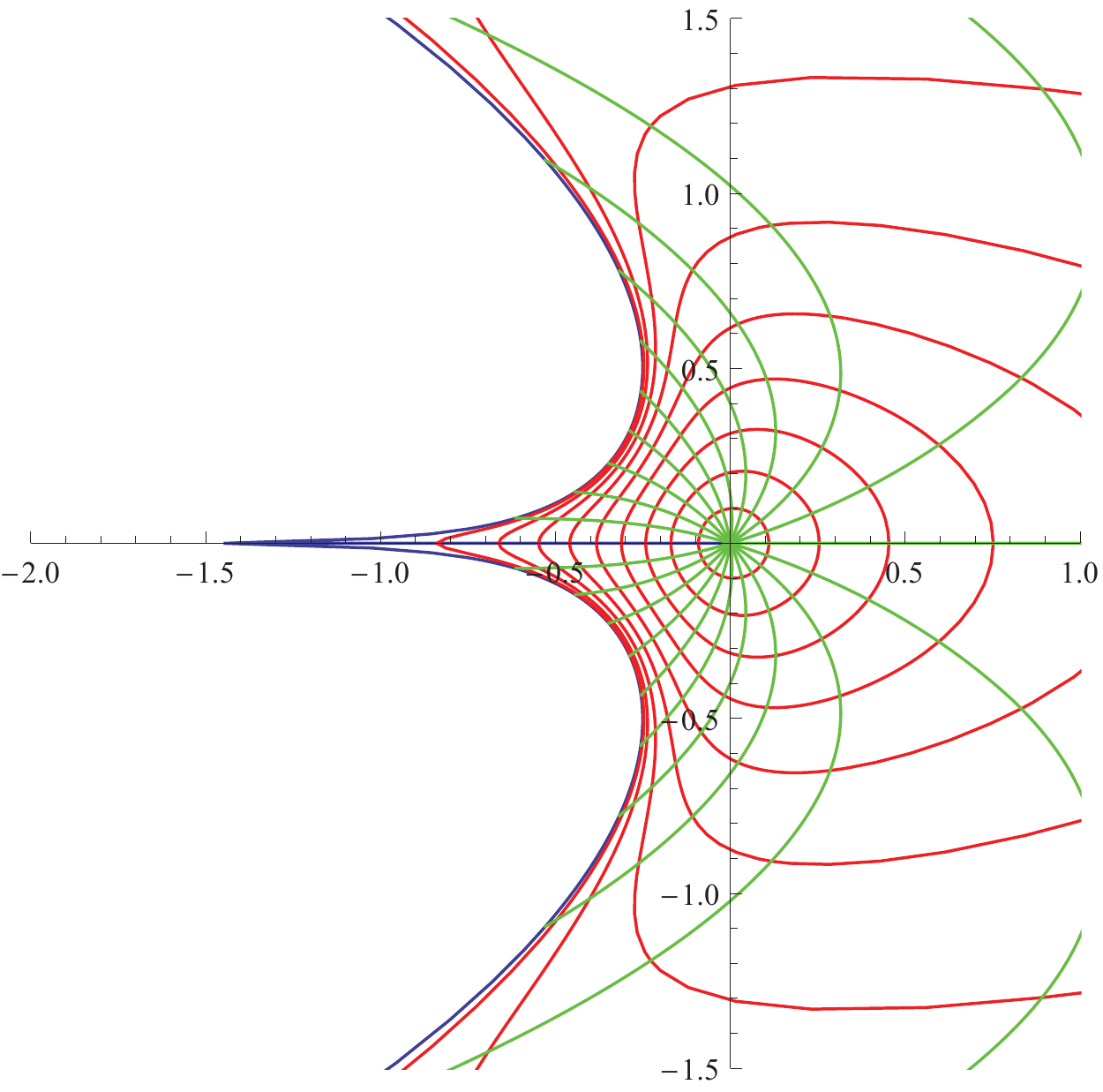}
\end{minipage}}
\subfigure[$a=0.4$]
{\begin{minipage}[b]{0.45\textwidth}
\includegraphics[height=2.0in,width=2.0in,keepaspectratio]{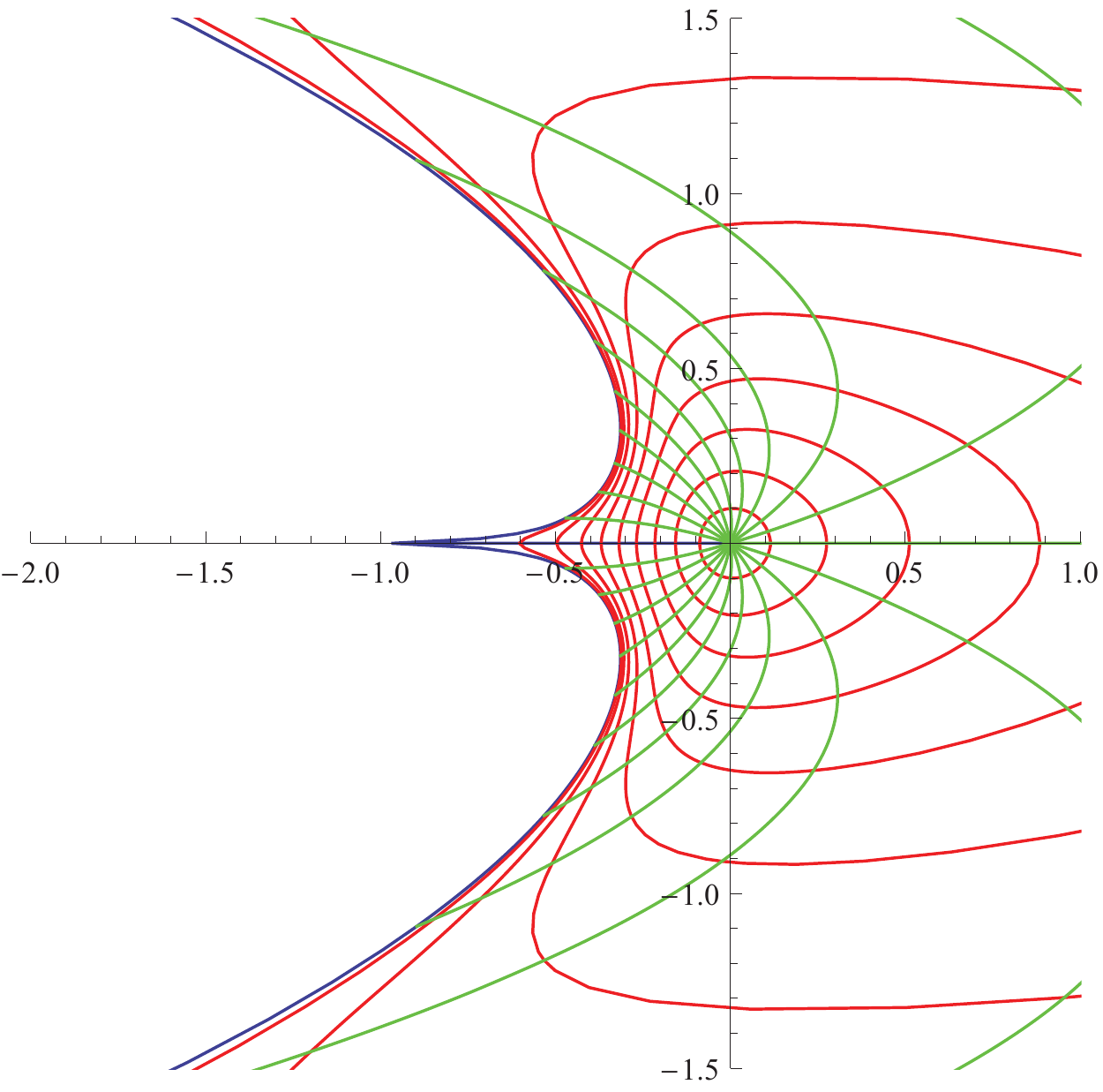}
\end{minipage}}
\subfigure[$a=1$]
{\begin{minipage}[b]{0.45\textwidth}
\includegraphics[height=2.0in,width=2.0in,keepaspectratio]{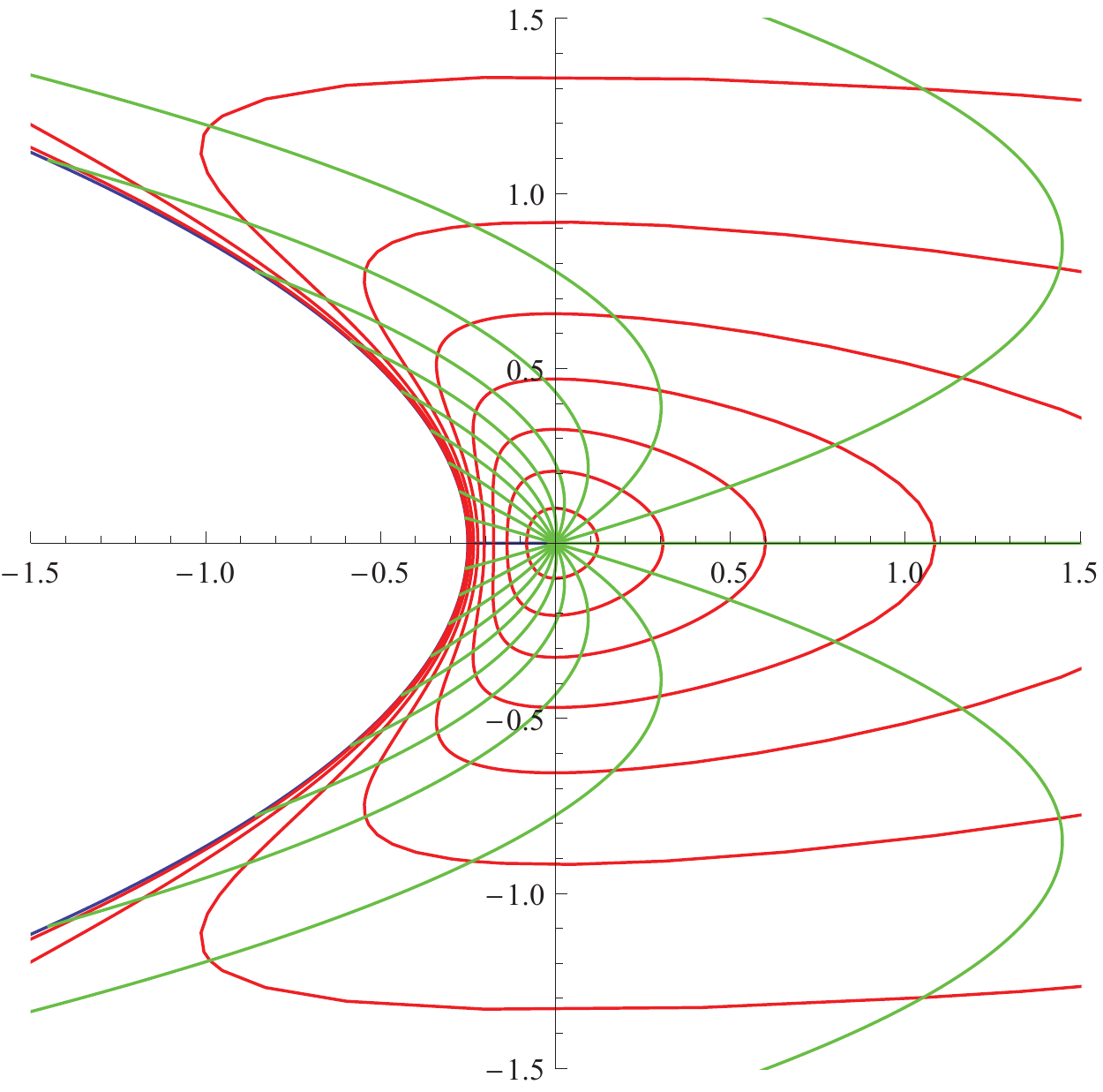}
\end{minipage}}
\caption{Shearing of the half-plane mapping for certain values of $a$ with $\omega(z)=z\frac{z+a}{1+az}$. }\label{halfplane}
\end{figure}
\end{example}

\begin{theorem}\label{thmca} For $c\in [0,2]$, and $a\in[-1,1]$, let $F_{c,a}=H_{c,a}+\overline{G_{c,a}}\in \mathcal{S}_{H}^{0}$ such that
\begin{equation}\label{wca}
H_{c,a}(z)-G_{c,a}(z)=k_{c}(z)\quad {\rm and}\quad \omega_{a}(z)=z\frac{z+a}{1+az},
 \end{equation}
where $k_{c}(z)$ is given by~\eqref{GKF}. Then $F_{c,a}(\mathbb{D})$ is convex in the horizontal direction, and as $c$ varies from $0$ to $2$, $F_{c,a}(\mathbb{D})$ transforms from a strip mapping to a slit mapping.
\end{theorem}
\begin{proof} According to Theorem~\ref{thmA}, we obtain that $F_{c,a}(\mathbb{D})$ is convex in the horizontal direction. Throughout the proof,
it suffices to assume that $c\in (0,2]\backslash\{1\}$.
By \eqref{wca}, we have
$$H'_{c,a}(z)-G'_{c,a}(z)=\frac{1}{(1+z)(1-z)}\left(\frac{1+z}{1-z}\right)^{c}\quad {\rm and}\quad G'_{c,a}(z)=z\frac{z+a}{1+az}H'_{c,a}(z).
$$
Solving these two equations, we obtain
\begin{equation*}
\begin{split}
H'_{c,a}(z)&=\frac{1+az}{(1+z)^2(1-z)^2}\left(\frac{1+z}{1-z}\right)^{c}\\
&=\left\{\frac{1}{4}\left(\frac{1}{1-z}
+\frac{1}{1+z}\right)+\frac{1+a}{4}\frac{1}{(1-z)^2}+\frac{1-a}{4}\frac{1}{(1+z)^2}\right\}\left(\frac{1+z}{1-z}\right)^{c}.
\end{split}
\end{equation*}
Straightforward integration gives
\begin{equation*}
H_{c,a}(z)=\frac{(1-2c^2+ac)+2c(1-ac)z+(ac-1)z^2}{4c(1-c^2)(1-z^2)}\left(\frac{1+z}{1-z}\right)^{c}
 -\frac{1+ac-2c^2}{4c(1-c^2)}
\end{equation*}
and thus, we find that
\begin{equation*}
\begin{split}
G_{c,a}(z)&=h_{c,a}(z)-k_{c}(z)\\
&=\frac{(1-2c^2+ac)+2c(1-ac)z+(ac-1)z^2}{4c(1-c^2)(1-z^2)}\left(\frac{1+z}{1-z}\right)^{c}\\
&\qquad-\frac{1+ac-2c^2}{4c(1-c^2)}-\frac{1}{2c}\left(\left(\frac{1+z}{1-z}\right)^{c}-1\right).
\end{split}
\end{equation*}
In order to study the mapping properties of $F_{c,a}$, we perform a change of variables using
$$w=\frac{1+z}{1-z}.$$
With $z=(w-1)/(w+1)$, this substitution leads to
\begin{equation*}
\begin{split}
H_{c,a}(z)=\frac{1}{8}\left(\frac{a+1}{c+1}w^{c+1}+\frac{2}{c}w^c-\frac{a-1}{c-1}w^{c-1}-\frac{2(1+ac-2c^2)}{c(1-c^2)}\right)
\end{split}
\end{equation*}
and
\begin{equation*}
\begin{split}
G_{c,a}(z)=\frac{1}{8}\left(\frac{a+1}{c+1}w^{c+1}-\frac{2}{c}w^c-\frac{a-1}{c-1}w^{c-1}+\frac{2(1-ac)}{c(1-c^2)}\right)
\end{split}
\end{equation*}
which show that
\begin{equation}\label{eqFca}
\begin{split}
F_{c,a}(z)=\RE\left\{\frac{1}{4} \left(\frac{a+1}{c+1}w^{c+1}-\frac{a-1}{c-1}w^{c-1}+\frac{2 (c-a)}{1-c^2}\right)\right\}+i \IM\left\{\frac{1}{2 c}\left(w^c-1\right)\right\}.
\end{split}
\end{equation}
By writing $w=x+iy,~x>0$ and $y\in \mathbb{R}$, from Example~\ref{Exam2}, we see that $F_{0,a}$ maps $\mathbb{D}$ onto the strip $\{\zeta\in \mathbb{C}:|\IM \zeta |<\pi/4\}$.
If we substitute $a=2$ into~\eqref{eqFca}, then one has
\begin{equation*}
\begin{split}
F_{2,a}(z)&=\RE\left\{\frac{1}{4} \left(\frac{1+a}{3}w^{3}+(1-a)w-\frac{2 (2-a)}{3}\right)\right\}+i \IM\left\{\frac{1}{4}\left(w^2-1\right)\right\}\\
&=\frac{1}{4} \left(\frac{1+a}{3}(x^3-3xy^2)+(1-a)x-\frac{2 (2-a)}{3}\right)+i \frac{1}{2}xy,\quad x>0.
\end{split}
\end{equation*}
Observe now that each point $z\neq 1$ on the unit circle is carried onto a point $w$ on the imaginary axis so that $x=0$ and $F_{2,a}(z)=-(2-a)/6$. Similar discussion as in the case of harmonic Koebe function $K(z)$ in~\cite[Page 84-86]{Duren2004}
proves that $F_{2,a}(z)$ maps the unit disk $\ID$ onto the entire plane minus the real interval $(-\infty,-(2-a)/6]$.
\end{proof}

\begin{remark}
When $a=0$, $\omega_{a}(z)$ becomes $z^2$ and thus, Theorem~\ref{thmca} reduces to Theorem~\ref{thmD}, which can be lifted to the minimal surface in view of Theorem~\ref{thmB}. If $a=-1$ and $a=1$, then $\omega_{a}(z)$ becomes $-z$ and $z$, respectively, and hence, Theorem~\ref{thmca} is a generalization of the Theorem~\ref{thmD}.
\end{remark}

\section{Shearing Construction and Minimal Surfaces}\label{LJS1-sec3}
In this section, we use  Theorem~\ref{thmA} to build a family of harmonic univalent mappings with a CHD range that lifts to a family of minimal surfaces as described in Theorem~\ref{thmB}.

In 2004, Greiner~\cite{Greiner2004} constructed horizontal strip harmonic mappings with dilatation $\omega(z)=z^{n}$ by shearing  $$h_{0,n}(z)-g_{0,n}(z)=\frac{1}{2}\log\left(\frac{1+z}{1-z}\right).
$$
After tedious but straightforward calculation, the shear construction produces the harmonic mapping $f_{0,n}(z)$ defined for $n=2m+1$ ($m\in\IN$) by
\begin{equation*}
\begin{split}
f_{0,n}(z)&=\RE\left\{h_{0,n}(z)+g_{0,n}(z)\right\}+i\IM\left\{h_{0,n}(z)-g_{0,n}(z)\right\}\\
&=\RE\left\{\frac{1}{n}\left(\frac{z}{1-z}-i\sum_{k=1}^{(n-1)/2}\csc\frac{2k\pi}{n}
\log\left(\frac{1-ze^{-i\frac{2k\pi}{n}}}{1-ze^{i\frac{2k\pi}{n}}}\right)\right)\right\}\\
&\qquad  +i\IM\left\{\frac{1}{2}\log\left(\frac{1+z}{1-z}\right)\right\}.
\end{split}
\end{equation*}
Moreover, if $n=2m$ ($m\in\IN$), by virtue of Theorem~\ref{thmB}, $f_{0,n}(\mathbb{D})$  lifts to the minimal surfaces $\mathbf{X}_{0,n}(u,v)=(u,v,F(u,v))$, where
\begin{equation*}
\begin{aligned}
u&=\RE\left\{\frac{1}{n}\left(\frac{2z}{1-z^2}-i\sum_{k=1}^{(n/2)-1}
\csc\frac{2k\pi}{n}\log\left(\frac{1-ze^{-i\frac{2k\pi}{n}}}{1-ze^{i\frac{2k\pi}{n}}}\right)\right)\right\},\\
v&=\IM\left\{\frac{1}{2}\log\left(\frac{1+z}{1-z}\right)\right\}, ~\mbox{ and }\\
F(u,v)&=\IM\left\{\frac{1}{n}\left(\frac{z}{1-z}
+\frac{(-1)^{n/2}z}{1+z}-i\sum_{k=1}^{(n/2)-1}
(-1)^{k}\csc\frac{2k\pi}{n}\log\left(\frac{1-ze^{-i\frac{2k\pi}{n}}}{1-ze^{i\frac{2k\pi}{n}}}\right)\right)\right\}.
\end{aligned}
\end{equation*}

\begin{remark}
For $\omega(z)=z$, the expression for $f_{0,1}(z)$ simplifies to
$$f_{0,1}(z)=\RE\left\{\frac{z}{1-z}\right\}+i\IM\left\{\frac{1}{2}\log\left(\frac{1+z}{1-z}\right)\right\}.
$$
The image is shown in~\cite[Figure 3.4]{Duren2004}.

For $\omega(z)=z^2$, the expression of $f_{0,2}(z)$ is given by (see \cite[Theorem 3]{Dorff2014AAA} or~\cite[Figure 3.5]{Duren2004})
$$f_{0,2}(z)=\RE\left\{\frac{z}{1-z^2}\right\}+i\IM\left\{\frac{1}{2}\log\left(\frac{1+z}{1-z}\right)\right\}.
$$
In Figure~\ref{f0nMS}, we illustrate the minimal surfaces of the harmonic mappings $f_{0,n}(z)$ onto strip domains whenever $\omega(z)=z^{n}$ with $n=4,6,8,10,12,14$.
\begin{figure}[!h]
\centering
\subfigure[$c=0,n=4$]
{\begin{minipage}[b]{0.45\textwidth}
\includegraphics[height=2.4in,width=2.4in,keepaspectratio]{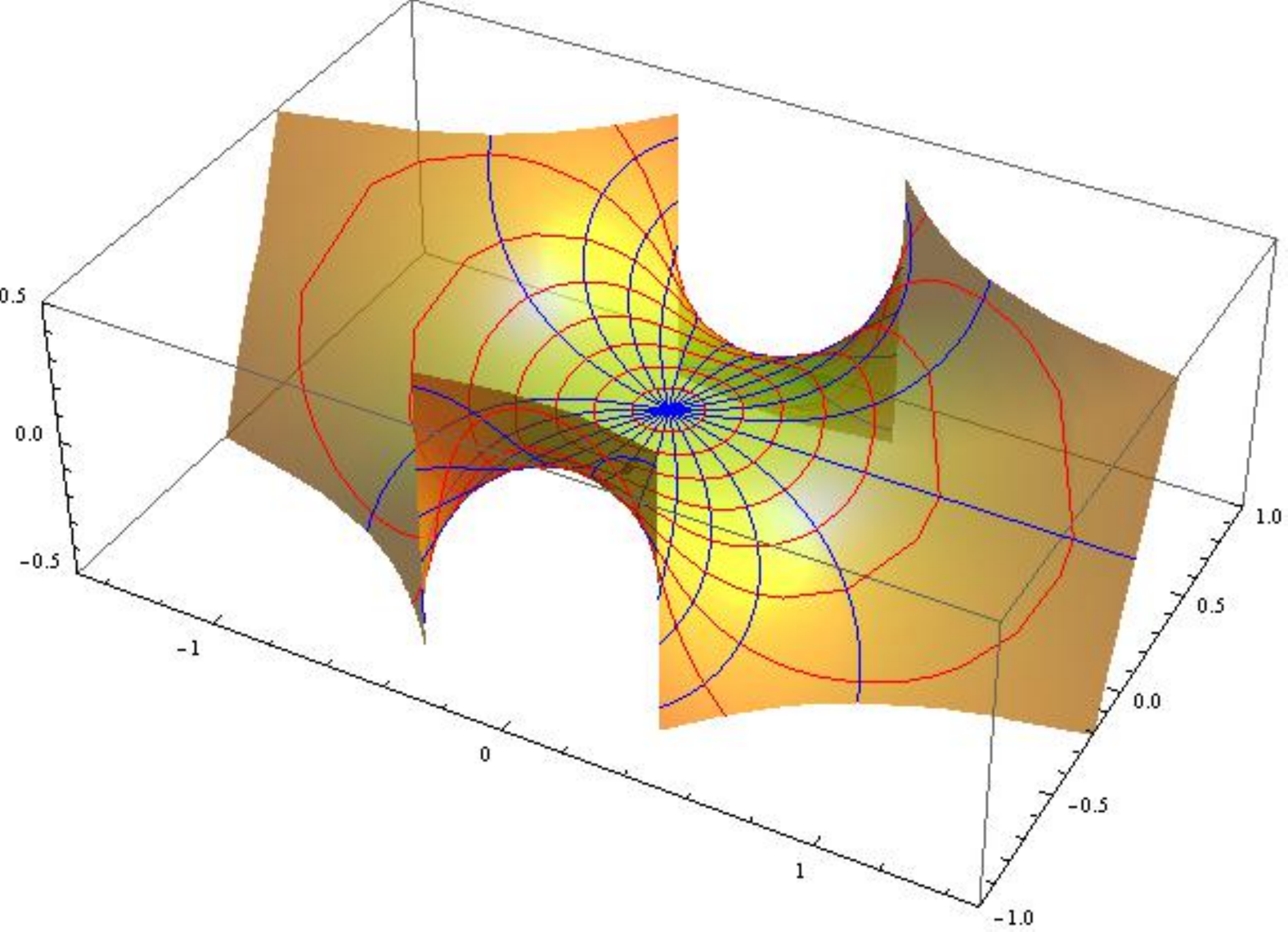}
\end{minipage}}
\subfigure[$c=0,n=6$]
{\begin{minipage}[b]{0.45\textwidth}
\includegraphics[height=2.4in,width=2.4in,keepaspectratio]{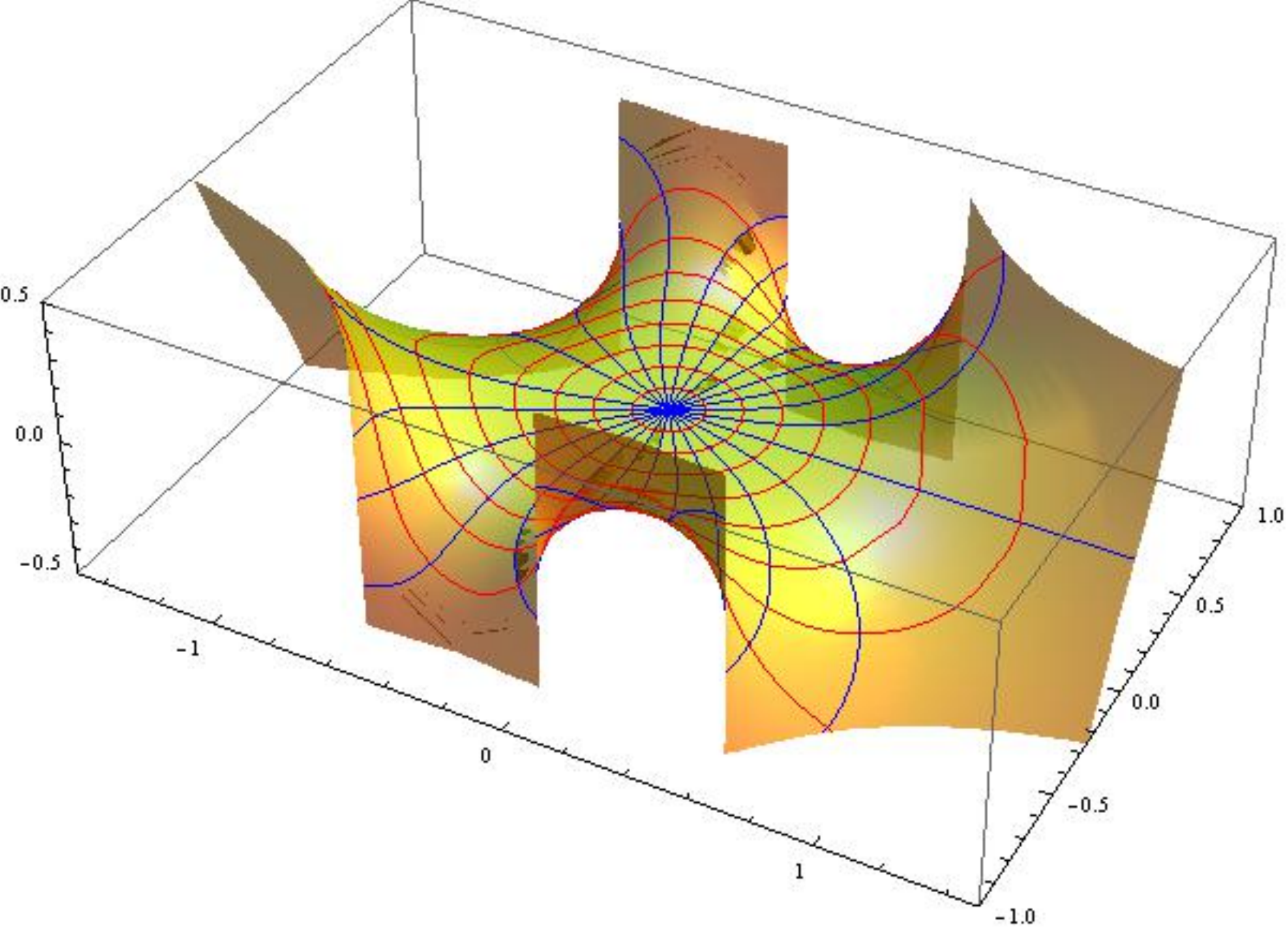}
\end{minipage}}
\subfigure[$c=0,n=8$]
{\begin{minipage}[b]{0.45\textwidth}
\includegraphics[height=2.4in,width=2.4in,keepaspectratio]{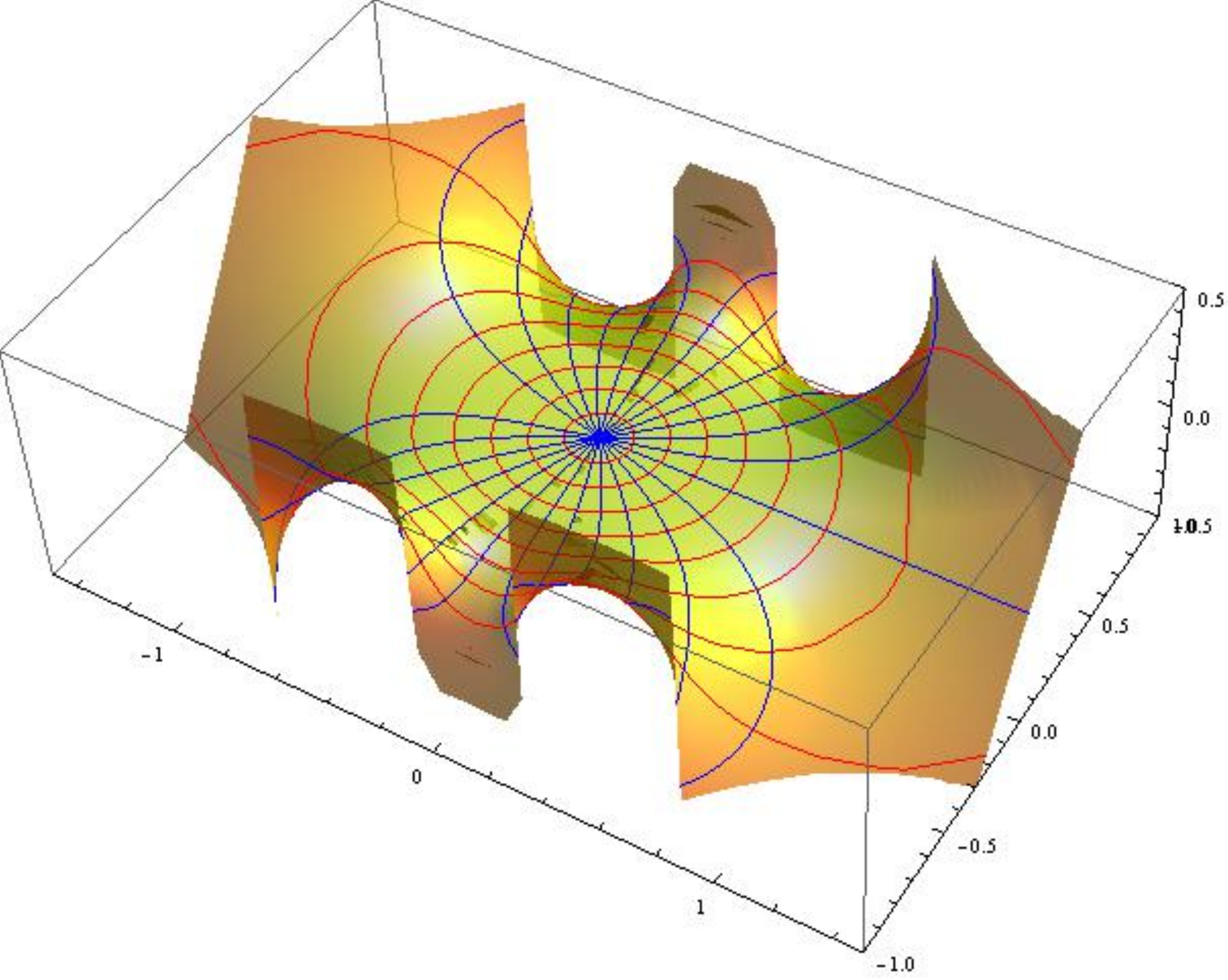}
\end{minipage}}
\subfigure[$c=0,n=10$]
{\begin{minipage}[b]{0.45\textwidth}
\includegraphics[height=2.4in,width=2.4in,keepaspectratio]{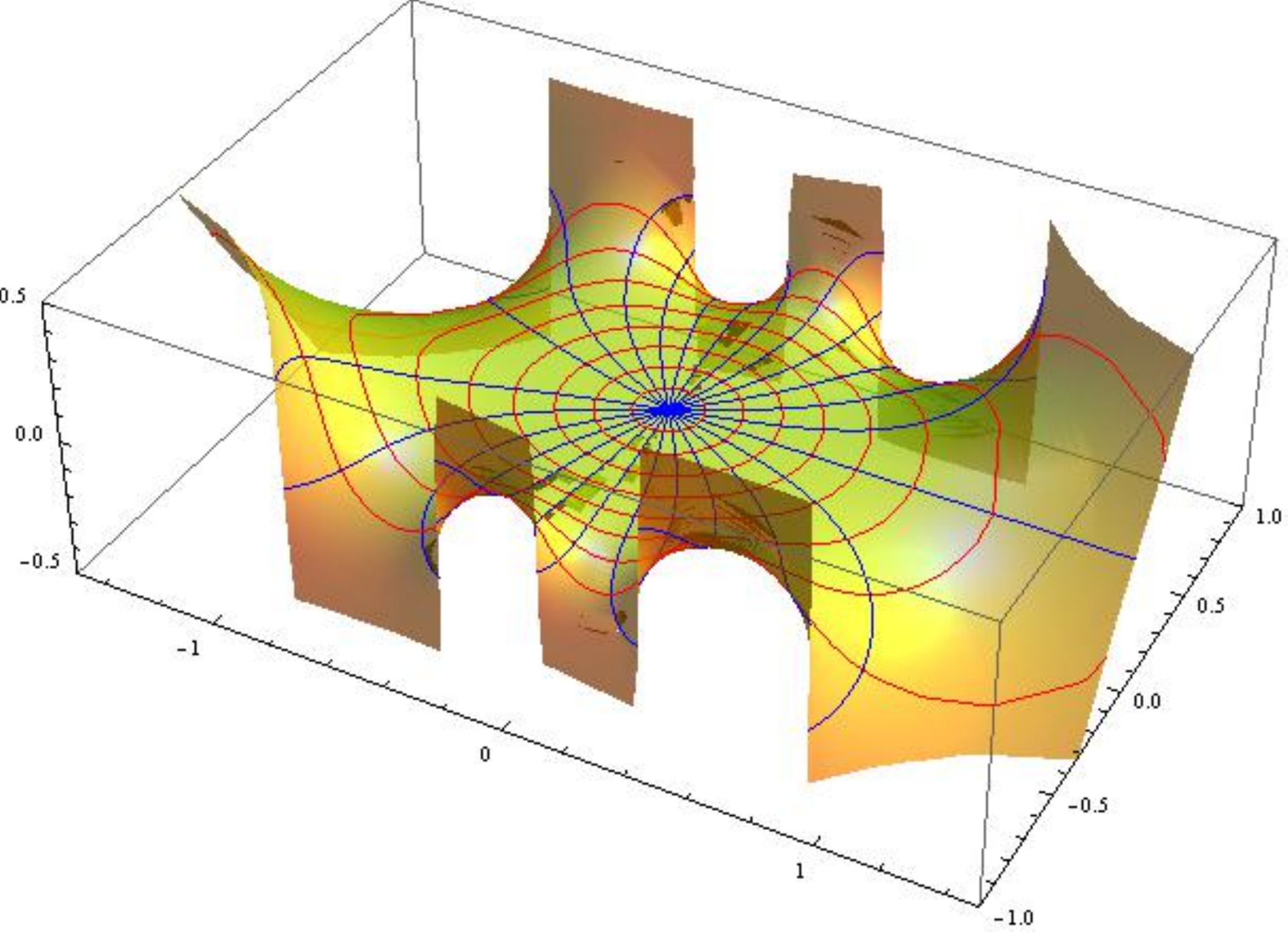}
\end{minipage}}
\subfigure[$c=0,n=12$]
{\begin{minipage}[b]{0.45\textwidth}
\includegraphics[height=2.4in,width=2.4in,keepaspectratio]{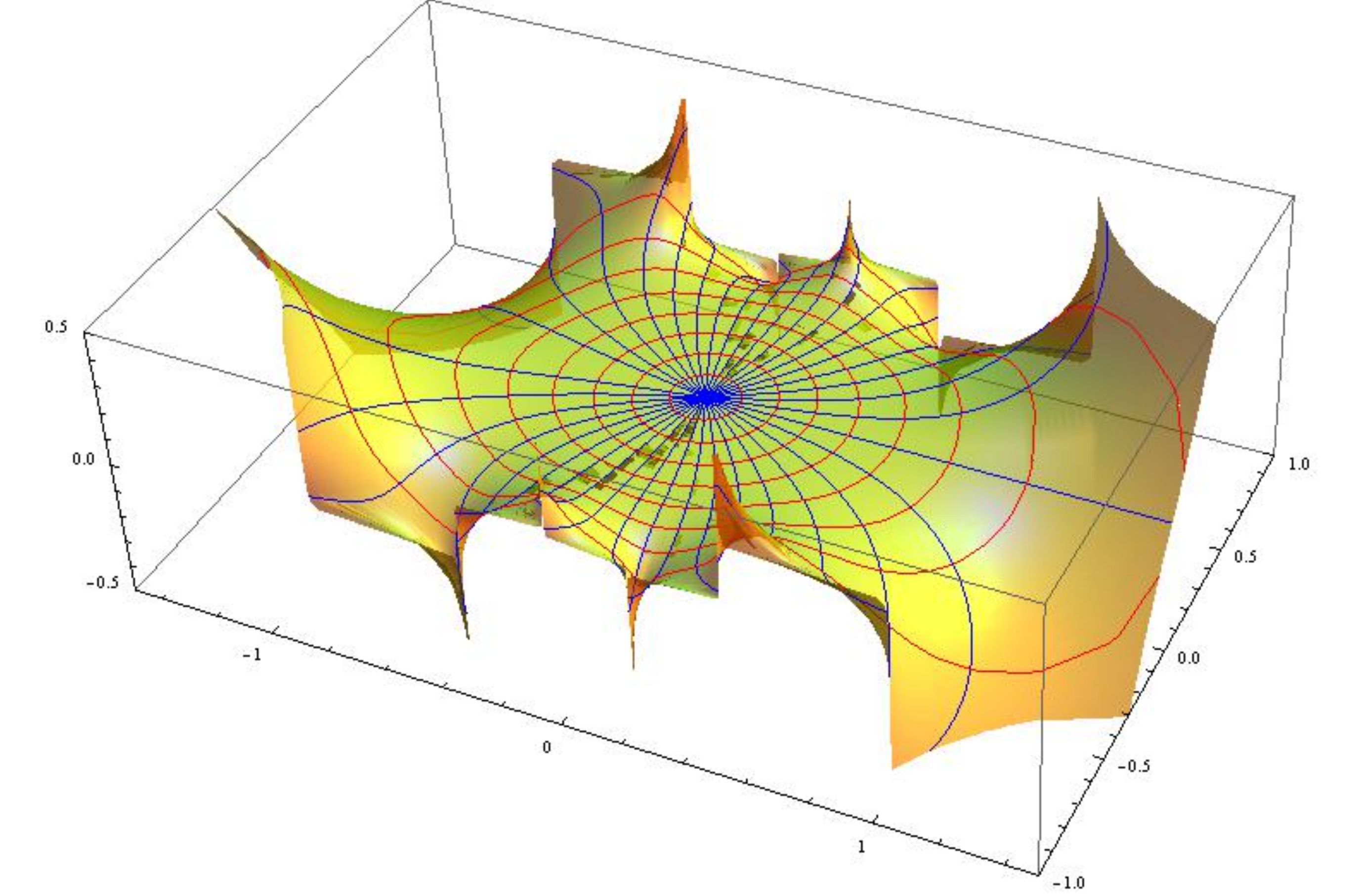}
\end{minipage}}
\subfigure[$c=0,n=14$]
{\begin{minipage}[b]{0.45\textwidth}
\includegraphics[height=2.4in,width=2.4in,keepaspectratio]{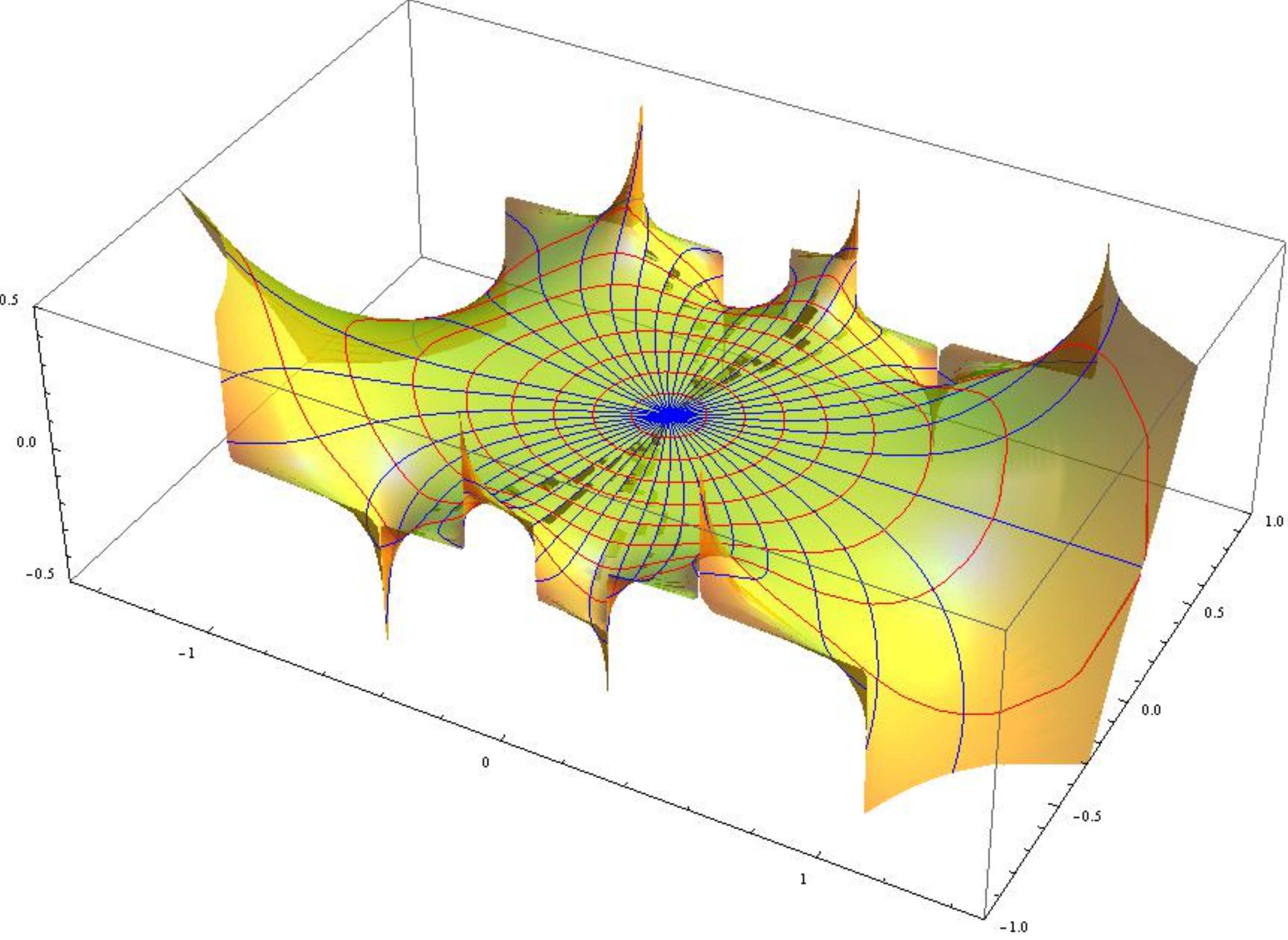}
\end{minipage}}
\caption{$f_{0,n}(\mathbb{D})$ lift to the minimal surfaces for certain values of $c=0,n=4,6,8,10,12,14$. }\label{f0nMS}
\end{figure}
\end{remark}

\begin{theorem}\label{thmc1} Let $f_{1,n}=h_{1,n}+\overline{g_{1,n}}\in \mathcal{S}_{H}^{0}$ such that
\begin{equation}\label{w1n}
h_{1,n}(z)-g_{1,n}(z)=\frac{z}{1-z}\quad {\rm and}\quad \omega(z)=\frac{g'_{1,n}(z)}{h'_{1,n}(z)}=z^{n} ~(n\in\mathbb{N}).
 \end{equation}
If $n=2m+1$ ($m\in\IN$), then we have
\begin{equation*}
\begin{split}
f_{1,n}(z)&=\RE\left \{\frac{1}{n}\bigg(\frac{-z}{1-z}+\frac{z(2-z)}{(1-z)^2}-\frac{n^2-1}{6}\log(1-z) \right .\\
&\qquad \left . +\frac{1}{2}\sum_{k=1}^{(n-1)/2}\csc^{2}\frac{k\pi}{n}\log\left(1-2z\cos\frac{2k\pi}{n}+z^2\right)\bigg)\right \}
+i\IM\left\{\frac{z}{1-z}\right\}.
\end{split}
\end{equation*}
If $n=2m$ ($m\in\IN$), then $f_{1,n}(\mathbb{D})$  lifts to the minimal surfaces $\mathbf{X}_{1,n}(u,v)=(u,v,F(u,v))$, where
\begin{equation}\label{Xeven1n}
\begin{aligned}
u&=\RE\bigg\{\frac{1}{n}\bigg(\frac{-z}{1-z}+\frac{z(2-z)}{(1-z)^2}-\frac{n^2-1}{6}\log(1-z)+\frac{1}{2}\log(1+z)\\
&\qquad \left . +\frac{1}{2}\sum_{k=1}^{(n/2)-1}\csc^{2}\frac{k\pi}{n}\log\left(1-2z\cos\frac{2k\pi}{n}+z^2\right)\bigg)\right\},\\
v&=\IM\left\{\frac{z}{1-z}\right\},
\end{aligned}
\end{equation}
and
\begin{equation*}
\begin{aligned}
F(u,v)&=\IM\left \{\frac{1}{n}\bigg(\frac{z}{1-z}+\frac{z(2-z)}{(1-z)^2}+\frac{n^2+1}{6}\log(1-z)+\frac{(-1)^{n/2}}{2}\log(1+z) \right .\\
&\qquad \left . +\frac{1}{2}\sum_{k=1}^{(n/2)-1}(-1)^{k}\csc^{2}\frac{k\pi}{n}\log\left(1-2z\cos\frac{2k\pi}{n}+z^2\right)\bigg)\right \}.
\end{aligned}
\end{equation*}
\end{theorem}
\begin{proof} By assumption and \eqref{w1n}, we have
$$h'_{1,n}(z)-g'_{1,n}(z)=\frac{1}{(1-z)^2}\quad {\rm and}\quad g'_{1,n}(z)=z^{n}h'_{1,n}(z).
$$
Solving these two equations, we obtain
\begin{equation}\label{eq1n}
h'_{1,n}(z)=\frac{1}{(1-z)^2(1-z^{n})}
\end{equation}
which has a pole of order $3$ at $z=1$ and simple poles at the other $n$-th roots of unity when $n=2m+1~(m\in\mathbb{N})$, and has a pole of order $3$ at
$z=1$ and simple poles at $z=-1$ and at the other $n$-th roots of unity when $n=2m~(m\in\mathbb{N})$, respectively. In view of these observations,  we can decompose $h'_{1,n}(z)$ into partial fraction.
After tedious but straightforward partial fraction expression obtained for $h'_{1,n}(z)$, for the
case of odd values of $n$, we have
\begin{equation*}
\begin{split}
h'_{1,n}(z)&=\frac{1}{(1-z)^2(1-z^{n})}\\
&=\frac{\kappa_{1}}{1-z}+\frac{\kappa_{2}}{(1-z)^2}+\frac{\kappa_{3}}{(1-z)^3}
+\sum_{k=1}^{(n-1)/2}\frac{\alpha_{k}}{1-ze^{-i\frac{2k\pi}{n}}}+\sum_{k=1}^{(n-1)/2}\frac{\beta_{k}}{1-ze^{i\frac{2k\pi}{n}}}
\end{split}
\end{equation*}
and the constants may be computed by using the residue theorem: $$\kappa_{1}=\frac{n^2-1}{12n},\quad\kappa_{2}=\frac{n-1}{2n},\quad\kappa_{3}=\frac{1}{n},
\quad \alpha_{k}=\frac{1}{n(1-e^{i\frac{2k\pi}{n}})^2},\quad\beta_{k}=\frac{1}{n(1-e^{-i\frac{2k\pi}{n}})^2}.
$$
By integrating the previous expression we arrive at the expression for the case of odd values of $n$:
\begin{equation*}
\begin{split}
h_{1,n}(z)&=\frac{n-1}{2n}\frac{z}{1-z}+\frac{1}{2n}\frac{z(2-z)}{(1-z)^2}-\frac{n^2-1}{12n}\log(1-z)\\
&\qquad  +\frac{1}{4n}\sum_{k=1}^{(n-1)/2}\csc^{2}\frac{k\pi}{n}\log\left(1-2z\cos\frac{2k\pi}{n}+z^2\right).
\end{split}
\end{equation*}

In the case of even values of $n$ one  can  write \eqref{eq1n} into partial fraction as
\begin{equation*}
\begin{split}
h'_{1,n}(z)&=\frac{1}{(1-z)^2(1-z^{n})}\\
&=\frac{\lambda_{1}}{1-z}+\frac{\lambda_{2}}{(1-z)^2}+\frac{\lambda_{3}}{(1-z)^3}+\frac{\lambda_{4}}{1+z}
+\sum_{k=1}^{(n/2)-1}\frac{\gamma_{k}}{1-ze^{-i\frac{2k\pi}{n}}}+\sum_{k=1}^{(n/2)-1}\frac{\delta_{k}}{1-ze^{i\frac{2k\pi}{n}}}.
\end{split}
\end{equation*}
Again, using the residue theorem or otherwise, we find that
\begin{equation*}
\begin{split}
\lambda_{1}&=\frac{n^2-1}{12n},\quad\lambda_{2}=\frac{n-1}{2n},\quad\lambda_{3}=\frac{1}{n},\quad\lambda_{4}=\frac{1}{4n},\\
\gamma_{k}&=\frac{1}{n(1-e^{i\frac{2k\pi}{n}})^2},\quad\delta_{k}=\frac{1}{n(1-e^{-i\frac{2k\pi}{n}})^2}
\end{split}
\end{equation*}
and we arrive at the expression
\begin{equation*}
\begin{split}
h_{1,n}(z)&=\frac{n-1}{2n}\frac{z}{1-z}+\frac{1}{2n}\frac{z(2-z)}{(1-z)^2}-\frac{n^2-1}{12n}\log(1-z)+\frac{1}{4n}\log(1+z)\\
&\qquad  + \frac{1}{4n}\sum_{k=1}^{(n/2)-1}\csc^{2}\frac{k\pi}{n}\log\left(1-2z\cos\frac{2k\pi}{n}+z^2\right).
\end{split}
\end{equation*}
In both cases, the corresponding function $g_{1,n}(z)$ may be computed using the first relation \eqref{w1n} and the above two cases. Finally,
the desired harmonic mapping $f_{1,n}(z)$ follows from writing $f_{1,n}(z)$ as
$$f_{1,n}(z)=u+iv= \RE\{h_{1,n}(z)+g_{1,n}(z)\}+i\IM\{h_{1,n}(z)-g_{1,n}(z)\}.
$$
Consequently, $f_{1,n}(z)$  for the case of odd values of $n$ is given by
\begin{equation*}
\begin{split}
f_{1,n}(z)&=\RE\bigg\{\frac{1}{n}\bigg(\frac{-z}{1-z}+\frac{z(2-z)}{(1-z)^2}-\frac{n^2-1}{6}\log(1-z)\\
&\qquad  +\frac{1}{2}\sum_{k=1}^{(n-1)/2}\csc^{2}\frac{k\pi}{n}\log\left(1-2z\cos\frac{2k\pi}{n}+z^2\right)\bigg)\bigg\}
+i\IM\left\{\frac{z}{1-z}\right\},
\end{split}
\end{equation*}
and  $f_{1,n}(z)$ for the case of even values of $n$ takes the form $f_{1,n}=u+iv$, where $u$ is given by \eqref{Xeven1n} and
$$v(x,y)=\IM\left\{\frac{z}{1-z}\right\}.
$$
In view of Theorem~\ref{thmB}, $f_{1,n}(\mathbb{D})$ for the case of even $n$  lifts to the minimal surfaces $\mathbf{X}_{1,n}(u,v)=(u,v,F(u,v))$, where
$u$ is given by \eqref{Xeven1n}, $v=v(x,y)=\IM\left\{z/(1-z)\right\}$
and
\begin{equation*}
\begin{split}
F(u,v)&=2\IM\left\{\int_{0}^{z}\sqrt{\omega_{n}(\zeta)}h'_{1,n}(\zeta)d\zeta\right\}=2\IM\left\{\int_{0}^{z}\frac{\zeta^{n/2}}{(1-\zeta)^2(1-\zeta^{n})}d\zeta\right\}\\
&=\IM\bigg\{-\frac{1}{n}\frac{z}{(1-z)}+\frac{1}{n}\frac{z(2-z)}{(1-z)^2}+\frac{(-1)^{n/2}}{2n}\log(1+z)+\frac{n^2+1}{6n}\log(1-z)\\
&\qquad +\frac{1}{2n}\sum_{k=1}^{(n/2)-1}(-1)^{k}\csc^{2}\frac{k\pi}{n}\log\left(1-2z\cos\frac{2k\pi}{n}+z^2\right)\bigg\}.
\end{split}
\end{equation*}
The proof is complete.
\end{proof}

\begin{remark}
If $\omega(z)=z$ in Theorem~\ref{thmc1}, then the expression for $f_{1,1}(z)$ simplifies to
$$f_{1,1}(z)= \RE\{k(z)\}+i\IM\{l(z)\},
$$
where $k(z)=z/(1-z)^2$ and $l(z)=z/(1-z)$. Here we may compare $f_{1,1}(z)$ with the well known harmonic half-plane mapping $L(z)$ defined by
$$L(z)=\RE\{l(z)\}+i\IM\{k(z)\}.
$$
For $\omega(z)=z^2$, the expression of $f_{1,2}(z)$ is given by
$$f_{1,2}(z)=\RE\left\{\frac{1}{2}\frac{z}{(1-z)^2}+\frac{1}{4}\log\left(\frac{1+z}{1-z}\right)\right\}+i\IM\left\{\frac{z}{1-z}\right\}.
$$
In this case, see also~\cite[Theorem 3]{Dorff2014AAA}.  In Figure~\ref{f1n}, we have drawn the harmonic mappings $f_{1,n}(z)$ of the unit disk $\mathbb{D}$ onto wave planes with $c=1$ and $n=3,4,5,6$. In Figure~\ref{X1n}, we have drawn the minimal surfaces of the harmonic mappings $f_{1,n}(z)$ onto wave planes and with dilatation $\omega(z)=z^{n}$ for $n=4,6,8,10,12,14$.
\begin{figure}[!h]
\centering
\subfigure[$n=3$]
{\begin{minipage}[b]{0.45\textwidth}
\includegraphics[height=1.6in,width=2.4in,keepaspectratio]{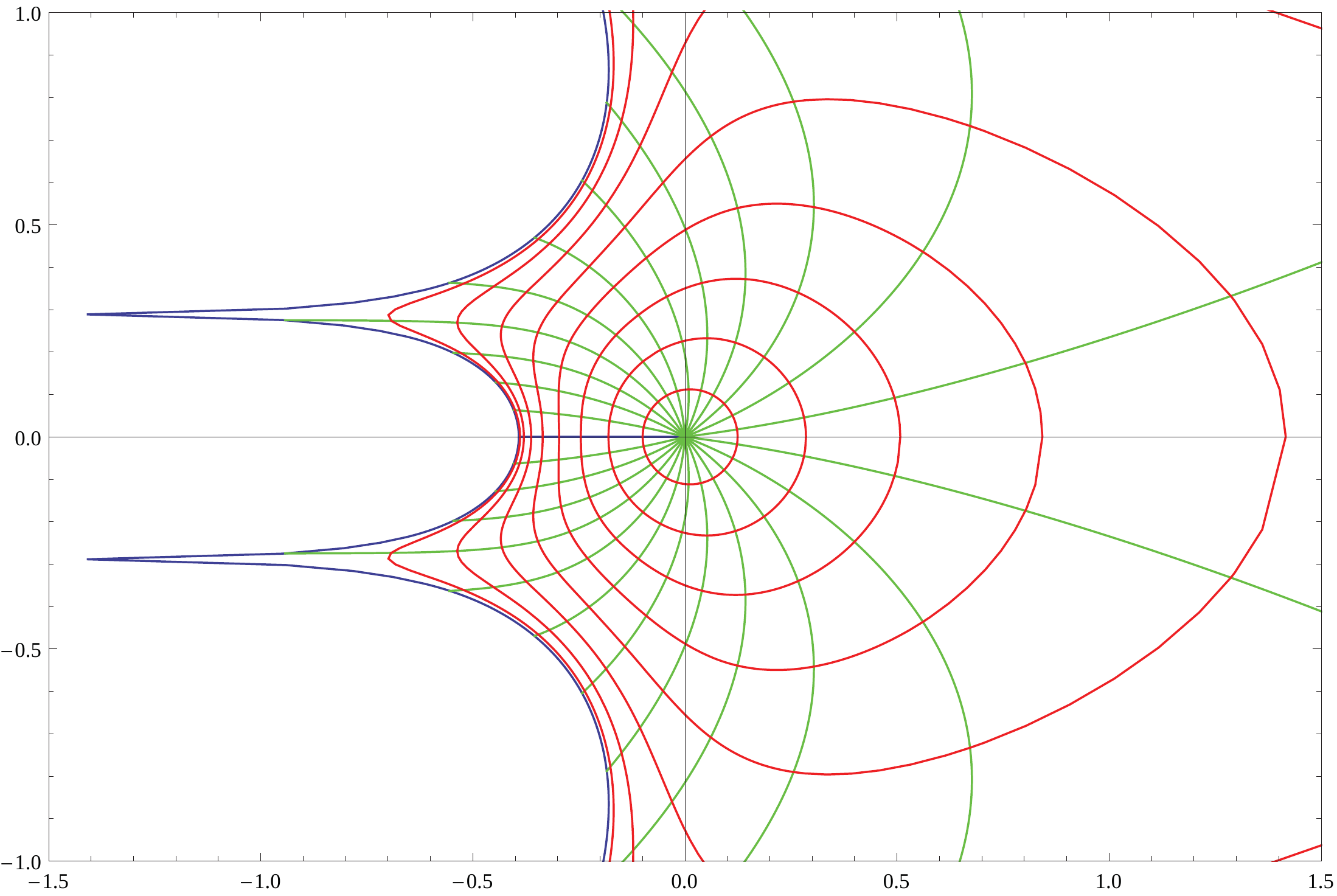}
\end{minipage}}
\subfigure[$n=4$]
{\begin{minipage}[b]{0.45\textwidth}
\includegraphics[height=1.6in,width=2.4in,keepaspectratio]{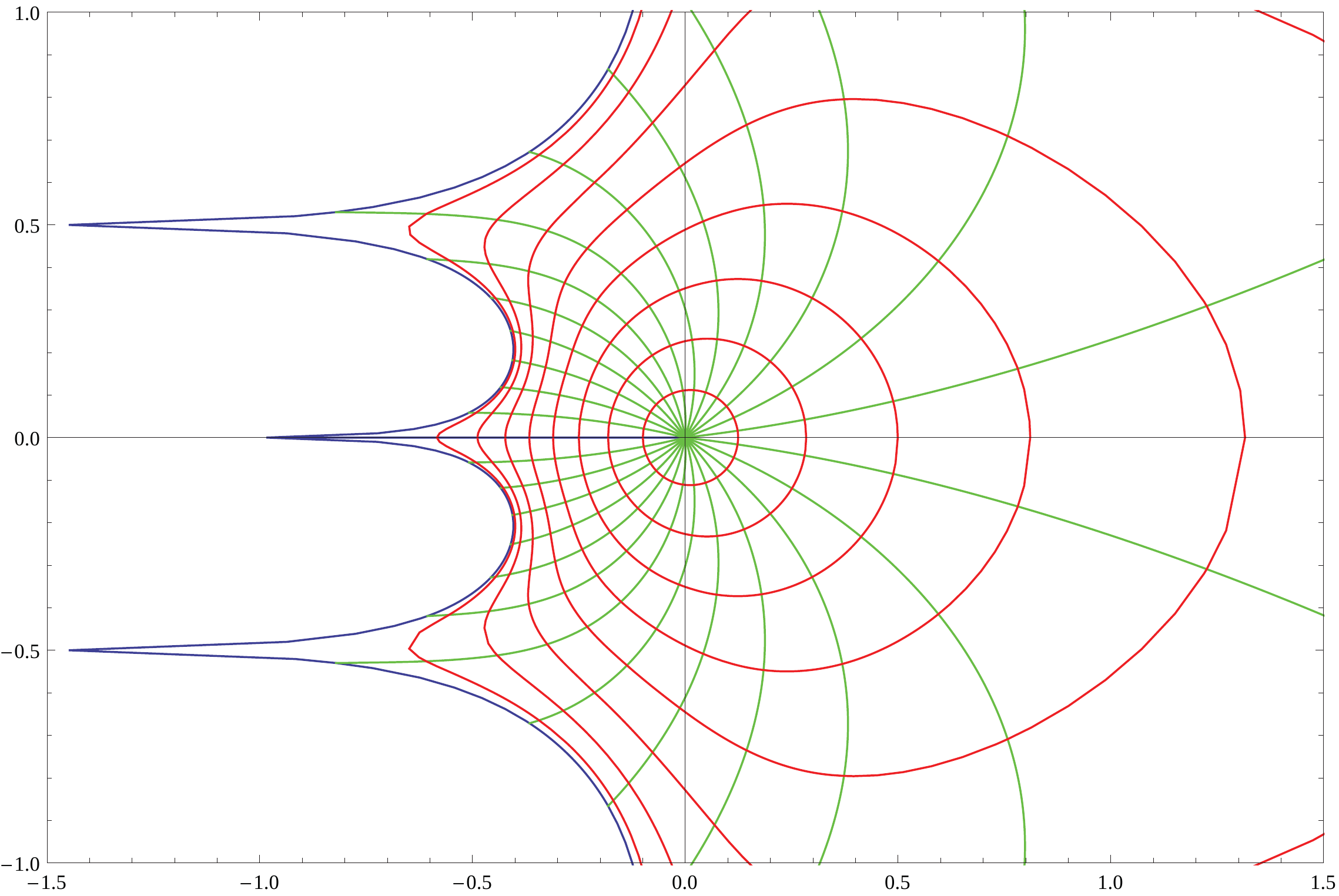}
\end{minipage}}
\subfigure[$n=5$]
{\begin{minipage}[b]{0.45\textwidth}
\includegraphics[height=1.6in,width=2.4in,keepaspectratio]{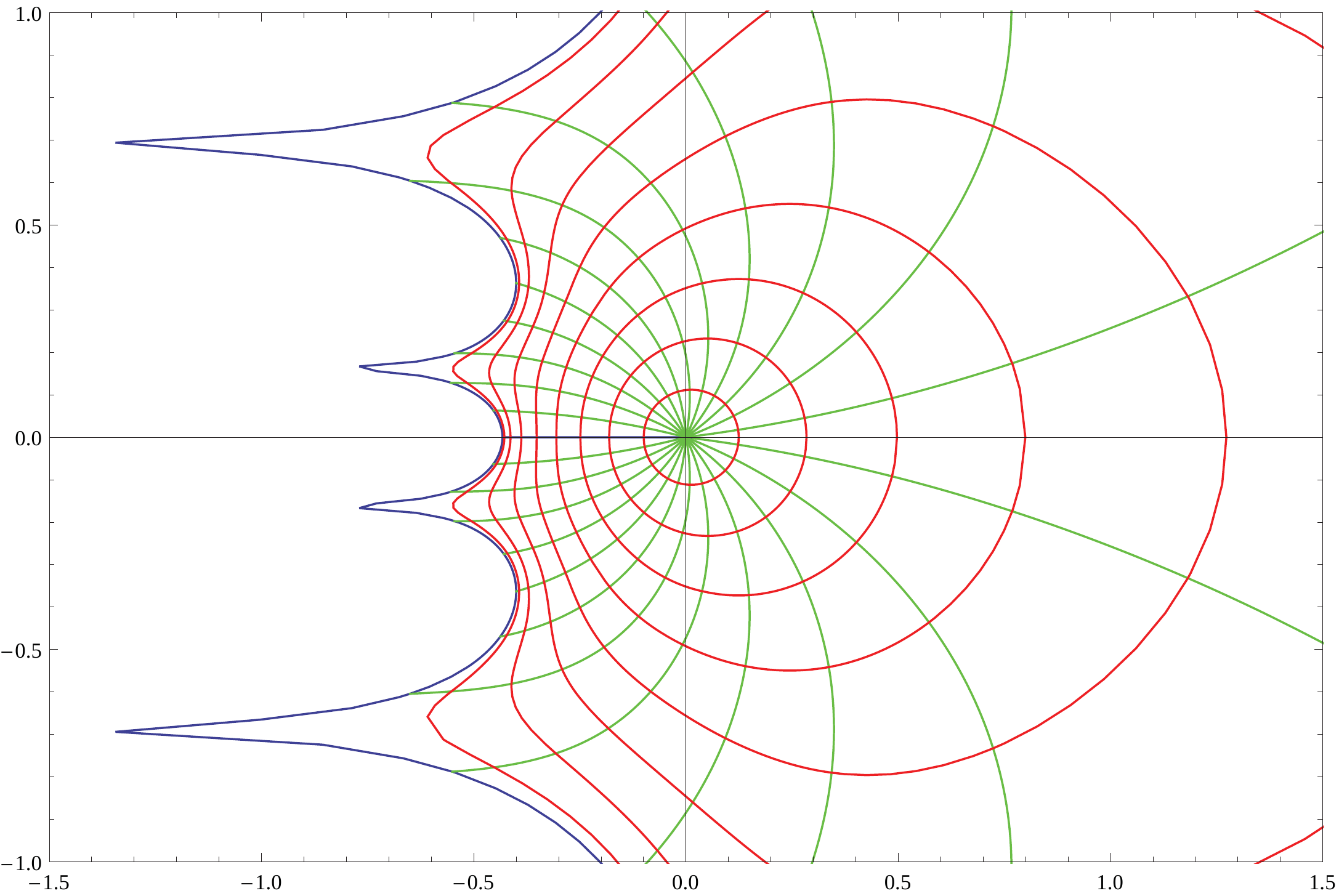}
\end{minipage}}
\subfigure[$n=6$]
{\begin{minipage}[b]{0.45\textwidth}
\includegraphics[height=1.6in,width=2.4in,keepaspectratio]{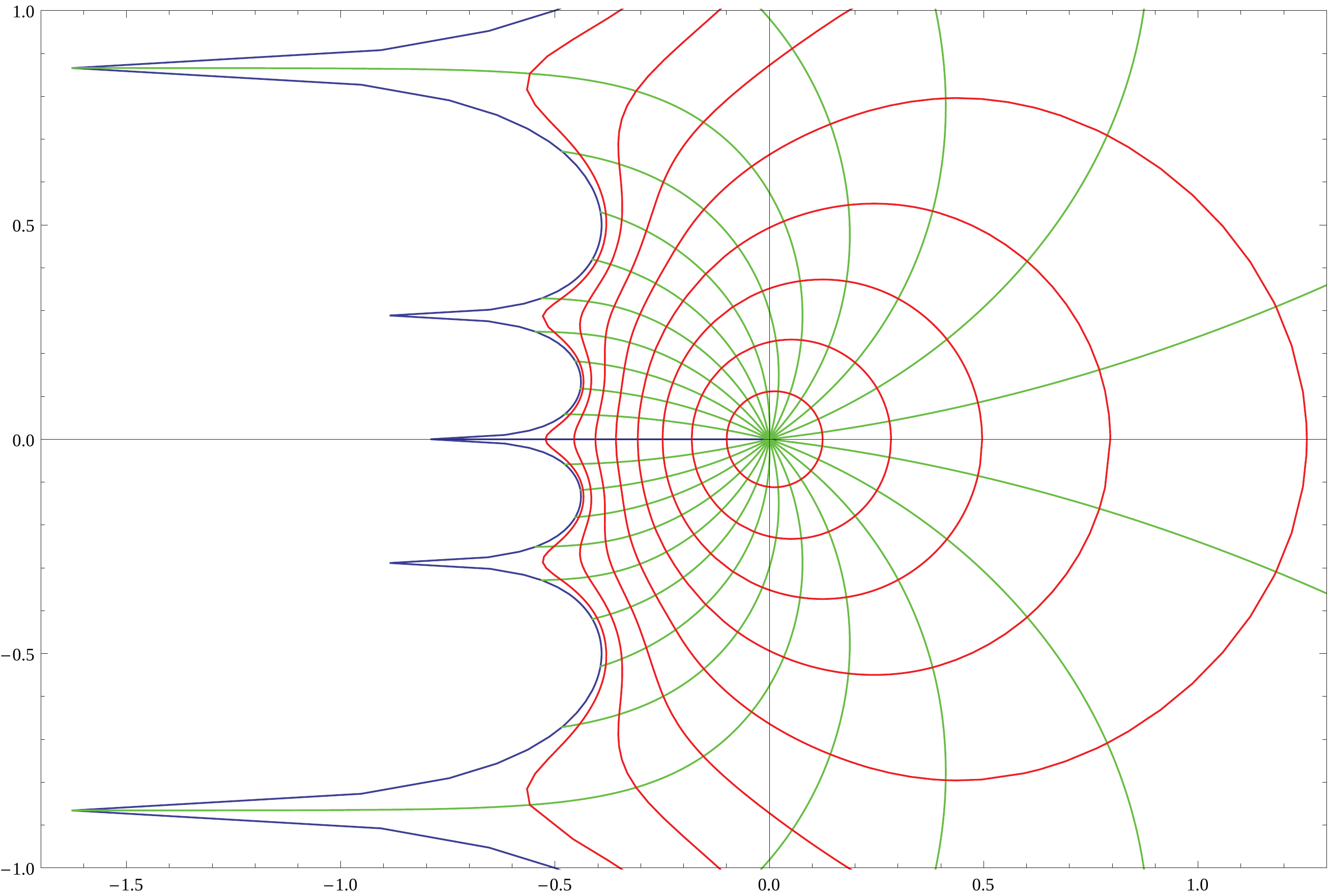}
\end{minipage}}
\caption{Wave planes of $f_{1,n}(\mathbb{D})$ for various values of $n=3,4,5,6$. }\label{f1n}
\end{figure}

\begin{figure}[!h]
\centering
\subfigure[$c=1,n=4$]
{\begin{minipage}[b]{0.45\textwidth}
\includegraphics[height=2.4in,width=2.4in,keepaspectratio]{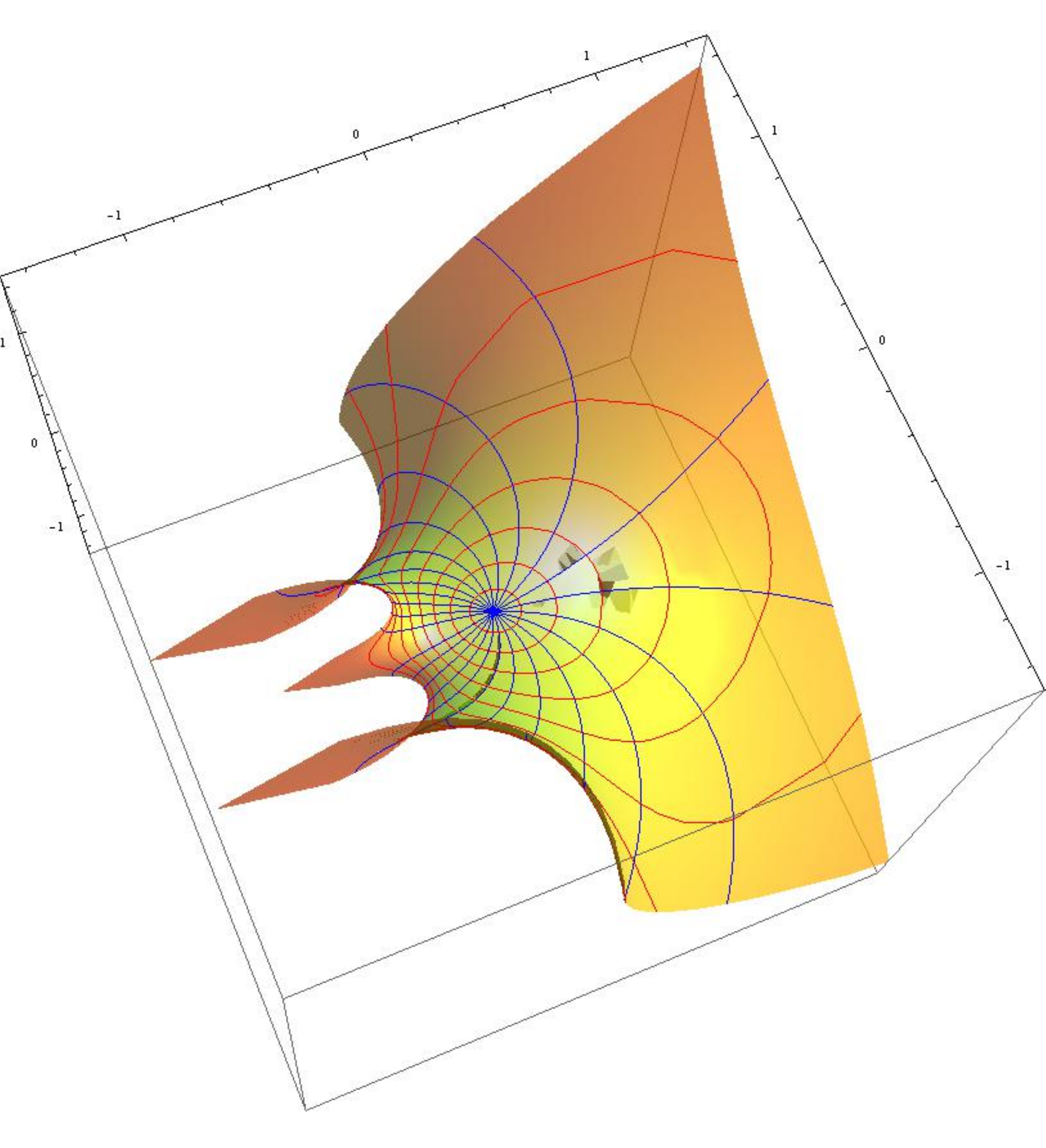}
\end{minipage}}
\subfigure[$c=1,n=6$]
{\begin{minipage}[b]{0.45\textwidth}
\includegraphics[height=2.4in,width=2.4in,keepaspectratio]{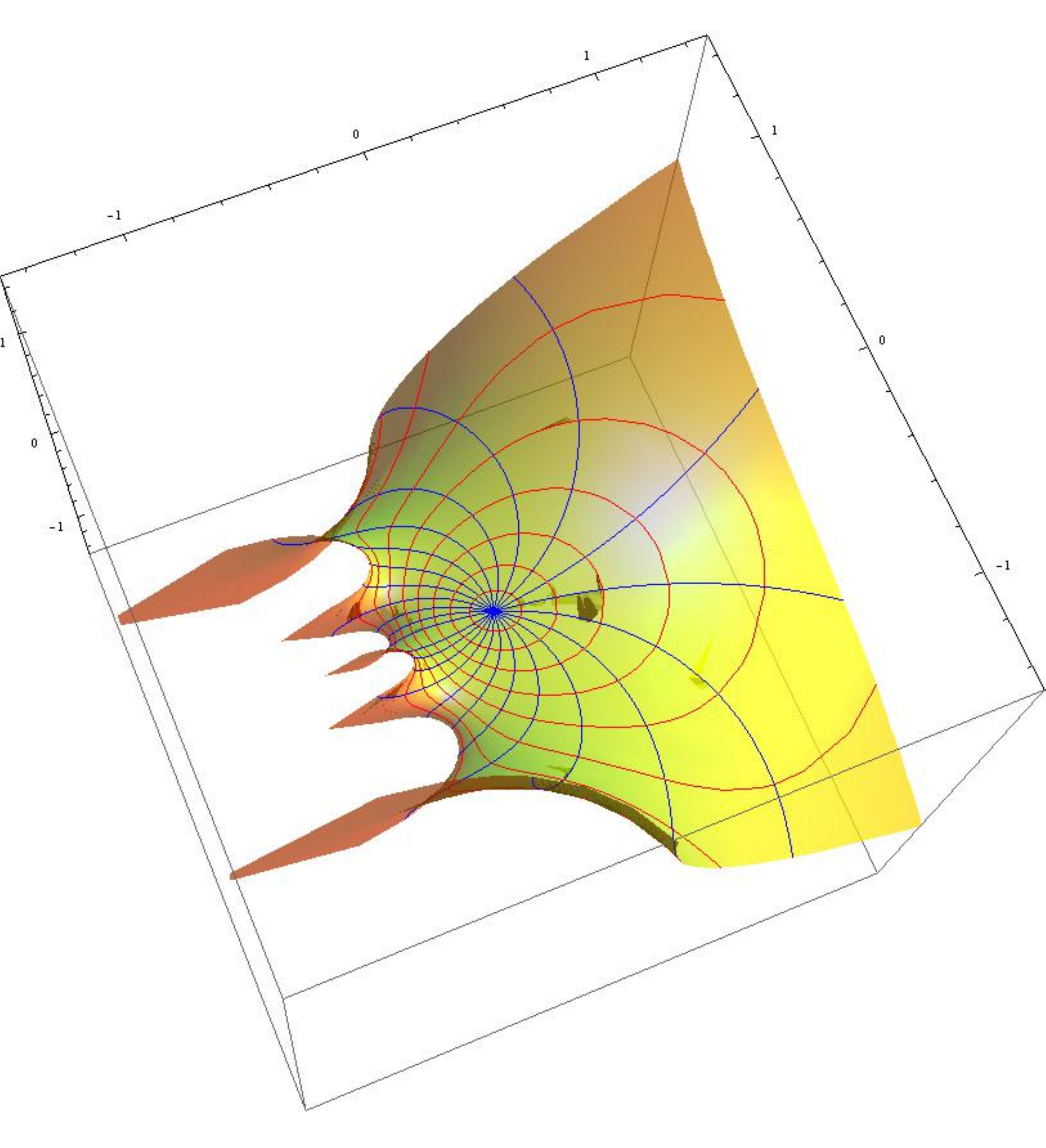}
\end{minipage}}
\subfigure[$c=1,n=8$]
{\begin{minipage}[b]{0.45\textwidth}
\includegraphics[height=2.4in,width=2.4in,keepaspectratio]{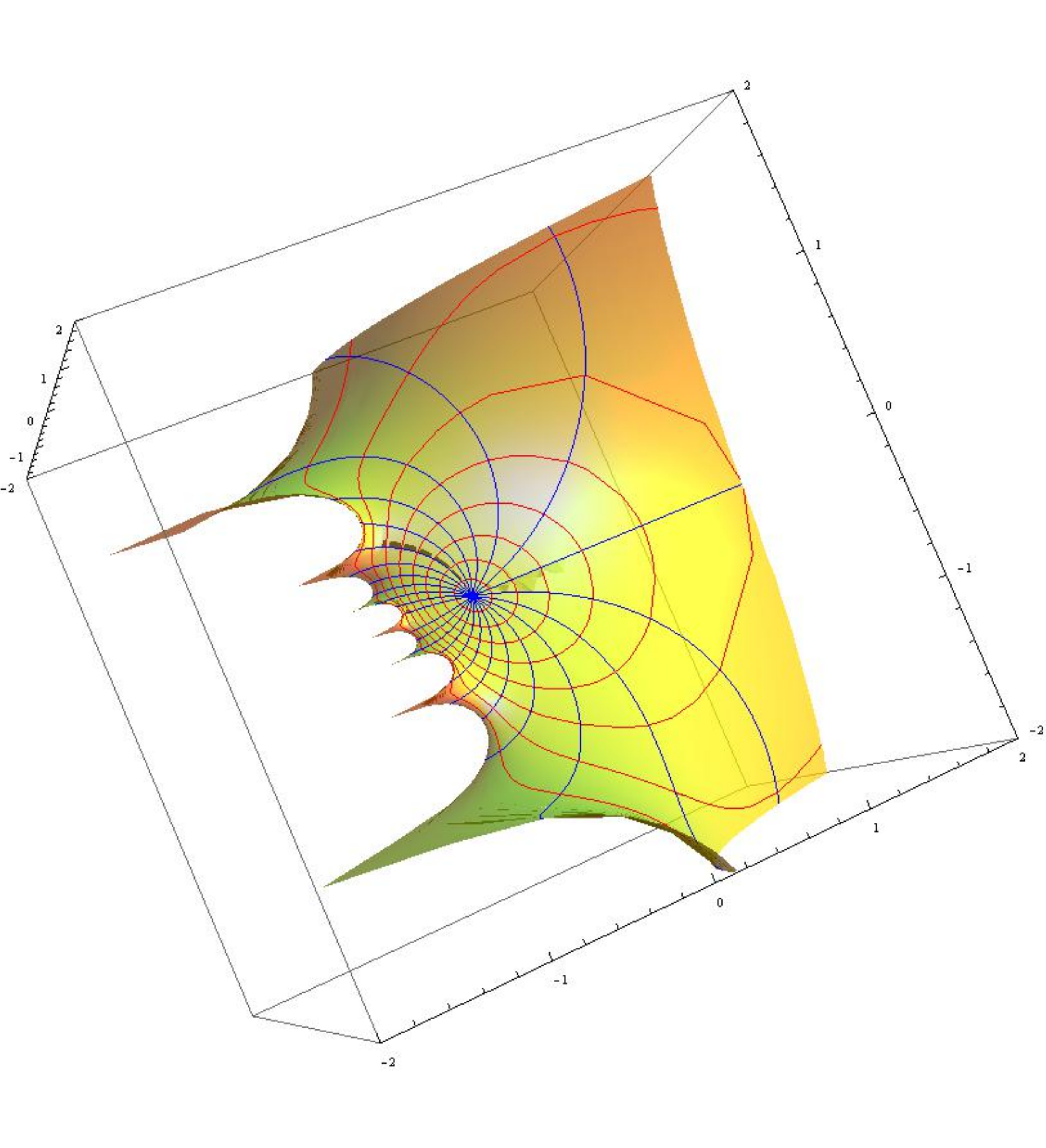}
\end{minipage}}
\subfigure[$c=1,n=10$]
{\begin{minipage}[b]{0.45\textwidth}
\includegraphics[height=2.4in,width=2.4in,keepaspectratio]{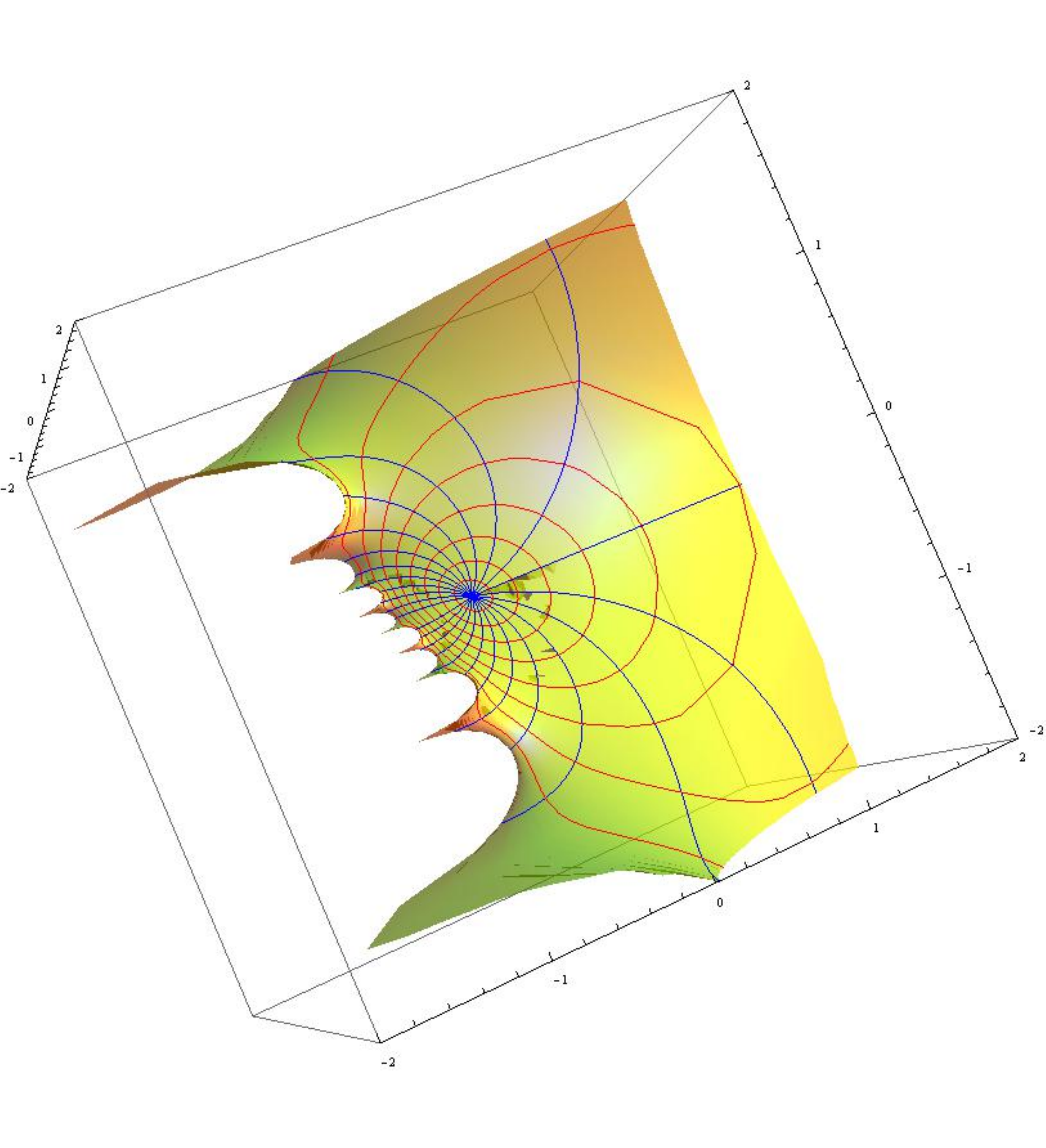}
\end{minipage}}
\subfigure[$c=1,n=12$]
{\begin{minipage}[b]{0.45\textwidth}
\includegraphics[height=2.4in,width=2.4in,keepaspectratio]{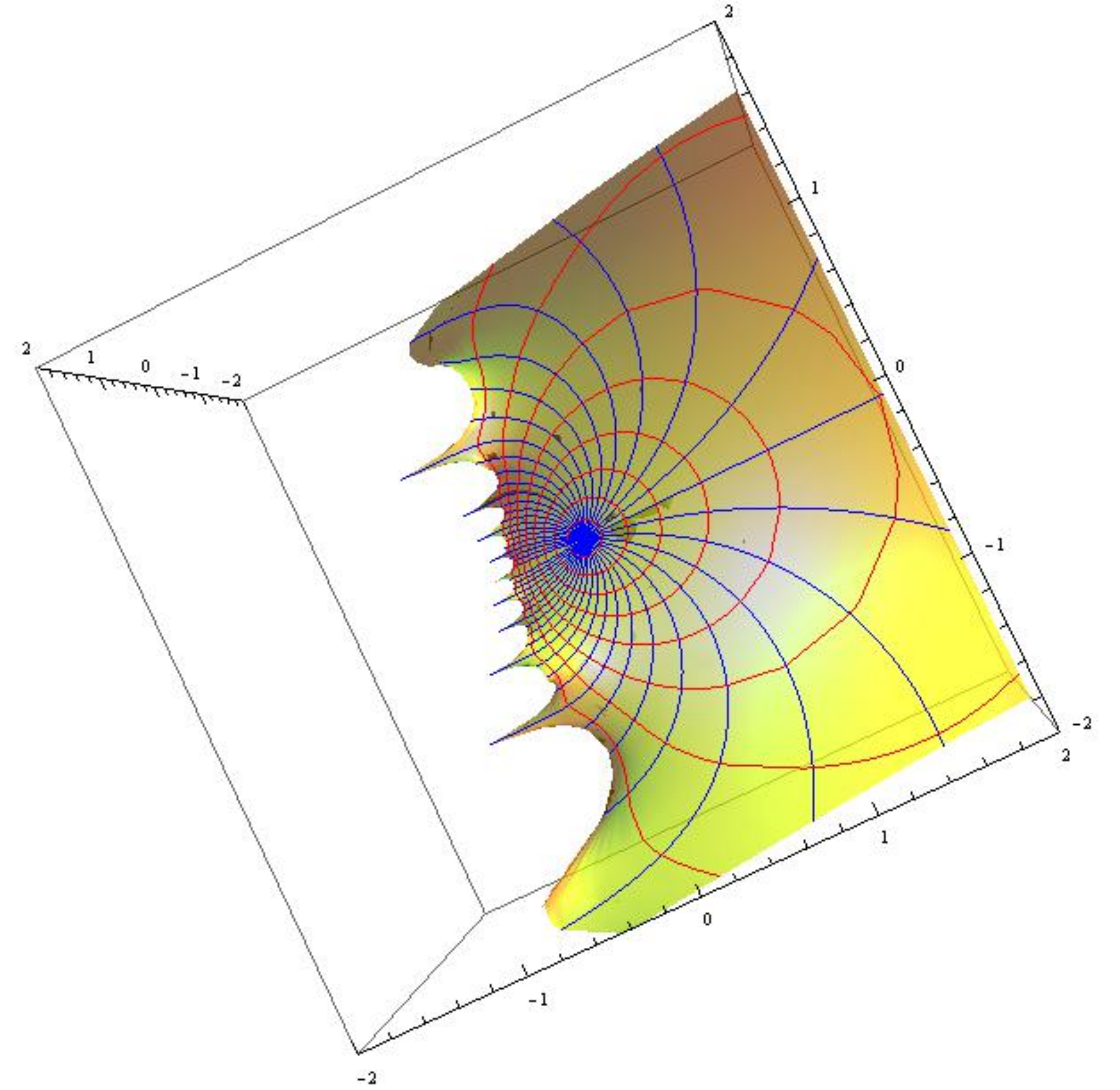}
\end{minipage}}
\subfigure[$c=1,n=14$]
{\begin{minipage}[b]{0.45\textwidth}
\includegraphics[height=2.4in,width=2.4in,keepaspectratio]{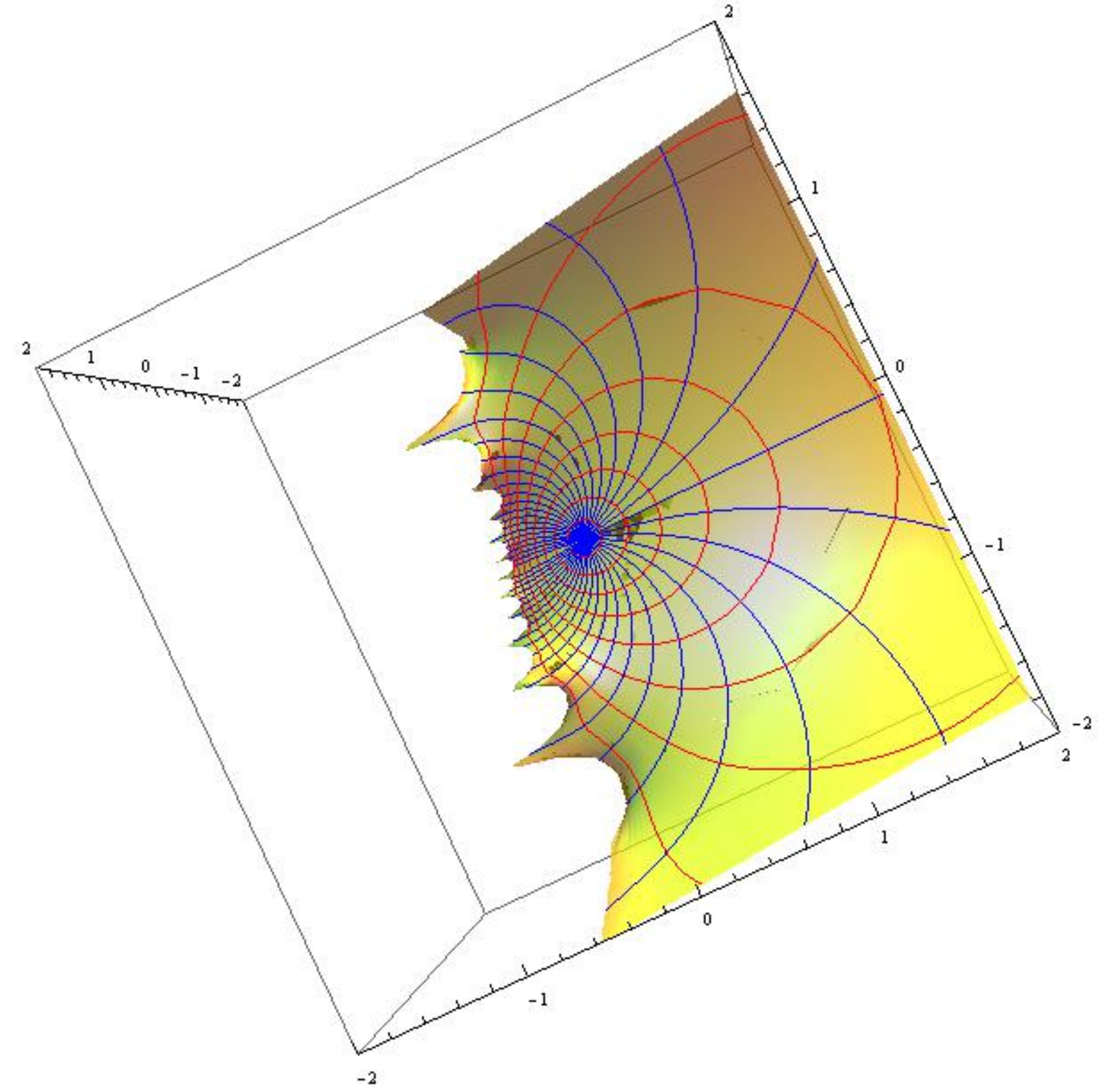}
\end{minipage}}
\caption{$f_{1,n}(\mathbb{D})$ lift to the minimal surfaces for various values of $c=1,n=4,6,8,10,12,14$.}\label{X1n}
\end{figure}
\end{remark}

\begin{theorem}\label{thmc2}
For $n\in \IN$, let $f_{2,n}=h_{2,n}+\overline{g_{2,n}}\in \mathcal{S}_{H}^{0}$ such that
\begin{equation}\label{w2n}
h_{2,n}(z)-g_{2,n}(z)=\frac{z}{(1-z)^2}\quad {\rm and}\quad \omega_{n}(z)=\frac{g'_{0,n}(z)}{h'_{0,n}(z)}=z^{n}.
\end{equation}
If $n=2m+1$  ($m\in\IN$), then $f_{2,n}(z)$ is given by
\begin{equation*}
\begin{split}
f_{2,n}(z)&=\RE\left\{\frac{-z}{(1-z)^2}+\frac{(n-1)(n-2)}{3n}\frac{z}{1-z}+\frac{n-2}{n}\frac{z(2-z)}{(1-z)^2}
+\frac{4}{3n}\frac{z \left(z^2-3 z+3\right)}{(1-z)^3}\right .\\
&\qquad  \left . +\frac{i}{2n}\sum_{k=1}^{(n-1)/2}\cot \frac{k\pi}{n}\csc^2\frac{k\pi}{n}\log\left(\frac{1-z e^{-i\frac{ 2k\pi}{n}}}{1-z e^{i\frac{2k\pi}{n}}}\right)
\right \}+i\IM\left\{\frac{z}{(1-z)^2}\right\}.
\end{split}
\end{equation*}
If $n=2m$  ($m\in\IN$), then $f_{2,n}(\mathbb{D})$  lifts to the minimal surfaces $\mathbf{X}_{2,n}(u,v)=(u,v,F(u,v))$, where
\begin{eqnarray}
u&=&\RE\left\{\frac{-z}{(1-z)^2}+\frac{(n-1)(n-2)}{3n}\frac{z}{1-z}+\frac{n-2}{n}\frac{z(2-z)}{(1-z)^2}
+\frac{4}{3n}\frac{z \left(z^2-3 z+3\right)}{(1-z)^3}\right . \nonumber\\
&&\qquad  \left. +\frac{i}{2n}\sum_{k=1}^{(n/2)-1}\cot \frac{k\pi}{n}\csc^2\frac{k\pi}{n}\log\left(\frac{1-z e^{-i\frac{ 2k\pi}{n}}}{1-z e^{i\frac{2k\pi}{n}}}\right)
\right \}, \label{EqX2n} \\
v&=&\IM\left\{\frac{z}{(1-z)^2}\right\}, \nonumber
\end{eqnarray}
and
\begin{equation*}\label{EqX2n-a}
\begin{aligned}
F(u,v)&=\IM\left\{\frac{4-n^{2}}{6n}\frac{z}{1-z}-\frac{2}{n}\frac{z(2-z)}{(1-z)^2}
+\frac{4}{3n}\frac{z \left(z^2-3 z+3\right)}{(1-z)^3}\right .\\
&\qquad  \left . +\frac{i}{2n}\sum_{k=1}^{(n/2)-1}(-1)^{k}\cot \frac{k\pi}{n}\csc^2\frac{k\pi}{n}
\log\left(\frac{1-z e^{-i \frac{2k\pi}{n}}}{1-z e^{i \frac{2k\pi}{n}}}\right)\right\}.
\end{aligned}
\end{equation*}
\end{theorem}
\begin{proof}
By using~\eqref{w2n}, we have
$$h'_{2,n}(z)-g'_{2,n}(z)=\frac{1+z}{(1-z)^3}\quad {\rm and}\quad g'_{2,n}(z)=z^{n}h'_{2,n}(z).
$$
Solving these two equations, we obtain
\begin{equation*}
h'_{2,n}(z)=\frac{1+z}{(1-z)^3(1-z^{n})}.
\end{equation*}
Now, consider the case $n=2m+1~ (m\in\mathbb{N})$. Since $h'_{2,n}(z)$ has a pole of order $4$ at $z=1$ and simple
poles at other $n$-th roots of unity,   $h'_{2,n}(z)$ can be represented using partial fraction as follows:
\begin{equation*}
\begin{split}
h'_{2,n}(z)&=\frac{\lambda_{1}}{1-z}+\frac{\lambda_{2}}{(1-z)^2}+\frac{\lambda_{3}}{(1-z)^3}+\frac{\lambda_{4}}{(1-z)^4}\\
&\qquad  +\sum_{k=1}^{(n-1)/2}\frac{A_{k}}{1-ze^{-i\frac{2k\pi}{n}}}+\sum_{k=1}^{(n-1)/2}\frac{B_{k}}{1-ze^{i\frac{2k\pi}{n}}}.
\end{split}
\end{equation*}
By using the residue theorem or otherwise, one can easily see that
\begin{equation*}
\begin{split}
\lambda_{1}=0, \quad\lambda_{2}=\frac{(n-1)(n-2)}{6 n},
\quad \lambda_{3}=\frac{n-2}{n}, \quad \lambda_{4}=\frac{2}{n},
\end{split}
 \end{equation*}
and
\begin{equation*}
\begin{split}
A_{k}&=\frac{1}{n}\frac{1+e^{i\frac{2k\pi}{n}}}{(1-e^{i\frac{2k\pi}{n}})^3},\quad
B_{k}=\frac{1}{n}\frac{1+e^{-i\frac{2k\pi}{n}}}{(1-e^{-i\frac{2k\pi}{n}})^3},
\end{split}
\end{equation*}
Using these values, we arrive at the expression
\begin{equation*}
\begin{split}
h'_{2,n}(z)&=\frac{(n-1)(n-2)}{6n}\frac{1}{(1-z)^2}+\frac{n-2}{n}\frac{1}{(1-z)^3}+\frac{2}{n}\frac{1}{(1-z)^4}\\
&\qquad +\frac{1}{n}\left(\sum_{k=1}^{(n-1)/2}\frac{(1+e^{i\frac{2k\pi}{n}})}
{(1-e^{i\frac{2k\pi}{n}})^3}\frac{1}{(1-ze^{-i\frac{2k\pi}{n}})}
+\sum_{k=1}^{(n-1)/2}\frac{(1+e^{-i\frac{2k\pi}{n}})}{(1-e^{-i\frac{2k\pi}{n}})^3}\frac{1}{(1-ze^{i\frac{2k\pi}{n}})}\right).
\end{split}
\end{equation*}
Integration from $0$ to $z$ gives
\begin{equation*}
\begin{split}
h_{2,n}(z)&=\frac{(n-1)(n-2)}{6n}\frac{z}{1-z}+\frac{n-2}{2n}\frac{z(2-z)}{(1-z)^2}+\frac{2}{3n}\frac{z \left(z^2-3 z+3\right)}{(1-z)^3}\\
&\qquad +\frac{i}{4n}\sum_{k=1}^{(n-1)/2}\cot \frac{k\pi}{n}\csc^2\frac{k\pi}{n}
\log\left(\frac{1-z e^{-i\frac{ 2k\pi}{n}}}{1-z e^{i\frac{2k\pi}{n}}}\right),
\end{split}
\end{equation*}
and, as a consequence of it, $g_{2,n}(z)$ can be written explicitly by using the first relation in \eqref{w2n}.
Then the desired harmonic mapping $f_{2,n}\in \mathcal{S}_{H}^{0}$ for odd values of $n$ is given by
\begin{equation*}
\begin{split}
f_{2,n}(z)&=\RE\left \{\frac{-z}{(1-z)^2}+\frac{(n-1)(n-2)}{3n}\frac{z}{1-z}+\frac{n-2}{n}\frac{z(2-z)}{(1-z)^2}
+\frac{4}{3n}\frac{z \left(z^2-3 z+3\right)}{(1-z)^3}\right . \\
&\qquad  \left .+\frac{i}{2n}\sum_{k=1}^{(n-1)/2}\cot \frac{k\pi}{n}\csc^2\frac{k\pi}{n}
\log\left(\frac{1-z e^{-i\frac{ 2k\pi}{n}}}{1-z e^{i\frac{2k\pi}{n}}}\right)\right\}+i\IM\left\{\frac{z}{(1-z)^2}\right\}.
\end{split}
\end{equation*}

For even values of $n$, we have $f_{2,n}(z)= u+iv$,
where $u$ is given by \eqref{EqX2n} and $v=\IM\left\{z/{(1-z)^2}\right\}$.
In view of Theorem~\ref{thmB}, for $n=2m ~ (m\in \mathbb{N})$, $f_{2,n}(\mathbb{D})$ lifts to the minimal surfaces
$\mathbf{X}_{2,n}(u,v)=(u,v,F(u,v))$, where $u$ is given by \eqref{EqX2n}, $v=\IM\left\{z/{(1-z)^2}\right\}$, and $F(u,v)$ is obtained from
\begin{equation*}
\begin{split}
F(u,v)&=2\IM\left\{\int_{0}^{z}\sqrt{\omega_{n}(\zeta)}h'_{1,n}(\zeta)d\zeta\right\}\\
&=2\IM\left\{\int_{0}^{z}\frac{(1+\zeta)\zeta^{n}}{(1-\zeta)^3(1-\zeta^{n})}d\zeta\right\}\\
&=\IM\left \{\frac{4-n^{2}}{6n}\frac{z}{1-z}-\frac{2}{n}\frac{z(2-z)}{(1-z)^2}+\frac{4}{3n}\frac{z \left(z^2-3 z+3\right)}{(1-z)^3}
\right .\\
&\qquad  \left . +\frac{i}{2n}\sum_{k=1}^{(n/2)-1}(-1)^{k}\cot \frac{k\pi}{n}\csc^2\frac{k\pi}{n}
\log\left(\frac{1-z e^{-i \frac{2k\pi}{n}}}{1-z e^{i \frac{2k\pi}{n}}}\right)\right \}.
\end{split}
\end{equation*}
The proof is complete.
\end{proof}

\begin{remark}
For $\omega(z)=z$ in Theorem~\ref{thmc2}, the resulting function $f_{2,1}(z)$ is the well-known harmonic Koebe function.
\end{remark}

In Figure~\ref{f2n}, we have illustrated the harmonic mappings $f_{2,n}(z)$ of the unit disk $\mathbb{D}$ onto split domains with
$c=2$ and $n=3,4,5,6$. In Figure~\ref{X2n}, we have drawn the minimal surfaces of the harmonic mappings $f_{2,n}(z)$ onto split domains and with dilatation $\omega(z)=z^{n}$ for $n=4,6,8,10,12,14$.
\begin{figure}[!h]
\centering
\subfigure[$n=3$]
{\begin{minipage}[b]{0.45\textwidth}
\includegraphics[height=1.5in,width=2.4in,keepaspectratio]{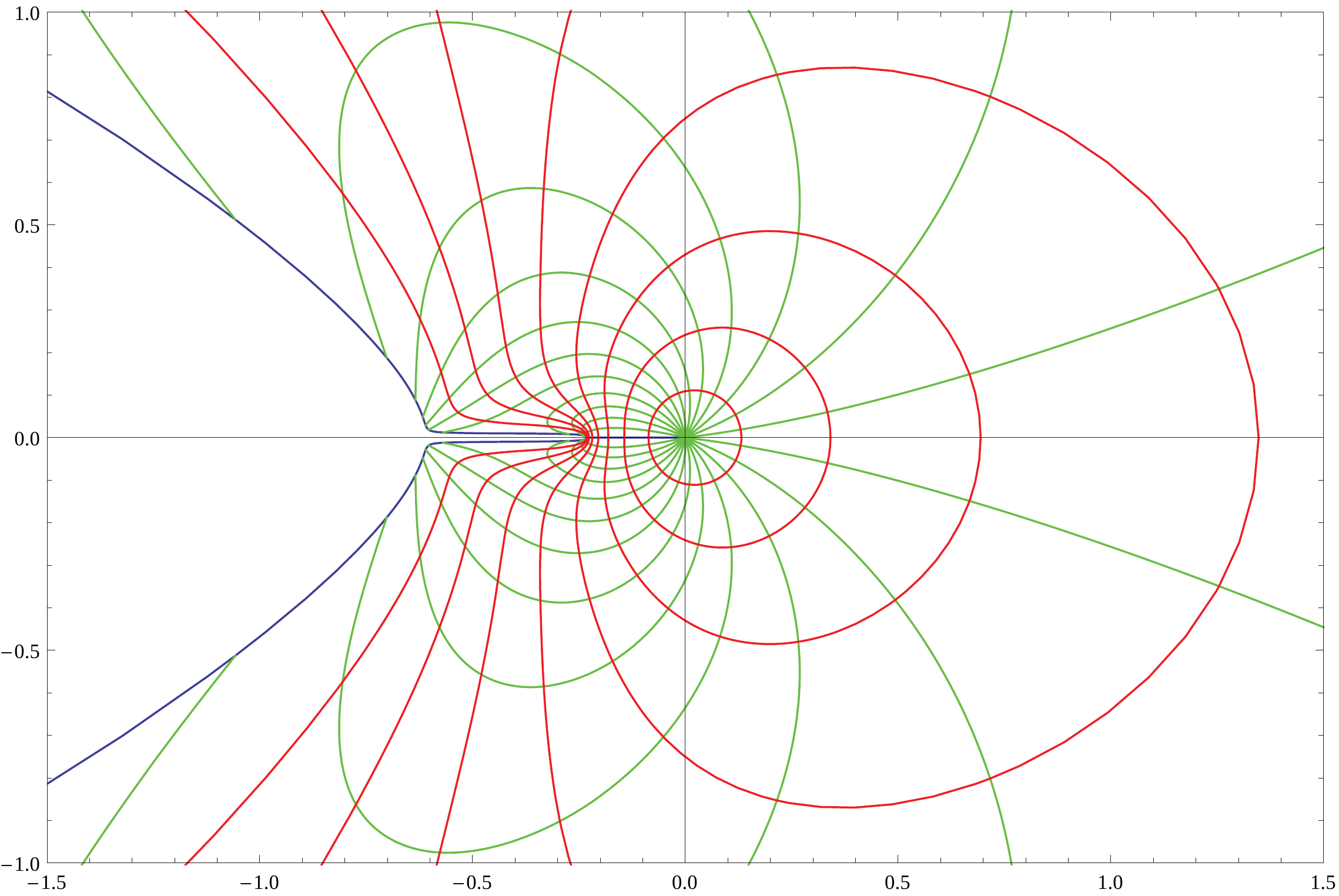}
\end{minipage}}
\subfigure[$n=4$]
{\begin{minipage}[b]{0.45\textwidth}
\includegraphics[height=1.5in,width=2.4in,keepaspectratio]{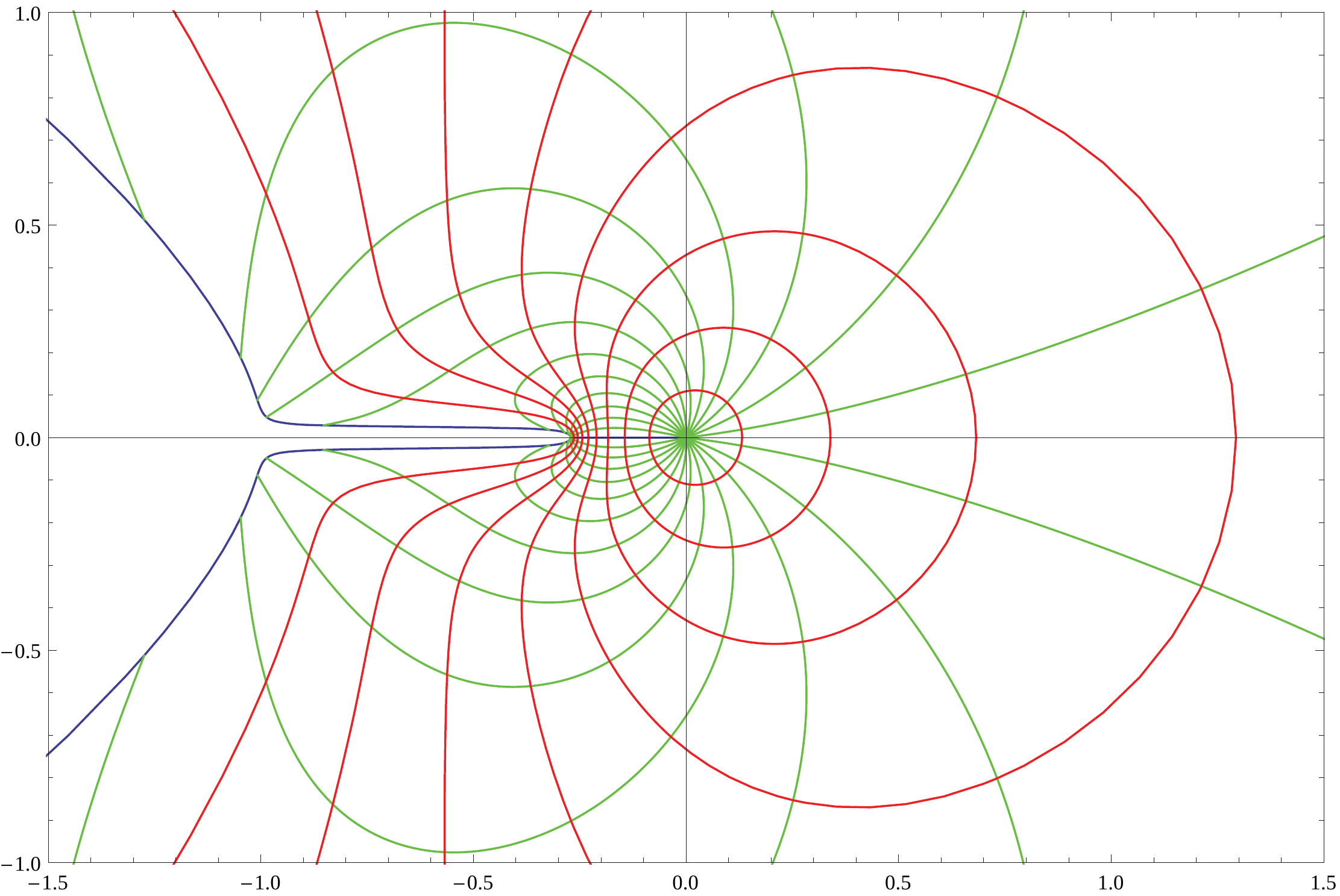}
\end{minipage}}
\subfigure[$n=5$]
{\begin{minipage}[b]{0.45\textwidth}
\includegraphics[height=1.5in,width=2.4in,keepaspectratio]{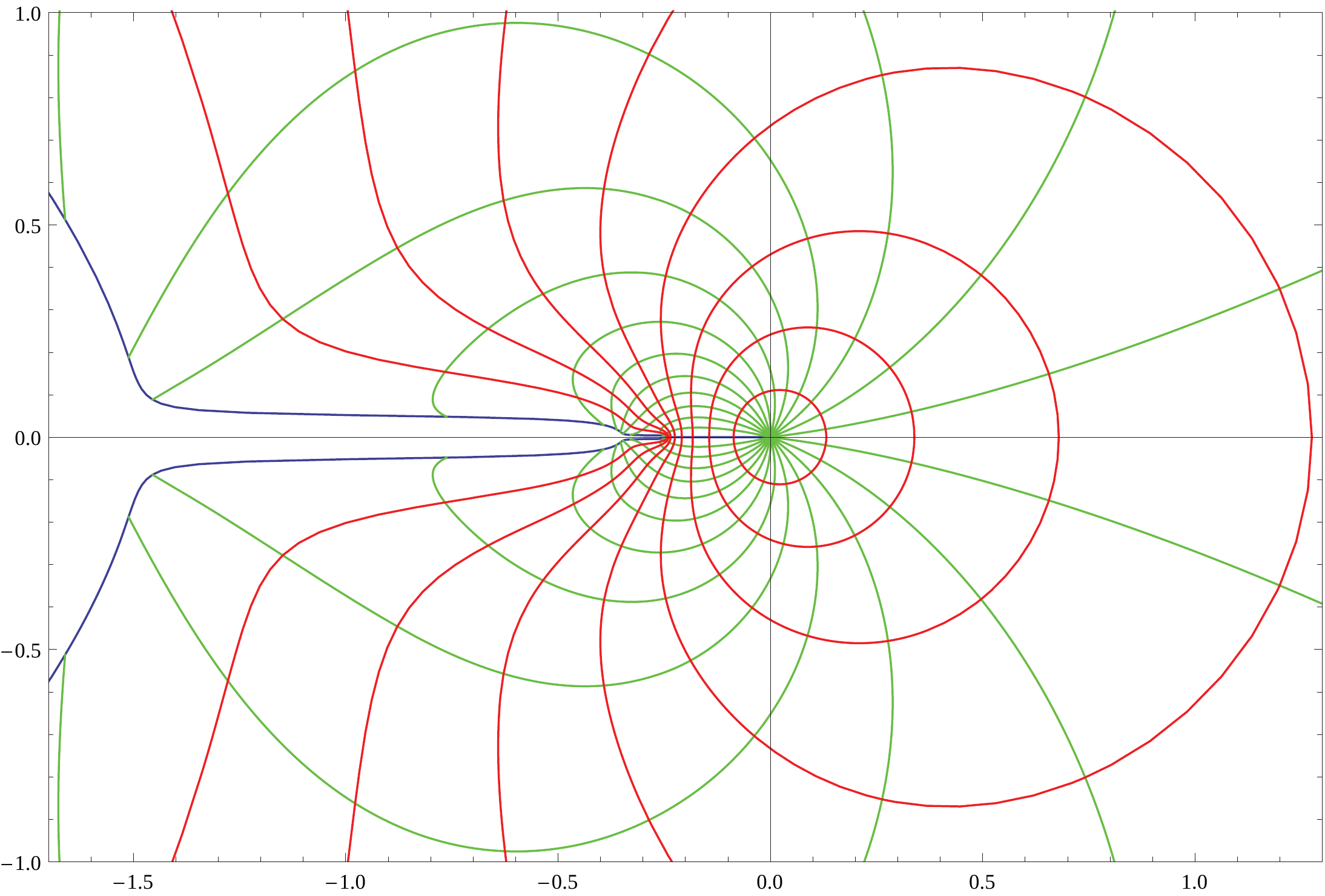}
\end{minipage}}
\subfigure[$n=6$]
{\begin{minipage}[b]{0.45\textwidth}
\includegraphics[height=1.5in,width=2.4in,keepaspectratio]{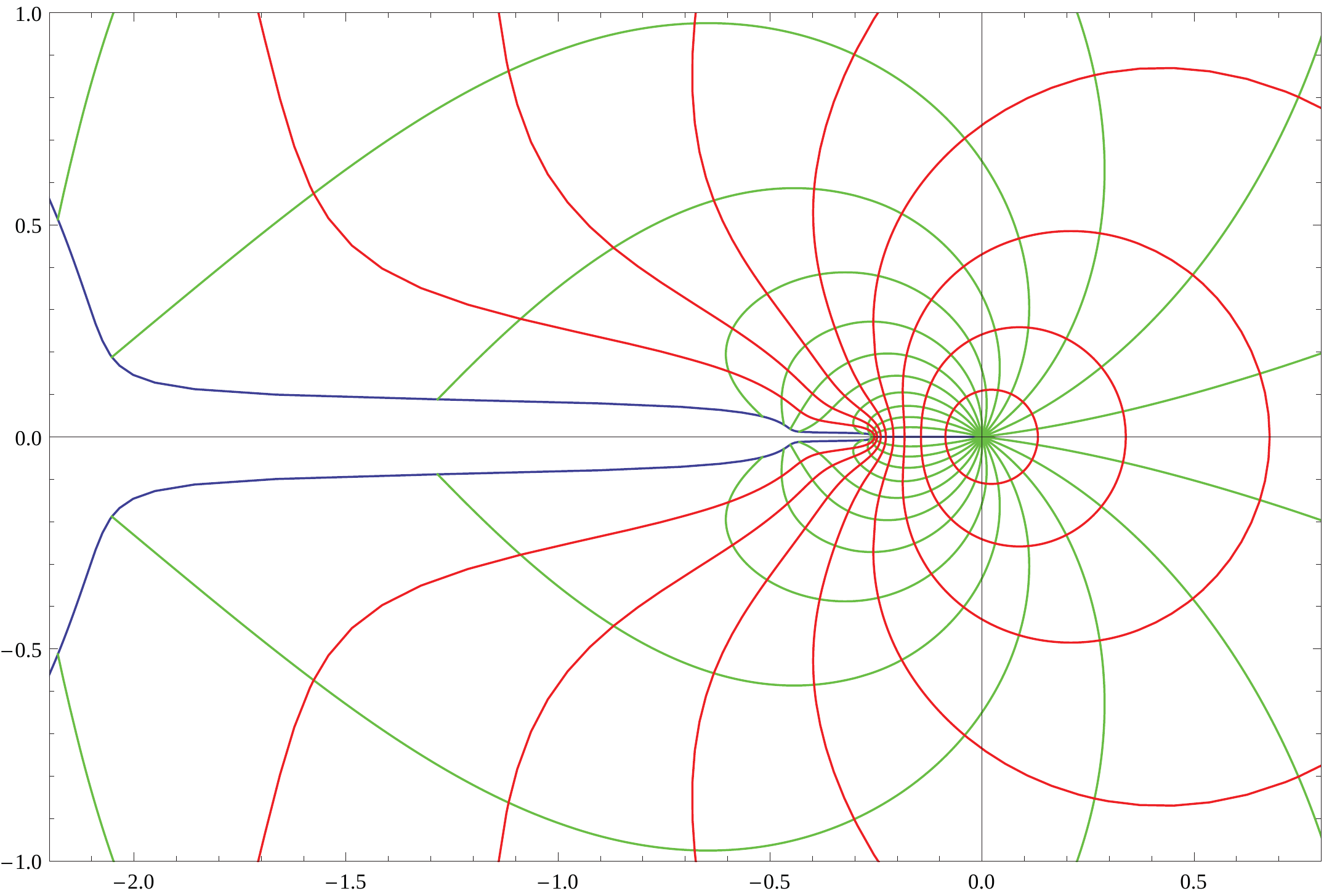}
\end{minipage}}
\caption{Slit images of $f_{2,n}(\mathbb{D})$ for various values of $n=3,4,5,6$. }\label{f2n}
\end{figure}
\begin{figure}[!h]
\centering
\subfigure[$c=2,n=4$]
{\begin{minipage}[b]{0.45\textwidth}
\includegraphics[height=2.4in,width=2.4in,keepaspectratio]{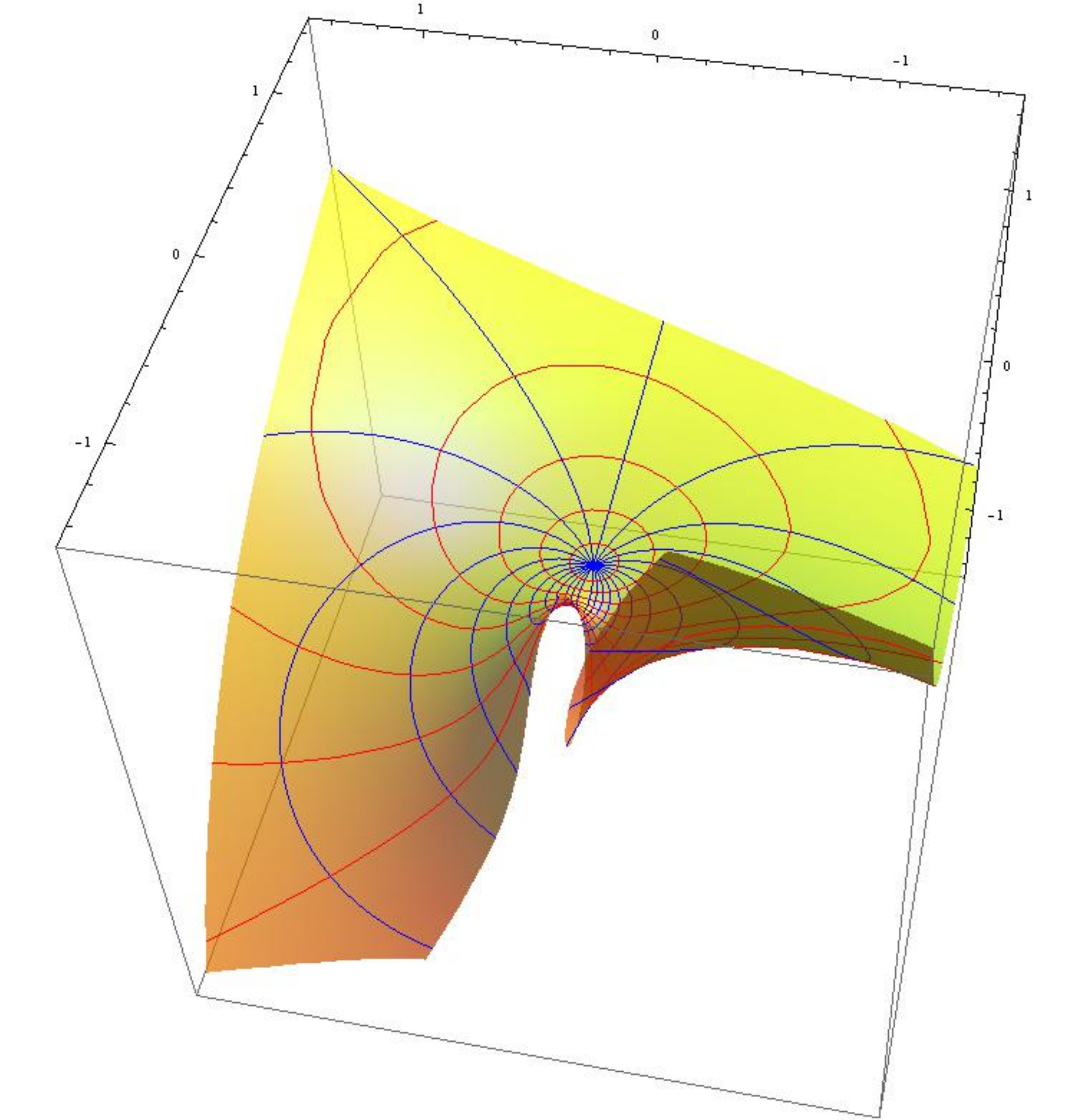}
\end{minipage}}
\subfigure[$c=2,n=6$]
{\begin{minipage}[b]{0.45\textwidth}
\includegraphics[height=2.4in,width=2.4in,keepaspectratio]{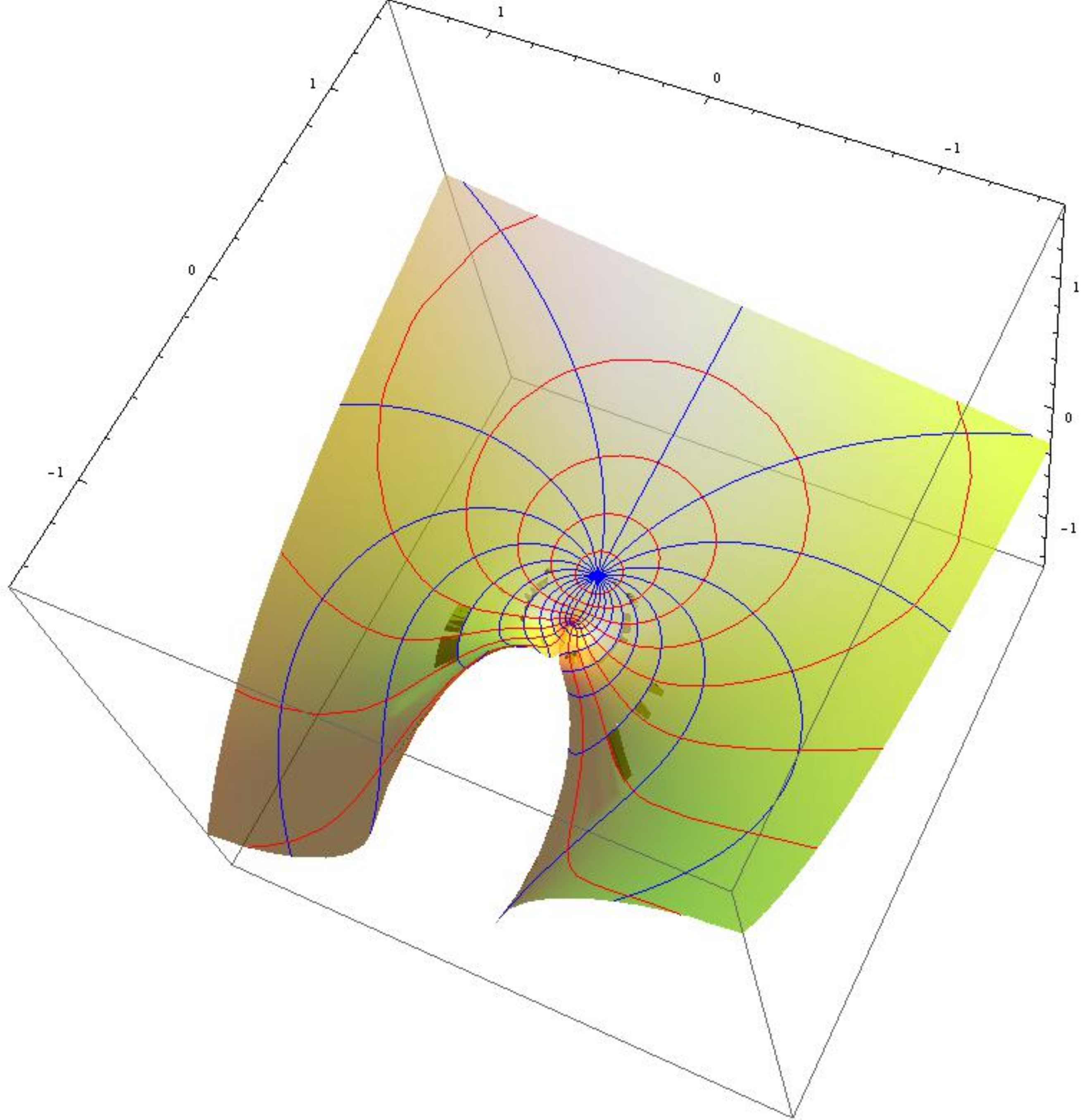}
\end{minipage}}
\subfigure[$c=2,n=8$]
{\begin{minipage}[b]{0.45\textwidth}
\includegraphics[height=2.4in,width=2.4in,keepaspectratio]{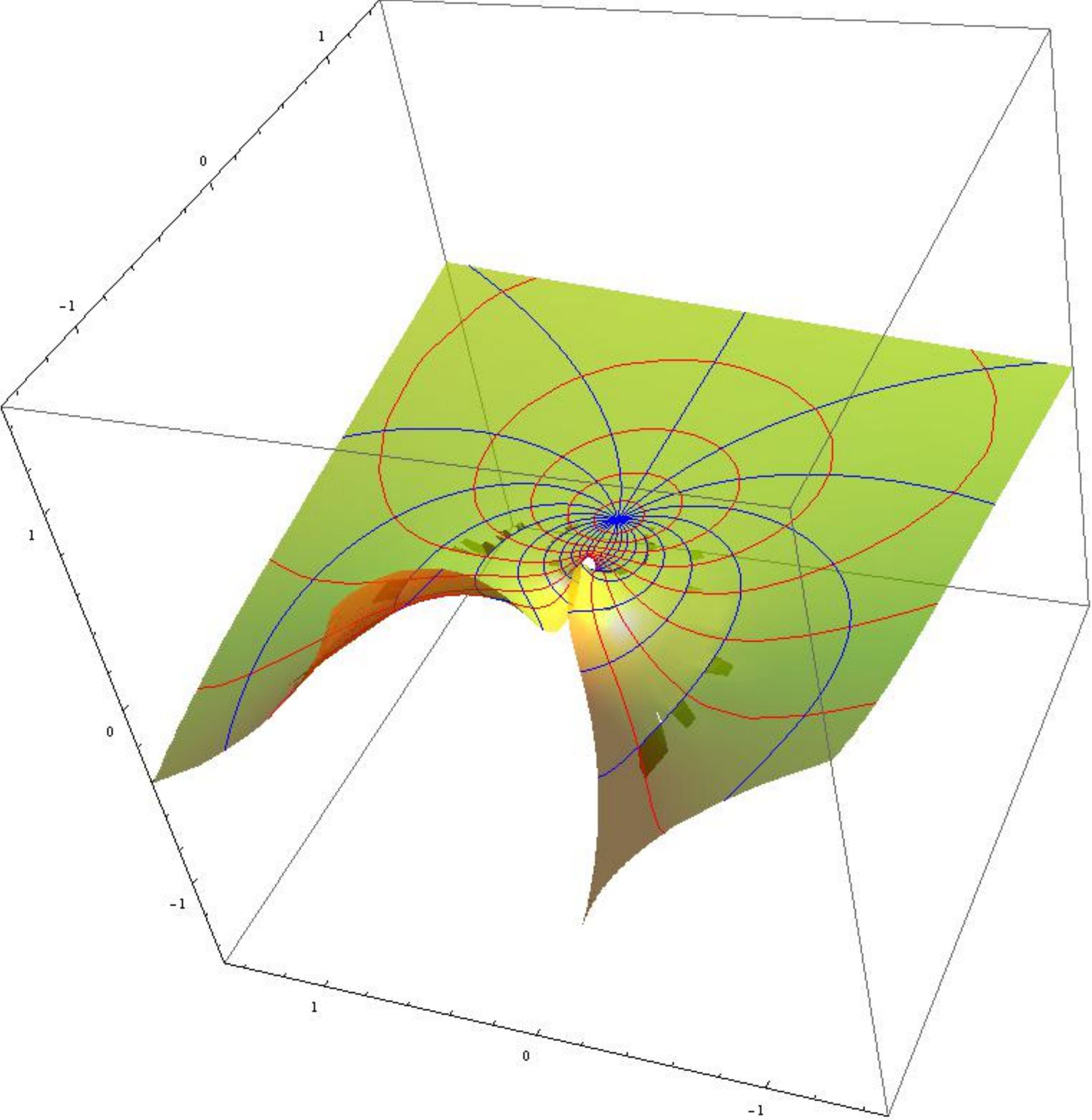}
\end{minipage}}
\subfigure[$c=2,n=10$]
{\begin{minipage}[b]{0.45\textwidth}
\includegraphics[height=2.4in,width=2.4in,keepaspectratio]{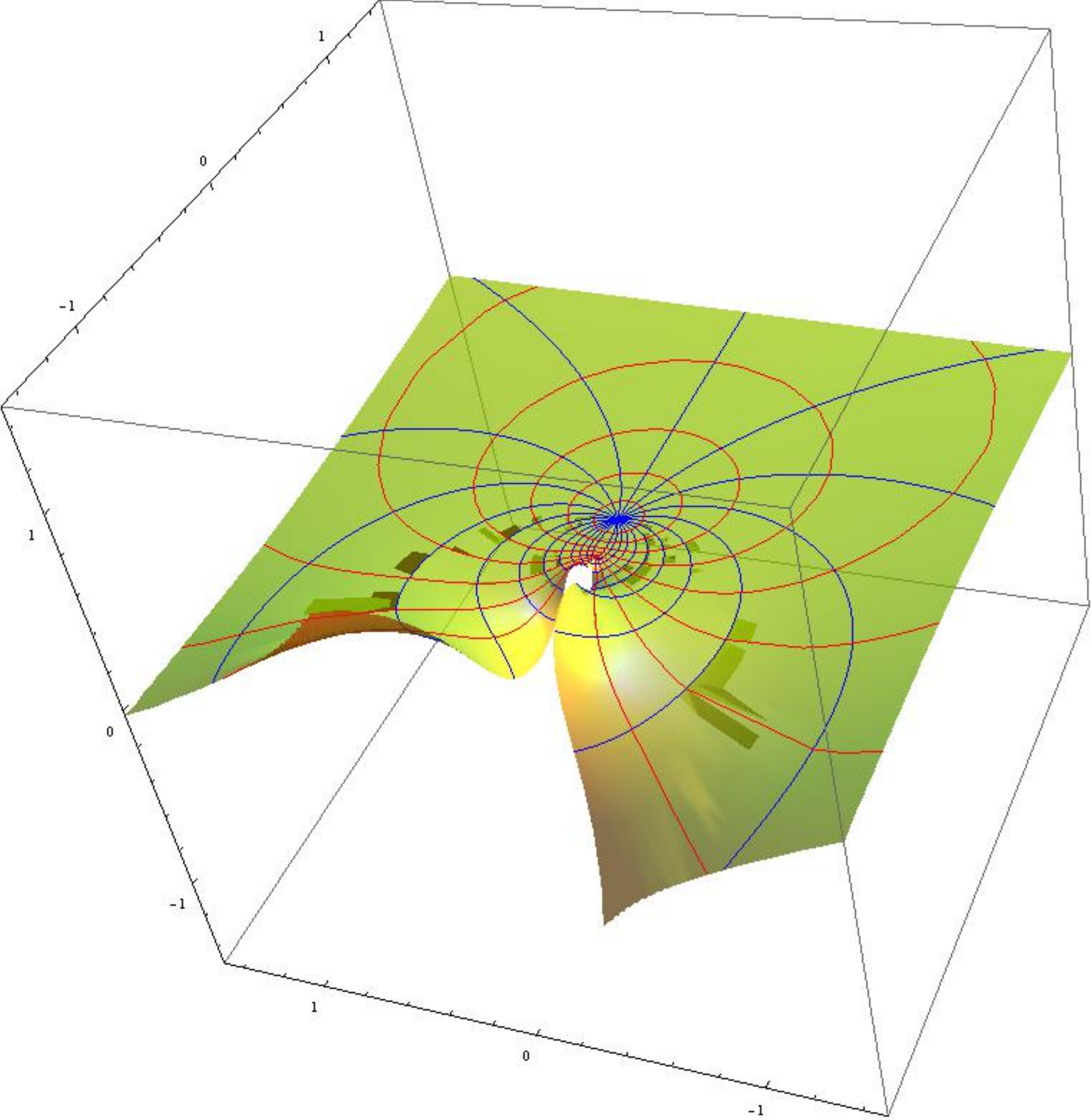}
\end{minipage}}
\subfigure[$c=2,n=12$]
{\begin{minipage}[b]{0.45\textwidth}
\includegraphics[height=2.4in,width=2.4in,keepaspectratio]{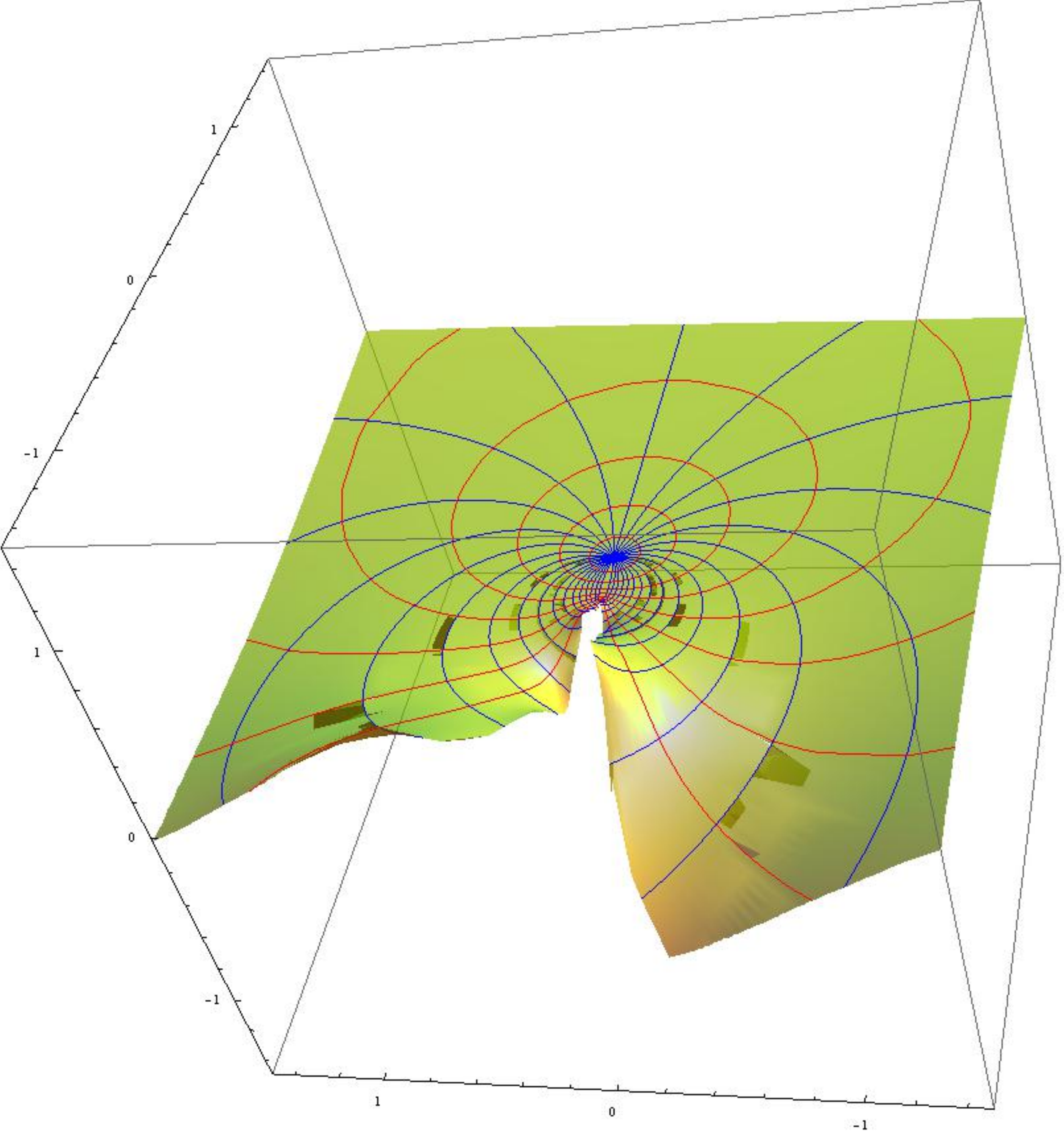}
\end{minipage}}
\subfigure[$c=2,n=14$]
{\begin{minipage}[b]{0.45\textwidth}
\includegraphics[height=2.4in,width=2.4in,keepaspectratio]{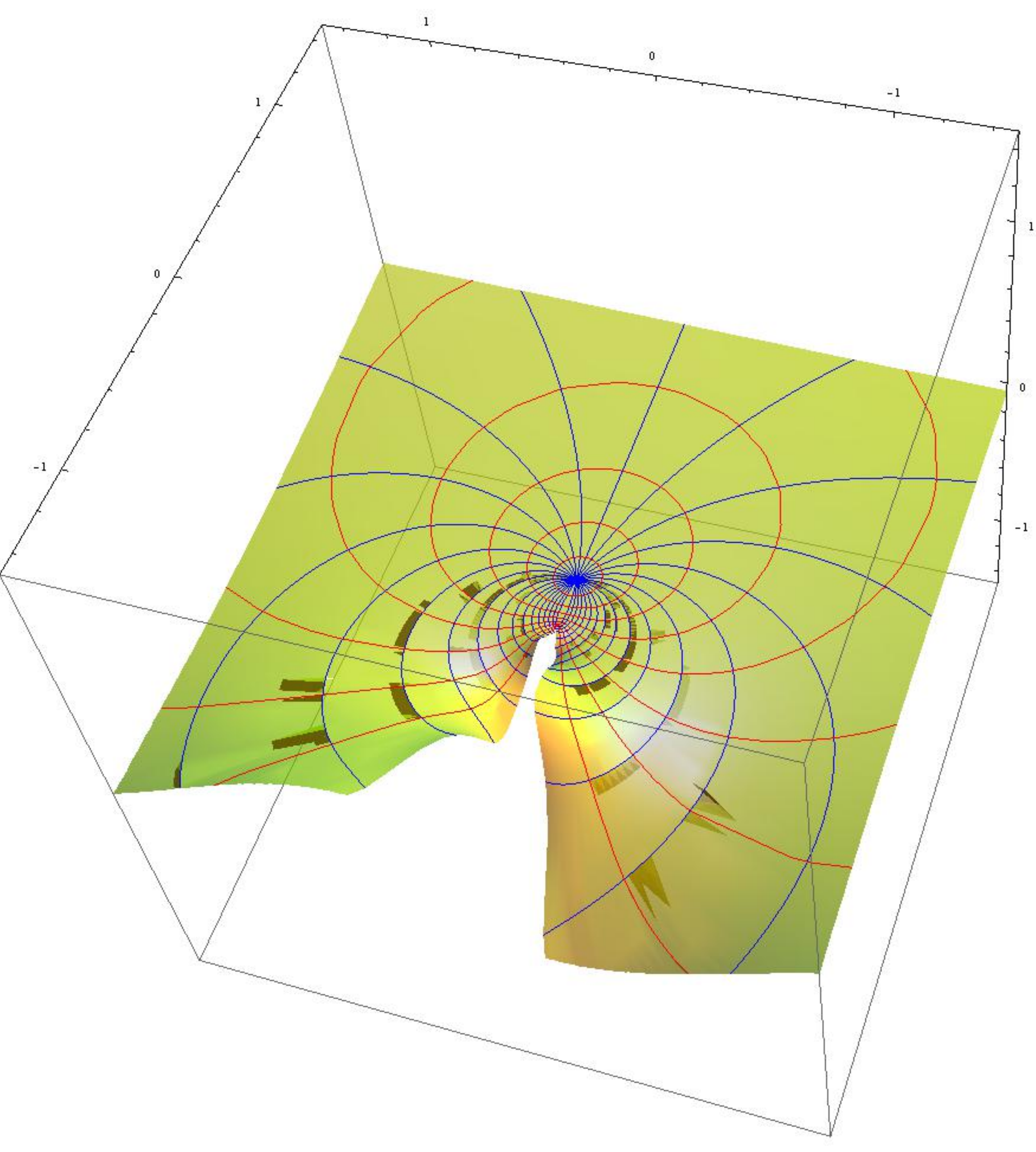}
\end{minipage}}
\caption{$f_{2,n}(\mathbb{D})$ lift to the minimal surfaces for various values of $c=2,n=4,6,8,10,12,14$.}\label{X2n}
\end{figure}
\begin{theorem}\label{thmCHD}
For $c\in[0,2]$ and $n\in\mathbb{N}$, consider the harmonic mappings  $f_{c,n}=h_{c,n}+\overline{g_{c,n}}\in \mathcal{S}_{H}^{0}$ which satisfy the conditions
\begin{equation}\label{equcn}
h_{c,n}(z)-g_{c,n}(z)=k_{c}(z)~\mbox{ and }~ g_{c,n}'(z)=z^{n}h_{c,n}'(z),
\end{equation}
where $k_{c}(z)$ is given by~\eqref{GKF}. Then $f_{c,n}(\mathbb{D})$ is convex in the horizontal direction, and as $c$ varies from $0$ to $2$, $f_{c,n}(\mathbb{D})$ transforms from a strip mapping to a wave plane and then to a slit mapping. In particular, $f_{c,n}(\mathbb{D})$ lifts to the minimal surfaces when $n$ is an even positive integer.
\end{theorem}
\begin{proof} For each $c\in[0,2]$, $k_{c}\in \mathcal{S}$ and $k_{c}(\mathbb{D})$ is a domain convex in the horizontal direction. Thus, by Theorem~\ref{thmA}, $f_{c,n}(\mathbb{D})$ is a CHD domain. What remain to be shown are the mapping properties of the function $f_{c,n}$.
Solving the two equations in \eqref{equcn}, one can easily find that
\begin{equation}\label{equcn1}
h'_{c,n}(z)=\left(\frac{1+z}{1-z}\right)^{c}\frac{1}{(1-z^2)(1-z^{n})}.
\end{equation}
As in the proof of earlier theorems, for the case $n=2m+1 \,(m\in\mathbb{N})$, we may rewrite $h'_{c,n}(z)$ as
\begin{equation*}
\begin{split}
h'_{c,n}(z)&=\left(\frac{1+z}{1-z}\right)^{c}\left [\frac{1}{4}\left(\frac{1}{1-z}+\frac{1}{1+z}\right)+\frac{1}{2n}\frac{1}{(1-z)^2} \right .\\
&\qquad  \left .+\frac{1}{n}\left(\sum_{k=1}^{(n-1)/2}\frac{1}{(1-e^{i\frac{4k\pi}{n}})(1-ze^{-i\frac{2k\pi}{n}})}
+\sum_{k=1}^{(n-1)/2}\frac{1}{(1-e^{-i\frac{4k\pi}{n}})(1-ze^{i\frac{2k\pi}{n}})}\right)\right ].
\end{split}
\end{equation*}
Integrating the last equation from $0$ to $z$ and then analyzing the resulting expression carefully, one obtains
\begin{equation*}
\begin{split}
h_{c,n}(z)&=\left(\frac{1+z}{1-z}\right)^c\left [\frac{1}{4 c}+\frac{1+z}{4 n(1+c) (1-z)} \right .\\
&\qquad  +\sum_{k=1}^{(n-1)/2}\frac{2^c (1-z) e^{ i\frac{2k\pi}{n}} F_1\left(1-c;-c,1;2-c;\frac{1-z}{2},\frac{1-z}{1-e^{ i\frac{2 k \pi }{n}}}\right)}{ n(1-c)(1+z)^{c}\left(1-e^{ i\frac{2 k\pi}{n}}\right) \left(1-e^{ i\frac{4k\pi  }{n}}\right)}\\
& \qquad  \left . -\sum_{k=1}^{(n-1)/2}\frac{2^c (1-z)  F_1\left(1-c;-c,1;2-c;\frac{1-z}{2},\frac{1-z}{1-e^{ -i\frac{2 k \pi}{n}}}\right)}{ n(1-c)(1+z)^{c}  \left(1-e^{ i\frac{2k\pi}{n}}\right) \left(1-e^{- i \frac{4k\pi}{n}}\right)}\right ]-N_{1},
\end{split}
\end{equation*}
where
\begin{equation*}
\begin{split}
N_{1}&=\frac{1}{4 c}+\frac{1}{4 n(1+c)}+\sum _{k=1}^{(n-1)/2} \frac{2^c e^{ i \frac{2k\pi}{n}} F_1\left(1-c;-c,1;2-c;\frac{1}{2},\frac{1}{1-e^{i \frac{2 k \pi }{n}}}\right)}{ n(1-c) \left(1-e^{ i\frac{2k\pi  }{n}}\right) \left(1-e^{ i\frac{4k\pi}{n}}\right)}\\
&\qquad  -\sum _{k=1}^{(n-1)/2}\frac{2^c F_1\left(1-c;-c,1;2-c;\frac{1}{2},\frac{1}{1-e^{-i\frac{2 k \pi }{n}}}\right)}{ n(1-c) \left(1-e^{ i \frac{2k\pi}{n}}\right) \left(1-e^{- i \frac{4k\pi  }{n}}\right)}.
\end{split}
\end{equation*}
As before, we need to deal with the two cases. Observe that if $f=u+iv=h+\overline{g}$ and $h-g=k_c$,
then we can write $f=\RE\{h+g\}+i\IM\{h-g\}=\RE\{2h(z)-k_c\}+i\IM\{k_c\}$
and as a consequence of it and \eqref{equcn}, the resulting harmonic mapping $f_{c,n}(z)$ for the case of odd values of $n$ has the form
\begin{equation*}
\begin{split}
f_{c,n}(z)&=\RE\{2h_{c,n}(z)-k_{c}(z)\}+i\IM\{k_{c}(z)\}\\
&=\RE\left \{\frac{1}{2c}-2N_{1}+\left(\frac{1+z}{1-z}\right)^c\bigg[\frac{1+z}{2n(1+c) (1-z)}\right .\\
&\qquad  +\sum_{k=1}^{(n-1)/2}\frac{2^{c-1} (1-z)   e^{ i\frac{2k\pi}{n}} F_1\left(1-c;-c,1;2-c;\frac{1-z}{2},\frac{1-z}{1-e^{ i\frac{2 k \pi }{n}}}\right)}{ n(1-c)(1+z)^{c}\left(1-e^{ i\frac{2 k\pi}{n}}\right) \left(1-e^{ i\frac{4k\pi  }{n}}\right)}\\
&\qquad  - \left . \left .\sum_{k=1}^{(n-1)/2}\frac{2^{c-1} (1-z)  F_1\left(1-c;-c,1;2-c;\frac{1-z}{2},\frac{1-z}{1-e^{-i\frac{2 k \pi}{n}}}\right)}
{ n(1-c)(1+z)^{c}  \left(1-e^{ i\frac{2k\pi}{n}}\right) \left(1-e^{- i \frac{4k\pi}{n}}\right)}\right]\right \}\\
&\qquad  +i\IM\left\{\frac{1}{2c}\left[\left(\frac{1+z}{1-z}\right)^{c}-1\right]\right\}.
\end{split}
\end{equation*}
Similarly, if $n=2m ~(m\in\mathbb{N})$, then from \eqref{equcn1} one can easily see that
\begin{equation*}
\begin{split}
h'_{c,n}(z)&=\left(\frac{1+z}{1-z}\right)^{c}\left[\frac{1}{4}\left(\frac{1}{1-z}+\frac{1}{1+z}\right)+\frac{1}{2n}\left(\frac{1}{(1-z)^2}+\frac{1}{(1+z)^2}\right)
\right .\\
&\qquad  \left .+\frac{1}{n}\left(\sum_{k=1}^{(n/2)-1}\frac{1}{(1-e^{i\frac{4k\pi}{n}})(1-ze^{-i\frac{2k\pi}{n}})}
+\sum_{k=1}^{(n/2)-1}\frac{1}{(1-e^{-i\frac{4k\pi}{n}})(1-ze^{i\frac{2k\pi}{n}})}\right)\right ]
\end{split}
\end{equation*}
and thus, integrating it from $0$ to $z$ gives
\begin{equation*}
\begin{split}
h_{c,n}(z)&=\left(\frac{1+z}{1-z}\right)^c\left [\frac{1}{4 c}+\frac{1+z}{4n(1+c)(1-z)}-\frac{1-z}{4 n(1-c) (1+z)}\right .\\
&\qquad +\sum_{k=1}^{(n/2)-1}\frac{2^c (1-z)   e^{ i\frac{2k\pi}{n}} F_1\left(1-c;-c,1;2-c;\frac{1-z}{2},\frac{1-z}{1-e^{ i\frac{2 k \pi }{n}}}\right)}{ n(1-c)(1+z)^{c}\left(1-e^{ i\frac{2 k\pi}{n}}\right)\left(1-e^{ i\frac{4k\pi  }{n}}\right)}\\
&\qquad  \left . -\sum_{k=1}^{(n/2)-1}\frac{2^c (1-z)  F_1\left(1-c;-c,1;2-c;\frac{1-z}{2},\frac{1-z}{1-e^{-i\frac{2 k \pi}{n}}}\right)}{ n(1-c)(1+z)^{c}  \left(1-e^{ i\frac{2k\pi}{n}}\right) \left(1-e^{- i \frac{4k\pi}{n}}\right)}\right ]-N_{2},
\end{split}
\end{equation*}
where
\begin{equation*}
\begin{split}
N_{2}&=\frac{1}{4 c}+\frac{c}{2n(1-c^2)}+\sum _{k=1}^{(n/2)-1} \frac{2^c e^{ i \frac{2k\pi}{n}} F_1\left(1-c;-c,1;2-c;\frac{1}{2},\frac{1}{1-e^{i \frac{2 k \pi }{n}}}\right)}{ n(1-c) \left(1-e^{ i\frac{2k\pi  }{n}}\right)^2 \left(1+e^{ i\frac{2k\pi}{n}}\right)}\\
&\qquad  -\sum _{k=1}^{(n/2)-1}\frac{2^c F_1\left(1-c;-c,1;2-c;\frac{1}{2},\frac{-e^{ i\frac{2 k \pi }{n}}}{1-e^{ i\frac{2 k \pi }{n}}}\right)}{ n(1-c) \left(1-e^{ i \frac{2k\pi}{n}}\right) \left(1-e^{- i \frac{4k\pi  }{n}}\right)}.
\end{split}
\end{equation*}
Again, using the observation made in the case of odd values of $n$, the resulting harmonic mapping $f_{c,n}(z)$ for even values of $n$ is given by
$$f_{c,n}(z)=\RE\{2h_{c,n}(z)-k_{c}(z)\}+i\IM\{k_{c}(z)\} =u+iv,
$$
where $u$ and $v$ in this case take the form
\begin{equation*}
\begin{split}
u&=\RE\bigg\{\frac{1}{2c}-2N_{2}+\left(\frac{1+z}{1-z}\right)^c\left [\frac{1+z}{2n(1+c)(1-z)}-\frac{1-z}{2n(1-c) (1+z)} \right .\\
&\qquad  +\sum_{k=1}^{(n/2)-1}\frac{2^{c-1} (1-z)   e^{ i\frac{2k\pi}{n}} F_1\left(1-c;-c,1;2-c;\frac{1-z}{2},\frac{1-z}{1-e^{ i\frac{2 k \pi }{n}}}\right)}{ n(1-c)(1+z)^{c}\left(1-e^{ i\frac{2 k\pi}{n}}\right)^2 \left(1+e^{ i\frac{2k\pi  }{n}}\right)}\\
&\qquad  \left .\left .+\sum_{k=1}^{(n/2)-1}\frac{2^{c-1} (1-z)  F_1\left(1-c;-c,1;2-c;\frac{1-z}{2},-\frac{e^{ i \frac{2k\pi }{n}} (1-z)}{1-e^{ i\frac{2 k \pi}{n}}}\right)}{ n(1-c)(1+z)^{c}  \left(-1+e^{ i\frac{2k\pi}{n}}\right) \left(1-e^{- i \frac{4k\pi}{n}}\right)}\right]\right \},~\mbox{ and}\\
v&=\IM\left\{\frac{1}{2c}\left[\left(\frac{1+z}{1-z}\right)^{c}-1\right]\right\},
\end{split}
\end{equation*}
respectively. Note that, by Theorem~\ref{thmB}, for even values of $n$, the harmonic mappings $f_{c,n}(\mathbb{D})$ lifts to the minimal surfaces $\mathbf{X}_{c,n}(u,v)=(u,v,F(u,v))$, where $u,v$ are as above and
\begin{equation*}
\begin{split}
F(u,v)&=\IM\bigg\{\left(\frac{1+z}{1-z}\right)^c\left [\frac{\left(1+i^n\right) (1+z) }{2n(1+c) (1-z)} \right .\\
&\qquad  +\sum_{k=1}^{(n/2)-1}\frac{(-1)^{k}2^{c-1} (1-z)   e^{ i\frac{2k\pi}{n}} F_1\left(1-c;-c,1;2-c;\frac{1-z}{2},\frac{1-z}{1-e^{ i\frac{2 k \pi }{n}}}\right)}{ n(1-c)(1+z)^{c}\left(1-e^{ i\frac{2 k\pi}{n}}\right)^2 \left(1+e^{ i\frac{2k\pi  }{n}}\right)}\\
&\qquad \left . \left . -\sum_{k=1}^{(n/2)-1}\frac{2^{c-1} (1-z)  F_1\left(1-c;-c,1;2-c;\frac{1-z}{2},-\frac{e^{ i \frac{2k\pi }{n}} (1-z)}{1-e^{ i\frac{2 k \pi}{n}}}\right)}{ n(1-c)(1+z)^{c}  \left(1-e^{ i\frac{2k\pi}{n}}\right) \left(1-e^{- i \frac{4k\pi}{n}}\right)}\right ]\right \}.
\end{split}
\end{equation*}
\end{proof}

\begin{remark}
If we take $n=2$, then Theorem~\ref{thmc2} reduces to Theorem 3 in~\cite{Dorff2014AAA}.
\end{remark}
\begin{figure}[!h]
\centering
\subfigure[$c=0$]
{\begin{minipage}[b]{0.45\textwidth}
\includegraphics[height=2.4in,width=2.4in,keepaspectratio]{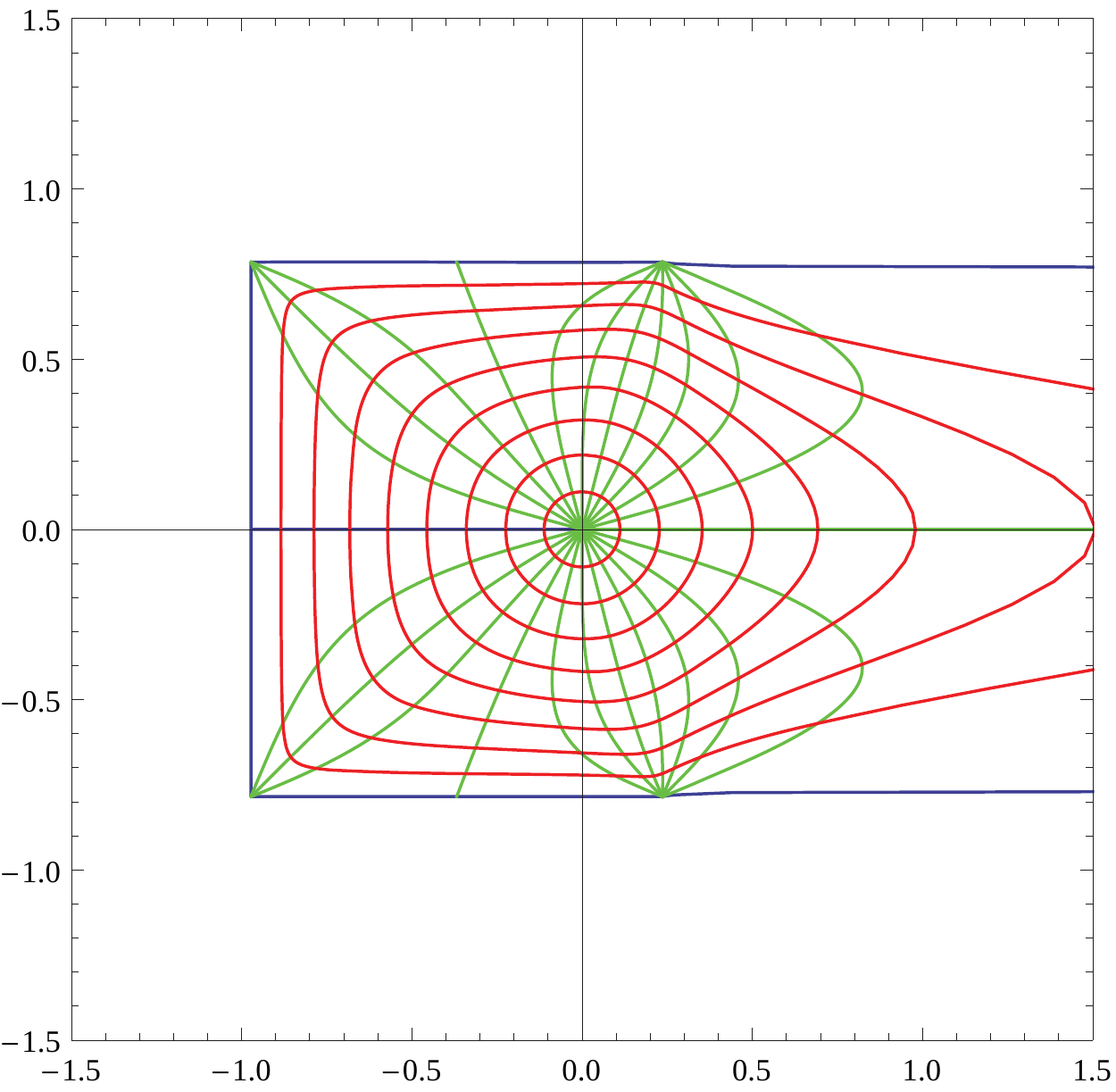}
\end{minipage}}
\subfigure[$c=0.2$]
{\begin{minipage}[b]{0.45\textwidth}
\includegraphics[height=2.4in,width=2.4in,keepaspectratio]{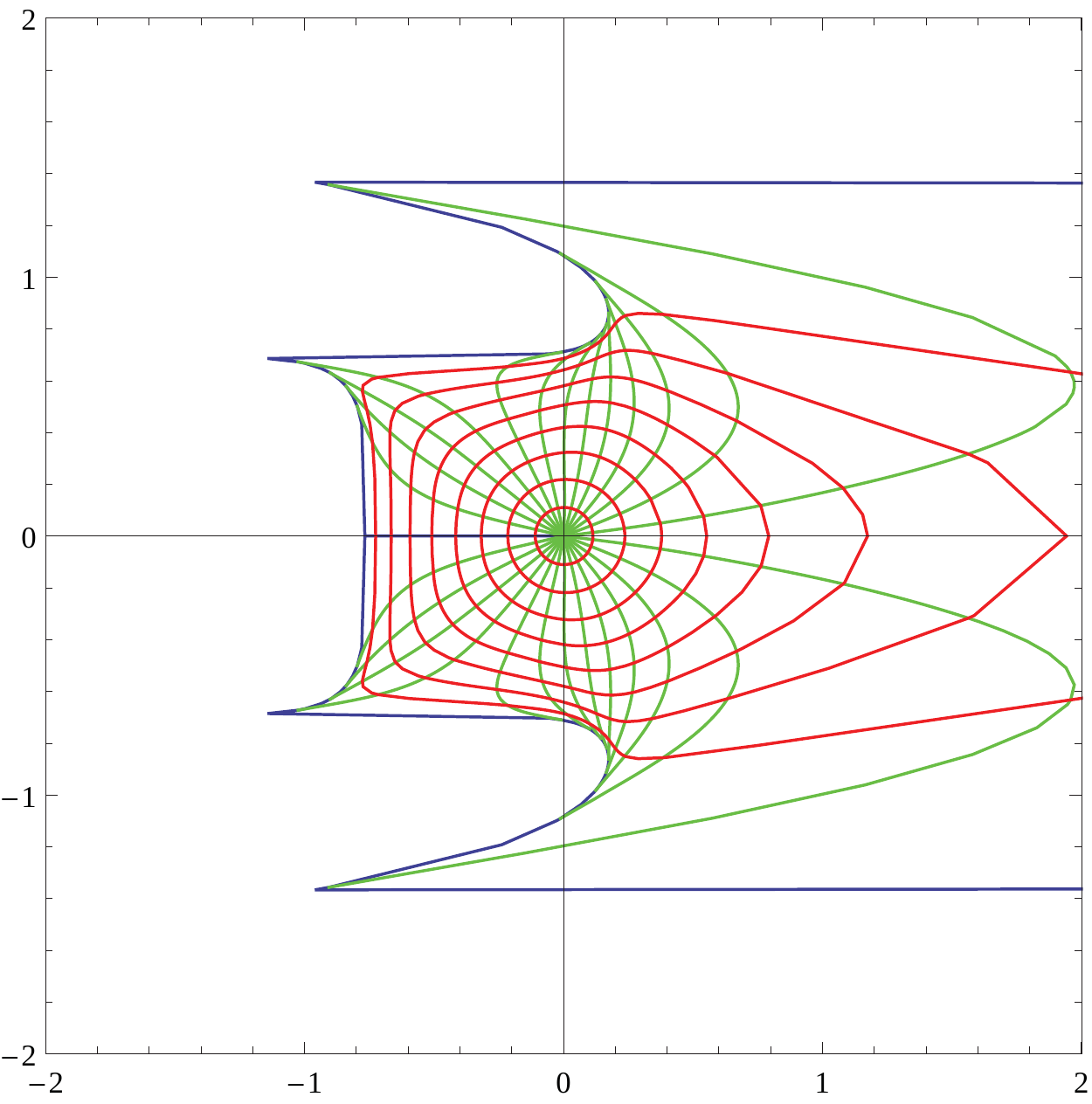}
\end{minipage}}
\subfigure[$c=0.5$]
{\begin{minipage}[b]{0.45\textwidth}
\includegraphics[height=2.4in,width=2.4in,keepaspectratio]{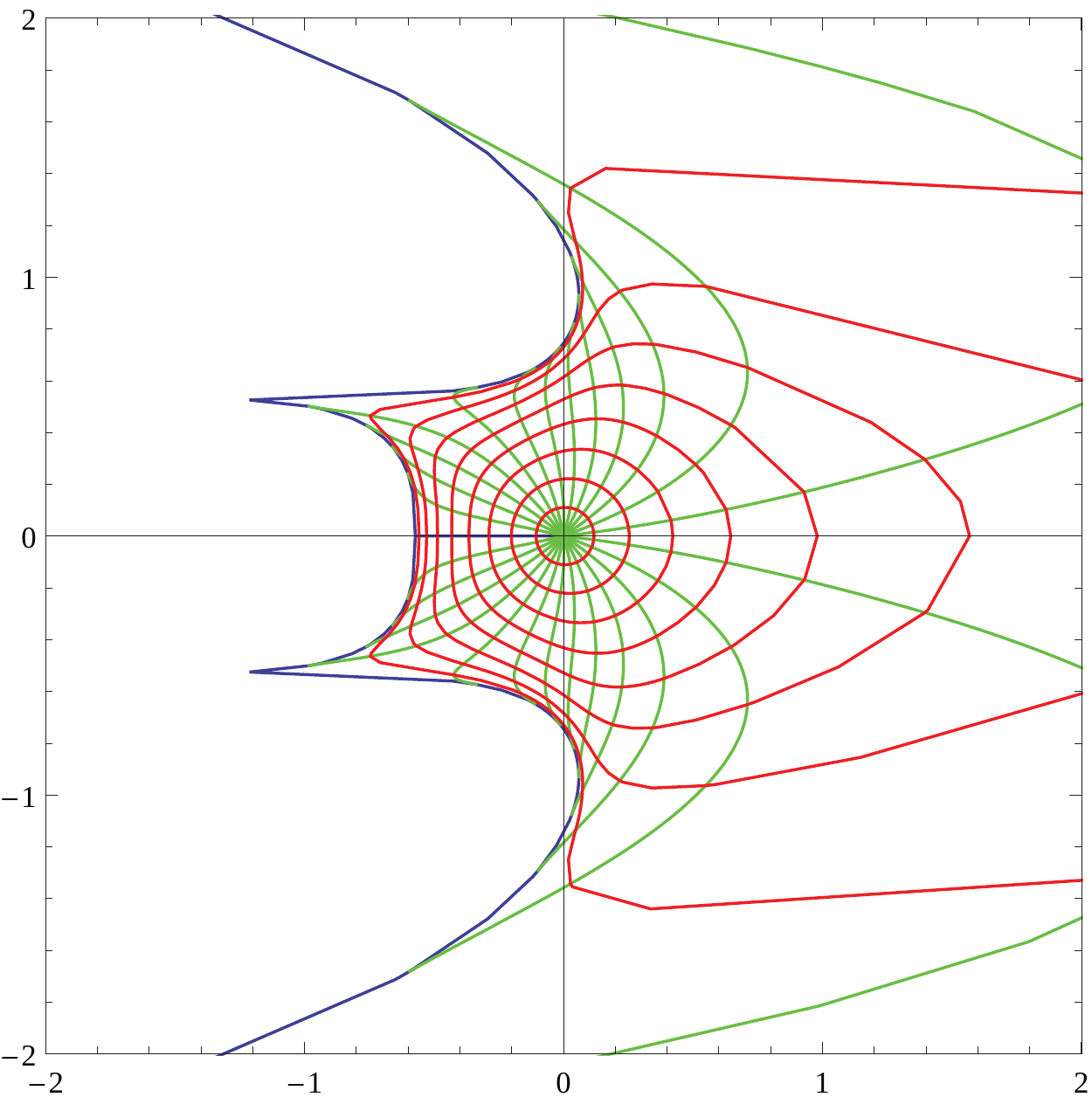}
\end{minipage}}
\subfigure[$c=0.8$]
{\begin{minipage}[b]{0.45\textwidth}
\includegraphics[height=2.4in,width=2.4in,keepaspectratio]{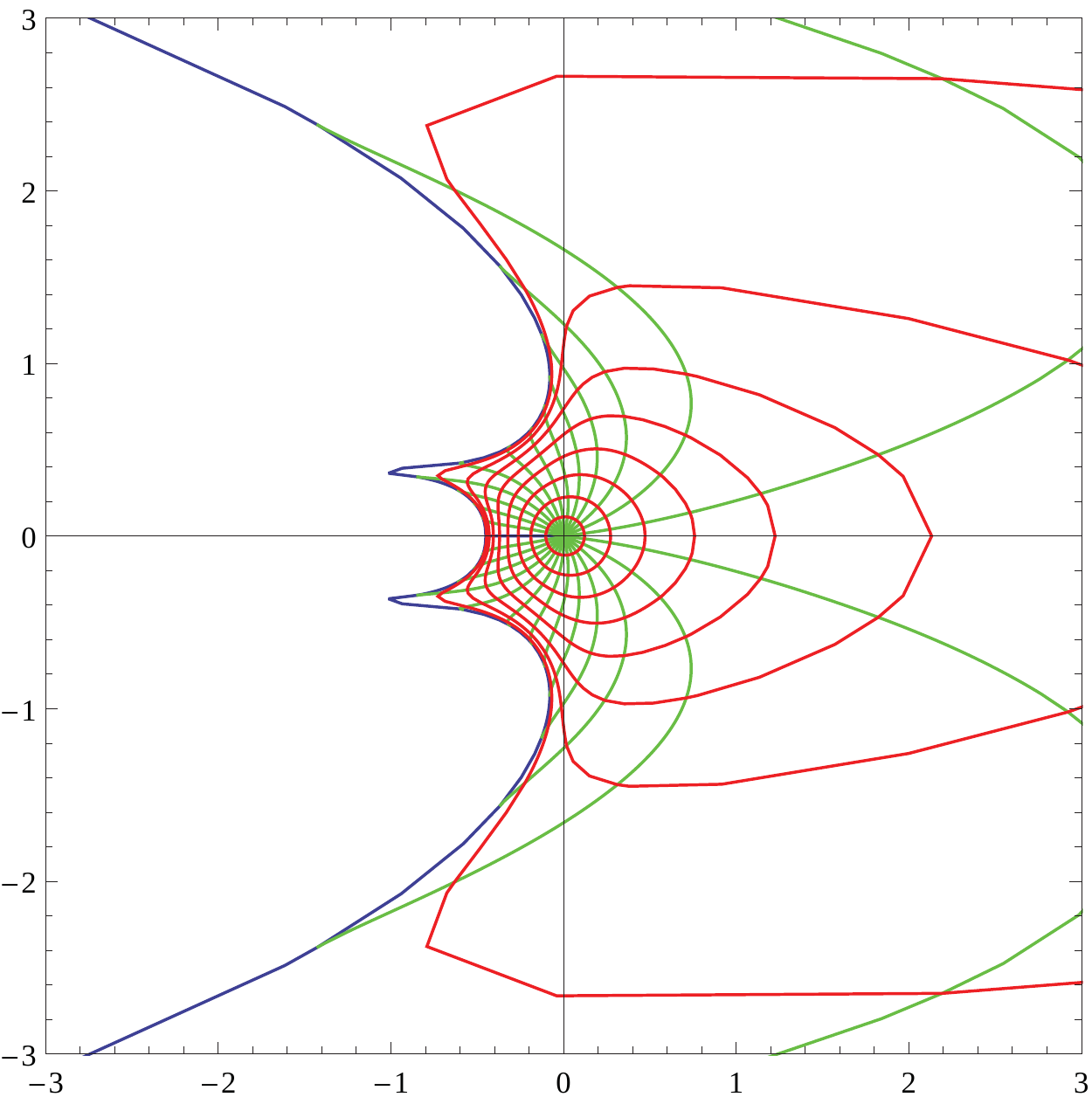}
\end{minipage}}
\subfigure[$c=1$]
{\begin{minipage}[b]{0.45\textwidth}
\includegraphics[height=2.4in,width=2.4in,keepaspectratio]{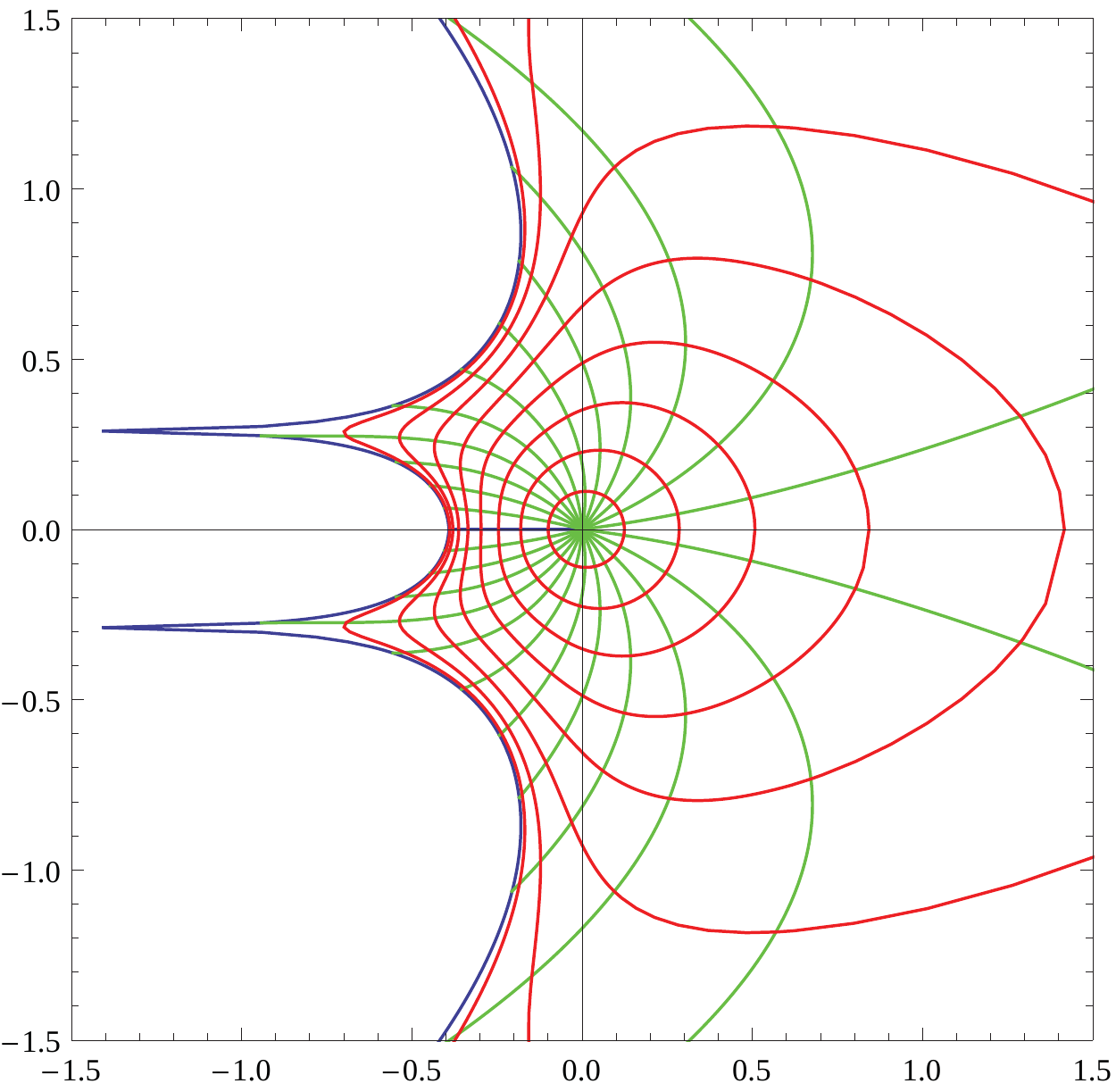}
\end{minipage}}
\subfigure[$c=1.5$]
{\begin{minipage}[b]{0.45\textwidth}
\includegraphics[height=2.4in,width=2.4in,keepaspectratio]{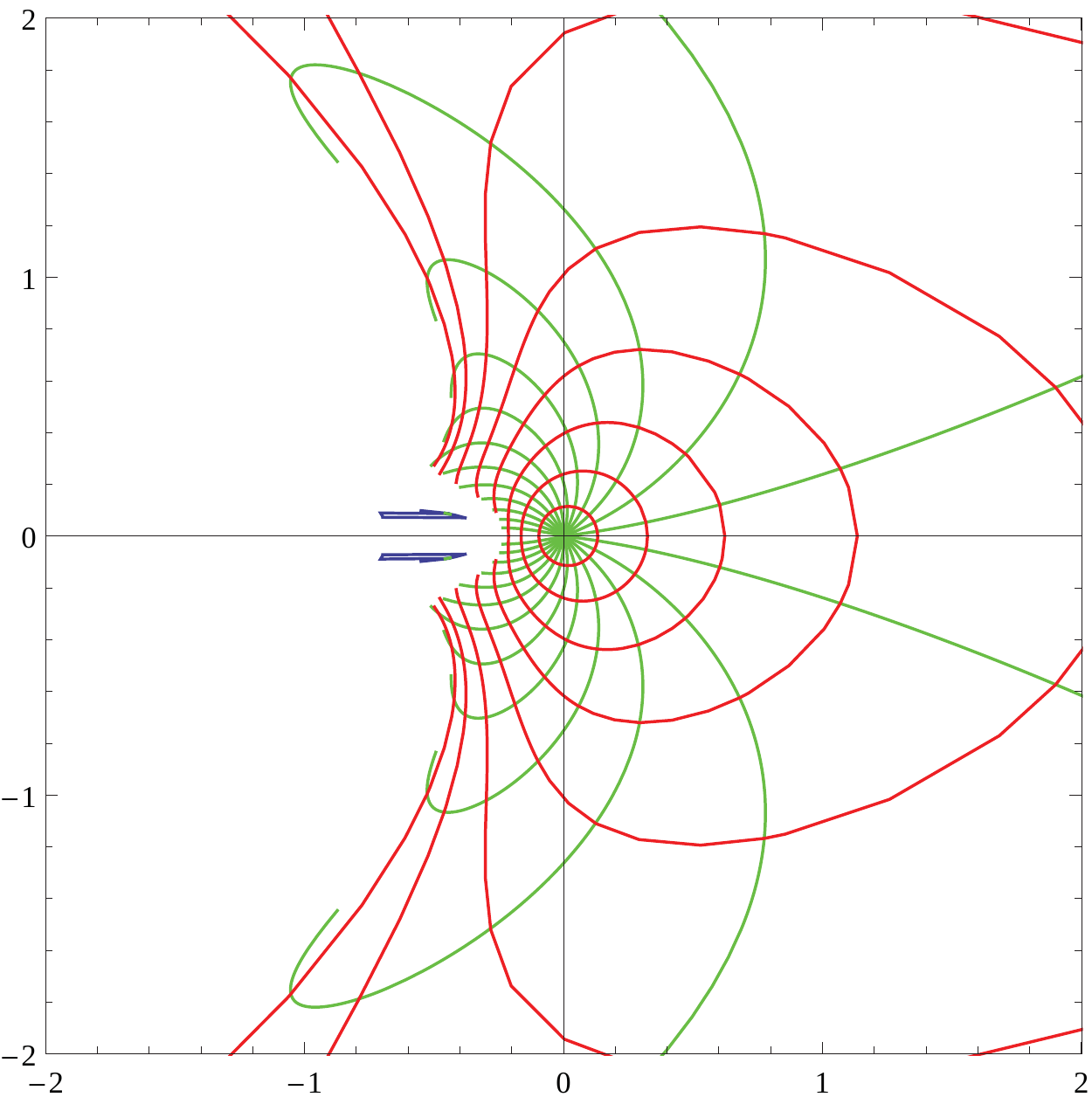}
\end{minipage}}
\caption{Images of $f_{c,3}(\mathbb{D})$ for various values of $c\in[0,2]$ }\label{fc3}
\end{figure}

\begin{figure}[!h]
\centering
\subfigure[$c=0$]
{\begin{minipage}[b]{0.45\textwidth}
\includegraphics[height=2.4in,width=2.4in,keepaspectratio]{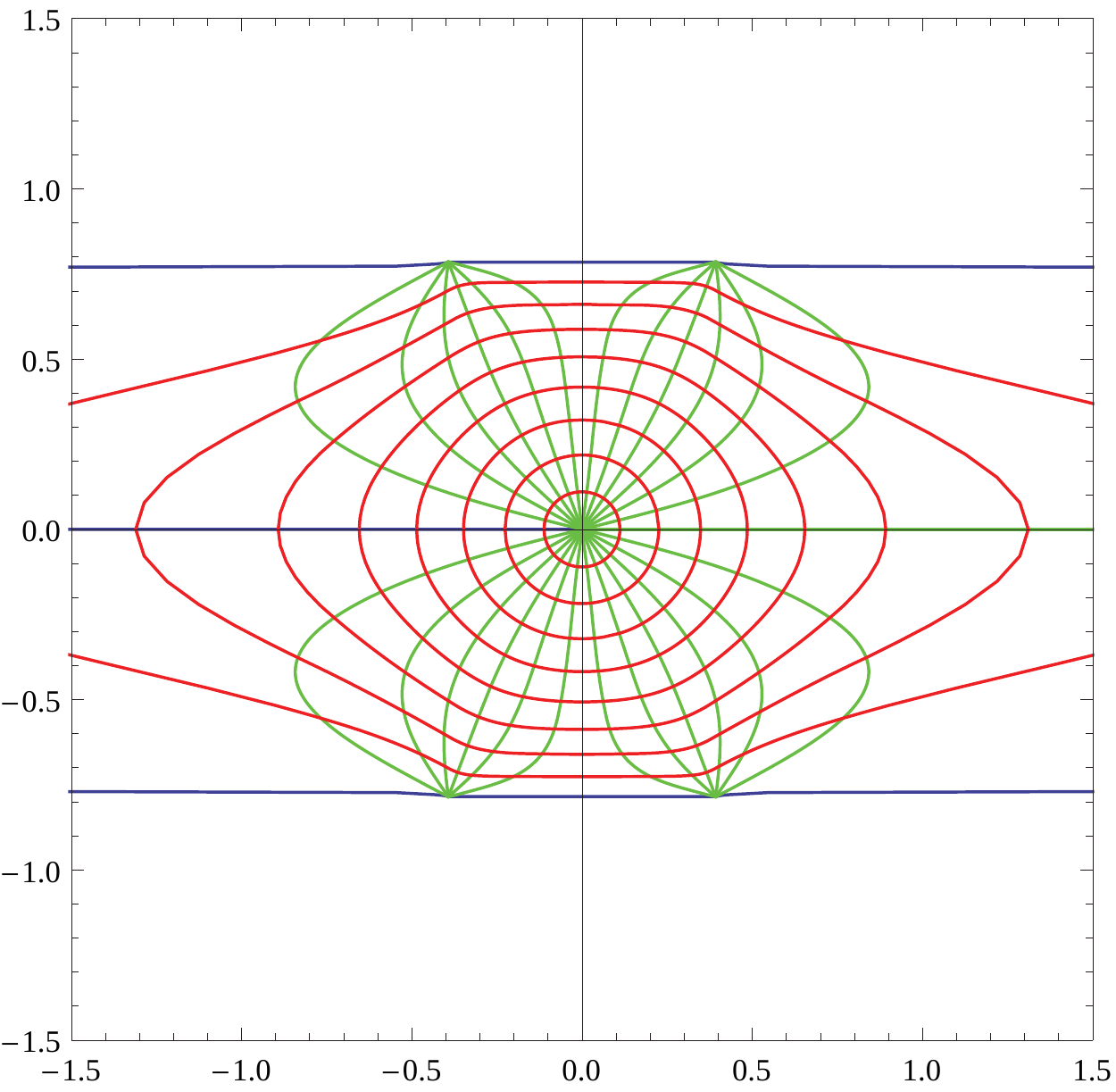}
\end{minipage}}
\subfigure[$c=0.2$]
{\begin{minipage}[b]{0.45\textwidth}
\includegraphics[height=2.4in,width=2.4in,keepaspectratio]{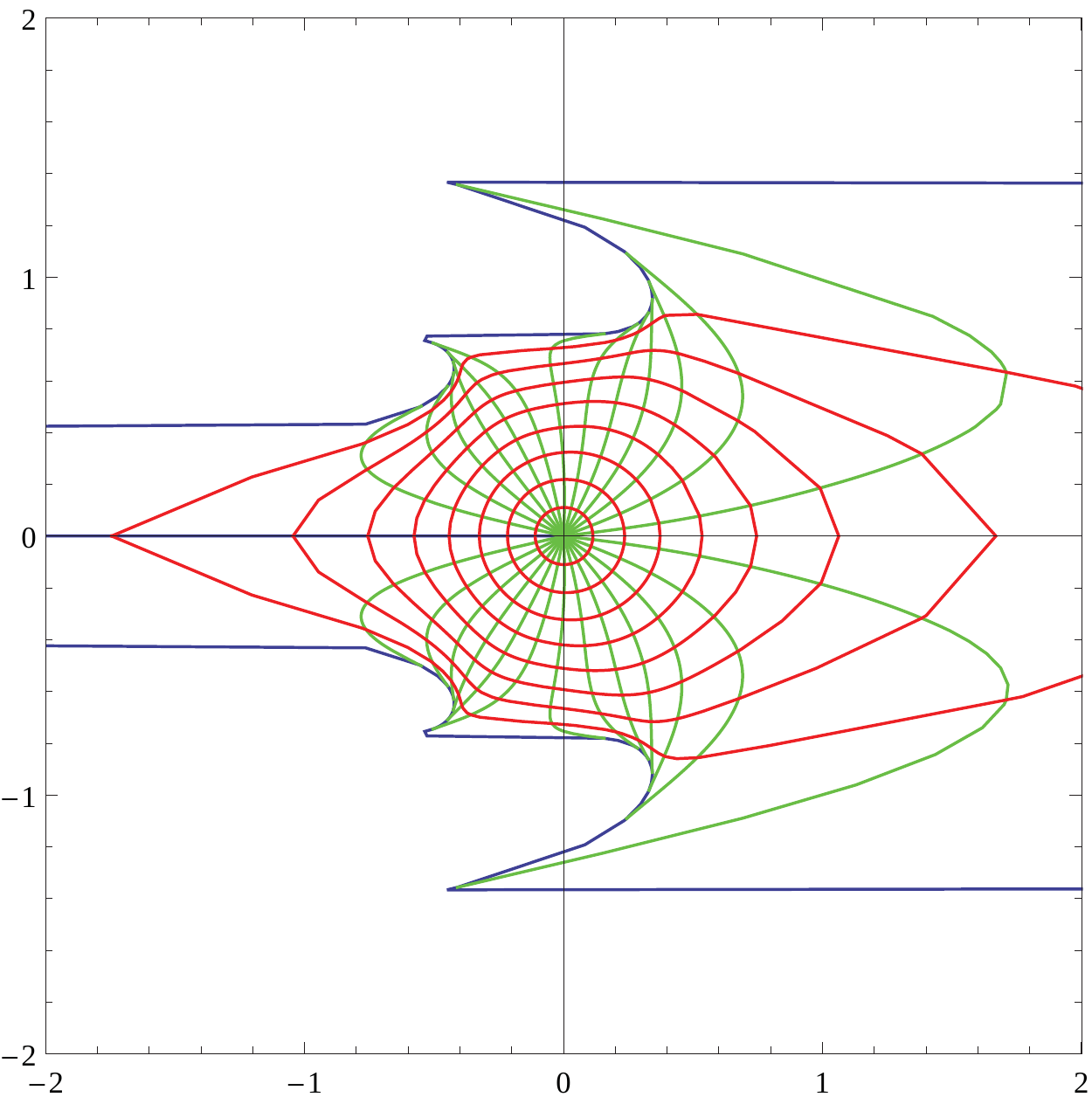}
\end{minipage}}
\subfigure[$c=0.5$]
{\begin{minipage}[b]{0.45\textwidth}
\includegraphics[height=2.4in,width=2.4in,keepaspectratio]{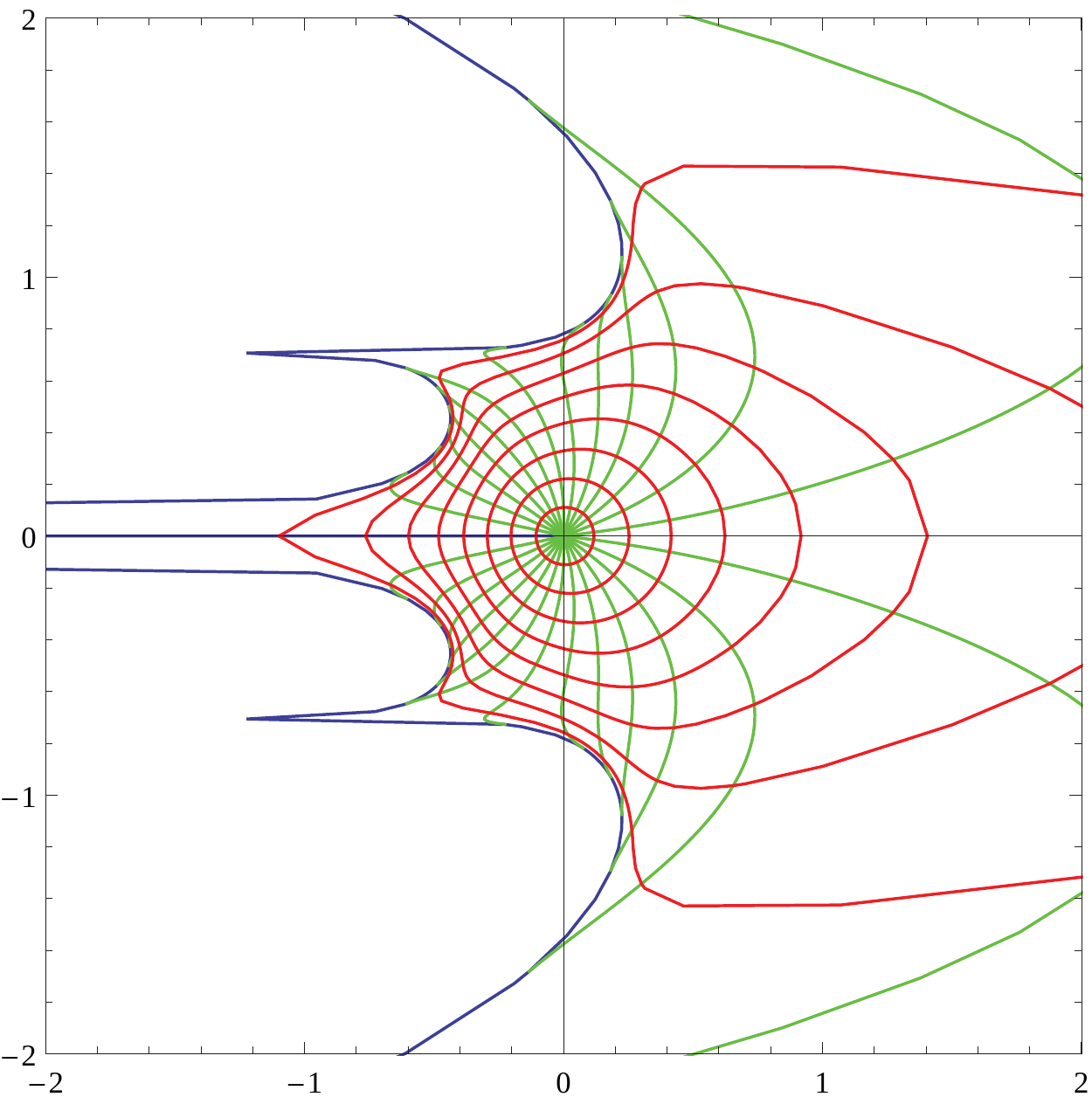}
\end{minipage}}
\subfigure[$c=0.8$]
{\begin{minipage}[b]{0.45\textwidth}
\includegraphics[height=2.4in,width=2.4in,keepaspectratio]{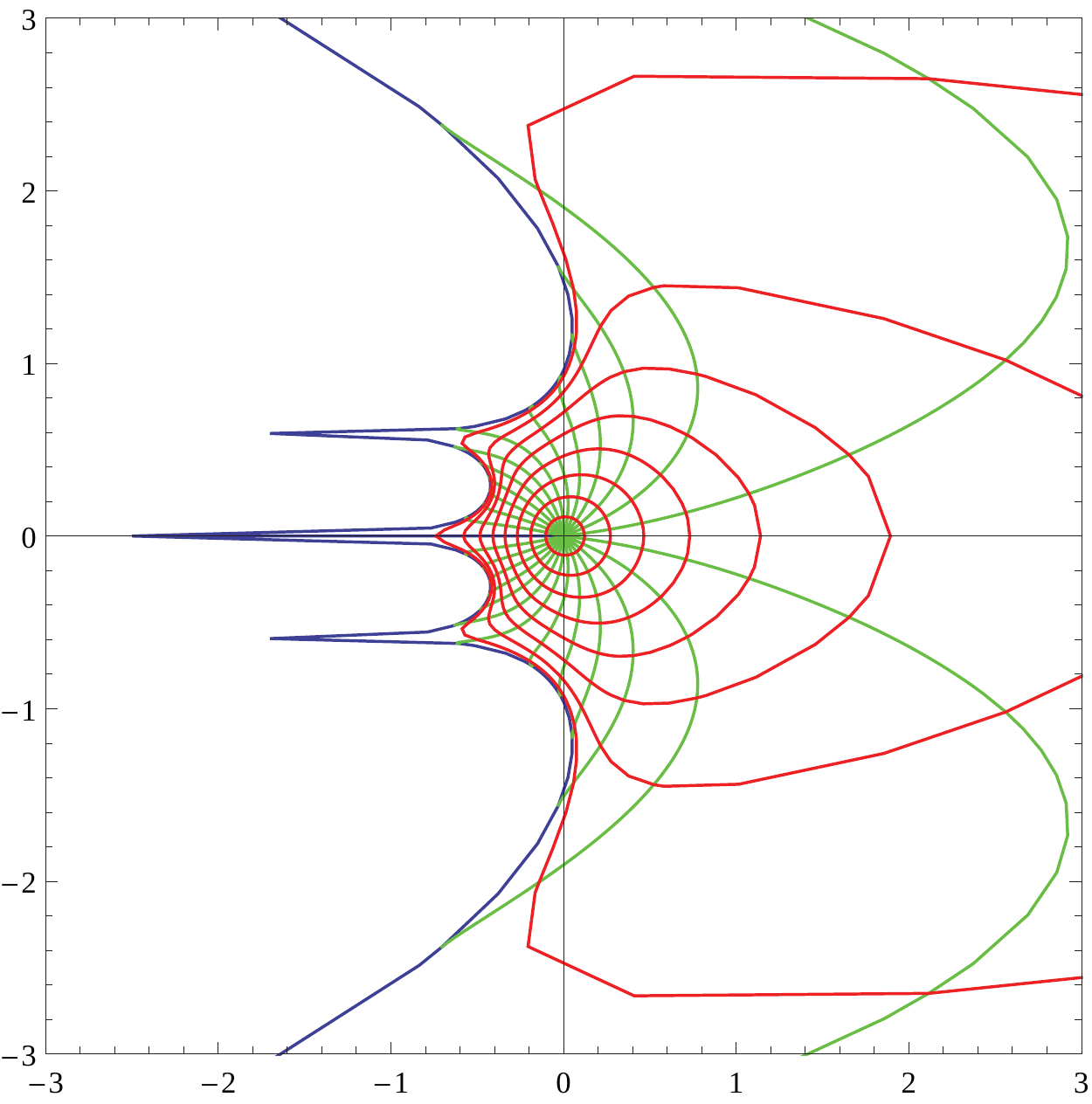}
\end{minipage}}
\subfigure[$c=1$]
{\begin{minipage}[b]{0.45\textwidth}
\includegraphics[height=2.4in,width=2.4in,keepaspectratio]{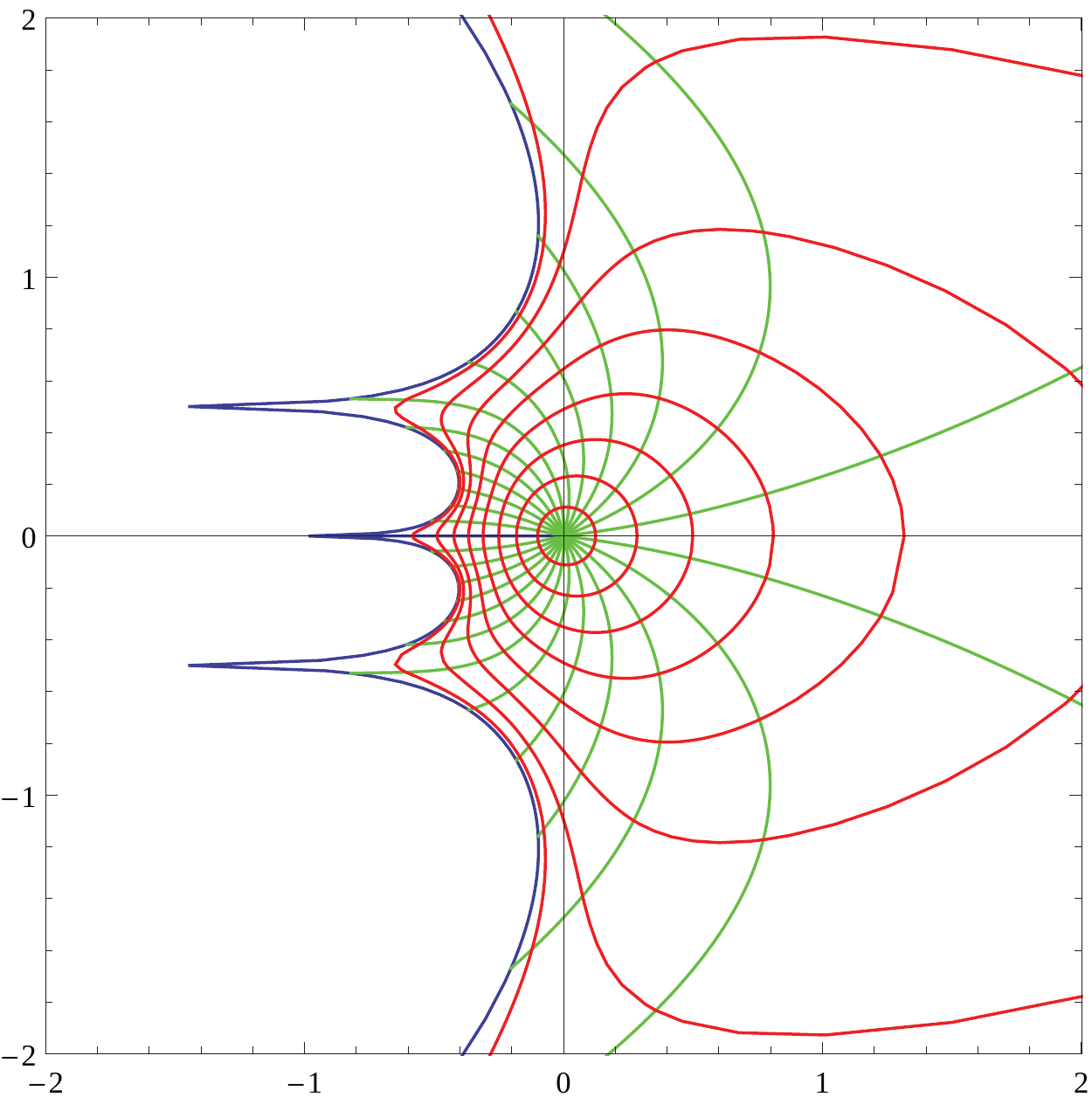}
\end{minipage}}
\subfigure[$c=1.5$]
{\begin{minipage}[b]{0.45\textwidth}
\includegraphics[height=2.4in,width=2.4in,keepaspectratio]{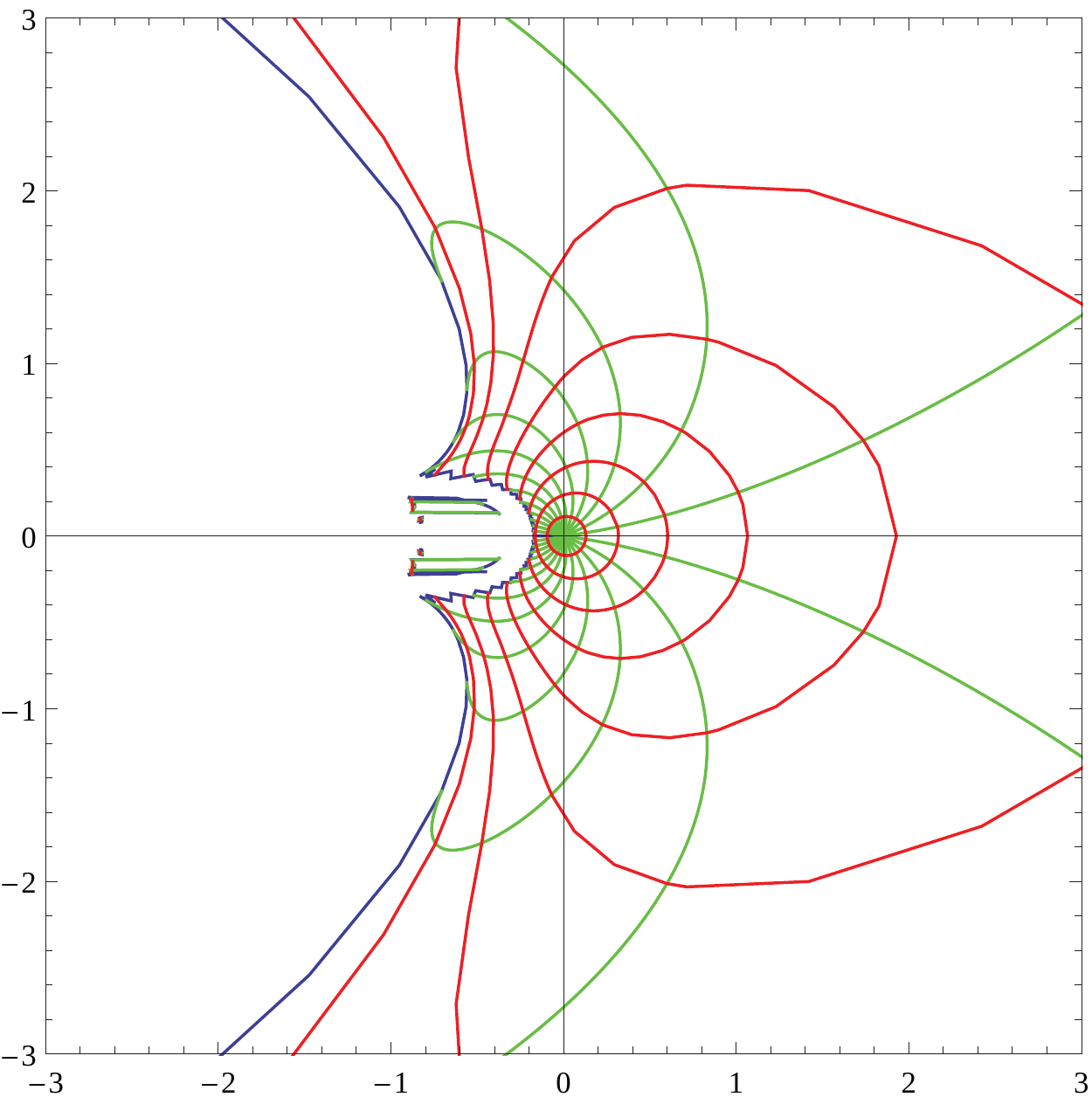}
\end{minipage}}
\caption{Images of $f_{c,4}(\mathbb{D})$ for various values of $c\in[0,2]$ }\label{fc4}
\end{figure}

Figures~\ref{fc3} and  \ref{fc4} are graphs of $f_{c,n}(z)$ for various values of $c\in[0,2]$ and we see that $f_{c}(\mathbb{D})$ transforms from  strip region to wave plane for various values of $c\in[0,1]$.


\vskip.20in
\subsection*{Acknowledgements}
The present investigation was supported by the National Natural Science Foundation under Grant 11371126 of the People's Republic of China, the First Batch of Young and Middle-aged Academic Training Object Backbone of Honghe University under Grant 2014GG0102.


\begin{thebibliography}{99}

\bibitem{Clunie} J. Clunie and T. Sheil-Small,
\textit{Harmonic univalent functions},
Ann. Acad. Sci. Fenn. Ser. A.I. Math. \textbf{9} (1984), 3--25.

\bibitem{Dorffbook2012} M. Dorff,
\textit{Soap films, differential geometry,and minimal surfaces, Explorations in Complex Analysis},
Math. Assoc. of America, Inc., Washington, DC, 2012. pp. 85--159.

\bibitem{Dorff2014AAA} M. Dorff and S. Muir,
\textit{A family of minimal surfaces and univalent plannar harmonic mappings},
Abstr. Appl. Anal. 2014, Art. ID 476061, 8 pages.

\bibitem{Duren2004} P. Duren,
\textit{Harmonic Mappings in the Plane},
Vol. 156 of Cambridge Tracts in Mathematics, Cambridge University Press,
Cambridge, UK, 2004.

\bibitem{Greiner2004} P. Greiner,
\textit{Geometric properties of harmonic shears},
Comput. Methods Funct. Theory, \textbf{4}(1) (2004), 77--96.

\bibitem{Lewy} H. Lewy,
\textit{On the non-vanishing of the Jacobian in certain one-to-one mappings},
\textrm{Bull. Amer. Math. Soc.} {\bf 42}(10) (1936), 689--692.

\bibitem{LiulanLi201204} L. Li, S. Ponnusamy and M. Vuorinen,
\textit{The minimal surfaces over the slanted half-planes, vertical strips and single slit},
Chapter in ``Current Topics in Pure and Computational Complex Analysis" (2014),  47--61. (Eds. S. Joshi, M. Dorff, I. Lahri),
Trends in Mathematics, Birkauser.
\href{http://arxiv.org/pdf/1204.2890v1.pdf}{arXiv:1204.2890v1}.

\bibitem{Olver2010} F.W.J. Olver, D.W. Lozier, R.F. Boisvert and C.W. Clark,
\textit{NIST Handbook of Mathematical Functions}, Cambridge University Press, 2010.

\bibitem{Ponnusamy2014AMC} S. Ponnusamy, T. Quach and A. Rasila,
\textit{Harmonic shears of slit and polygonal mappings},
Appl. Math. Comput. \textbf{233} (2014), 588--598.

\bibitem{Ponnusamy2014CVEE} S. Ponnusamy,  A. Sairam Kaliraj and A. Rasila,
\textit{Harmonic close-to-convex functions and minimal surfaces},
Complex Var. Elliptic Equ.  \textbf{59}(7) (2014), 986 --1002.

\bibitem{SaRa2013} S. Ponnusamy and A. Rasila,
\textit{Planar harmonic and quasiregular mappings},
Topics in Modern Function Theory (Editors. St. Ruscheweyh and S. Ponnusamy): Chapter in
CMFT, RMS-Lecture Notes Series No. 19, 2013, pp. 267--333.

\bibitem{Quach2014} T. Quach,
\textit{Harmonic Shears and Numerical Conformal Mappings},
\href{http://arxiv.org/pdf/1405.6759v1.pdf}{arXiv:1405.6759v1}.


\end{thebibliography}
\end{document}